\documentclass[12pt]{report}
\usepackage{graphicx}
\usepackage{latexsym}
\usepackage{amsmath}
\usepackage{amssymb}
\usepackage{amsthm}
\usepackage{CJK}
\usepackage[hyperindex,breaklinks]{hyperref}
\newtheorem{defn}{Definition}[section]
\newtheorem{thm}{Theorem}[section]
\newtheorem{lem}{Lemma}[section]
\newtheorem{cor}{Corollary}[section]
\newtheorem{fact}{Fact}[section]
\newtheorem{prop}{Proposition}[section]
\newtheorem*{conj}{Conjecture}
\newtheorem*{clm}{Claim}
\theoremstyle{remark}\newtheorem{rem}{Remark}[section]
\newtheorem*{sproof}{Sketch of the proof}
\begin{document}

\marginparsep 0pt
\textwidth 15.0 truecm


\setlength{\baselineskip}{0.780cm}
\pagestyle{empty}
\begin{center}
\Large{ \bf Survey on the Canonical Metrics on the Teichm\"{u}ller Spaces and the Moduli Spaces of Riemann Surfaces}
\end{center}

\vspace{12mm}
\begin{center}
CHAN, Kin Wai
\end{center}

\vspace{12mm}
\begin{center}
A Thesis Submitted in Partial Fulfillment \\
of the Requirements for the Degree of \\
Master of Philosophy \\
in \\
Mathematics
\end{center}

\vspace{4cm}
\begin{center}
The Chinese University of Hong Kong \\
September 2010
\end{center}

\newpage
\vspace*{8cm}
\begin{center}
\underline{Thesis Assessment Committee}\\
\vspace{1cm}
Professor LEUNG Nai Chung Conan (Chair)\\
Professor WAN Yau Heng Tom (Thesis Supervisor)\\
Professor AU Kwok Keung Thomas (Committee Member)\\
Professor CHEUNG Wing-Sum (External Examiner)
\end{center}

\newpage
\setcounter{page}{1}
\pagestyle{myheadings}
\markright{Canonical Metrics on the Teichm\"{u}ller Spaces}

\begin{center}
{\large {\bf ACKNOWLEDGMENTS}}
\end{center}

\vspace{5mm}

I shall hereby give a sincere gratitude to my supervisor, Prof. Tom Y. H. Wan, for his continuous support and encouragement to my research life during these two years.  With his aid, I gradually become more interested in and familiar with geometry.  Also, I am really thankful that he suggested me choosing this fascinating topic as my thesis topic.  Working on this topic actually let me acquire a clearer picture of complex geometry and the structure of moduli spaces.  I would like to thank Prof. Thomas K. K. Au as well, for his useful advice on my problems, both personal and academic.  On the road of applying foreign graduate schools he helped me a lot.  Moreover, I wish to express my appreciation to Prof. Conan N. C. Leung for his motivative lessons on different topics in geometry.  He is always able to bring us to the front of the research of geometry, and my knowledge and insight in geometry are thus broadened quite a lot.  Finally I offer my heartfelt thanks to all people who have helped me in any way, especially for those who gave me a hand for dealing with the difficulties I met when I was working on this thesis, like Andy Cheung, John Ma and Lijiang Wu.

\newpage
\noindent
{\Huge {\bf Abstract}}
\vspace{1.2cm}

\noindent
This thesis results from an intensive study on the canonical metrics on the Teichm\"{u}ller spaces and the moduli spaces of Riemann surfaces.  The Teichm\"{u}ller space $T_g$  of Riemann surfaces of genus $g$ can be defined to be the set of all the conformal equivalence classes of the Riemann surfaces of genus $g$ with quasiconformal mappings.  It is well known that the moduli space $\mathcal{M}_g$ of Riemann surfaces of genus g can be viewed as the quotient of $T_g$  by the mapping class group $Mod_g$.  There are several renowned classical metrics on $T_g$  and $\mathcal{M}_g$, including the Weil-Petersson metric, the Teichm\"{u}ller metric, the Kobayashi metric, the Bergman metric, the Carath\'{e}odory metric and the K\"{a}hler-Einstein metric.  The Teichm\"{u}ller metric, the Kobayashi metric and the Carath\'{e}odory metric are only (complete) Finsler metrics, but they are effective tools in the study of hyperbolic property of $\mathcal{M}_g$.  The Weil-Petersson metric is an incomplete K\"{a}hler metric, while the Bergman metric and the K\"{a}hler-Einstein metric are complete K\"{a}hler metrics.  However, McMullen\cite{bib3} introduced a new complete K\"{a}hler metric, called the McMullen metric, by perturbing the Weil-Petersson metric.  This metric is indeed equivalent to the Teichm\"{u}ller metric.  Recently, Liu-Sun-Yau\cite{bib1} proved that the equivalence of the K\"{a}hler-Einstein metric to the Teichm\"{u}ller metric, and hence gave a positive answer to a conjecture proposed by Yau\cite{bib15}.  Their approach in the proof is to introduce two new complete K\"{a}hler metrics, namely, the Ricci metric and the perturbed Ricci metric, and then establish the equivalence of the Ricci metric to the K\"{a}hler-Einstein metric and the equivalence of the Ricci metric to the McMullen metric.  The main purpose of this thesis is to survey the properties of these various metrics and the geometry of $T_g$ and $\mathcal{M}_g$ induced by these metrics.

\newpage
\begin{CJK}{Bg5}{bsmi}
\noindent
{\Huge {\bf 摘要}}
\vspace{1.2cm}

\noindent
在這篇論文中，我將會對Teichm\"{u}ller空間及黎曼曲面的模空間上的典範度量進行深入的研究。虧格為$g$的黎曼曲面的Teichm\"{u}ller空間$T_g$可以被定義為一個由配上了擬共形映射，而且虧格為$g$的黎曼曲面的共形等價類組成的集，而為人熟知的是虧格為$g$的黎曼曲面的模空間$\mathcal{M}_g$可以被視為由映射類群$Mog_g$對$T_g$所做成的商空間。在$T_g$和$\mathcal{M}_g$之上有好幾個有名的經典度量，包括Weil-Petersson度量、Teichm\"{u}ller度量、Kobayashi度量、Bergman度量、Carath\'{e}odory度量和K\"{a}hler-Einstein度量，其中Teichm\"{u}ller度量、Kobayashi度量和Carath\'{e}odory度量只是（完備的）芬斯勒度量，不過它們卻是用來研究$\mathcal{M}_g$的雙曲特性的好工具。Weil-Petersson度量是一個不完備的K\"{a}hler度量，而Bergman度量及K\"{a}hler-Einstein度量則是完備的K\"{a}hler度量；然而，McMullen\cite{bib3}以微擾Weil-Petersson度量的方法引進了一個新的完備K\"{a}hler度量，名為McMullen度量。這個度量其實和Teichm\"{u}ller度量是等價的。最近，Liu-Sun-Yau\cite{bib1}證明了K\"{a}hler-Einstein度量等價於Teichm\"{u}ller度量，而由此給了一個肯定的答案予Yau\cite{bib15}所提出的一個猜想。他們證明的進路是首先引進兩個新的完備K\"{a}hler度量：Ricci度量及微擾Ricci度量，然後分別建立Ricci度量和K\"{a}hler-Einstein度量的等價性及Ricci度量和McMullen度量的等價性。這篇論文的主要目的是去考察這些不同的度量的特性及它們給$T_g$和$\mathcal{M}_g$帶來了甚麼樣的幾何。
\end{CJK} 

\newpage
{\large {\tableofcontents}}

\chapter{Introduction}
This thesis is a survey on the canonical metrics, namely, the Teichm\"{u}ller metric $\omega_T$, the Kobayashi metric $\omega_K$, the Carath\'{e}odory metric $\omega_C$, the Bergman metric $\omega_B$, the Weil-Petersson metric $\omega_{WP}$, the McMullen metric $g_{1/l}$, the K\"{a}hler-Einstein metric $g_{KE}$, the Ricci metric $\tau$ and the perturbed Ricci metric $\tilde{\tau}$, on the Teichm\"{u}ller spaces $T_g$ and the moduli spaces of Riemann surfaces $\mathcal{M}_g$.  These metrics will be introduced one by one, and the geometric properties (say, the curvature bounds) related to some of these metrics will also be reviewed.  The equivalences between these metrics will also be discussed.  In fact, the main part of this thesis consists of the curvature formulae of the two newest complete K\"{a}hler metrics, i.e., the Ricci metric and the perturbed Ricci metric, and the estimates on their curvatures, including the holomorphic sectional curvature, bisectional curvature and Ricci curvature, by Liu-Sun-Yau\cite{bib1}\cite{bib2}.  The importance of these results will be apparent when we come to the equivalence of the Ricci metric (and the perturbed Ricci metric) to the K\"{a}hler-Einstein metric and to the McMullen metric respectively.  The essential point is that these two equivalences together solve a conjecture proposed by Yau\cite{bib15} twenty years ago.  Now let me introduce each of the chapters below, i.e., chapter 2 to chapter 5, in a more detailed manner.

Chapter 2 is a preparation chapter.  In this chapter, some of the background knowledge which are useful or helpful to the remaining chapters is reviewed.  The background knowledge includes three parts - the first part being the Riemann surface theory, the second part being the Teichm\"{u}ller theory and the third part being the Schwarz-Yau lemma - which are discussed in three different sections.  For the Riemann surface theory, the uniformization theorem (Theorem 2.1.1) provides a way of modeling the Riemann surfaces.  The Fuchsian groups and the quasiconformal mappings will help define the Teichm\"{u}ller spaces, while the Beltrami coefficients and the holomorphic quadratic differentials play a main role in the infinitesimal theory of Teichm\"{u}ller spaces.  Moreover, the nodal Riemann surfaces are candidates for the boundary points of the moduli spaces of Riemann surfaces.

In section 2.2, different equivalent definitions of the Teichm\"{u}ller spaces are given.  The Teichm\"{u}ller distance, which turns the Teichm\"{u}ller spaces into metric spaces, and the Bers embedding, which provides the Teichm\"{u}ller spaces with  the complex structures, are briefly reviewed.  The moduli spaces of Riemann surfaces are also defined.  With the aid of Bers embedding (and the Kodaira-Spencer map), the representations of the tangent space and the cotangent space at some point of a moduli space as the space of harmonic Beltrami differentials and the space of holomorphic quadratic differentials of a Riemann surface are obtained (Theorem 2.2.1 and the paragraph below the theorem).  The local behavior of the boundary of moduli spaces is briefly studied as well.  Two versions of the Schwarz-Yau lemma are included in section 2.3 for the later use in the chapter 5.

In chapter 3, the classical canonical metrics, i.e., the Teichm\"{u}ller metric, the Kobayashi metric, the Carath\'{e}odory metric, the Bergman metric and the Weil-Petersson metric, are examined.  The first three metrics are just Finsler metrics, while the latter two are K\"{a}hler.  In section 3.1 the Finsler metrics and the Bergman metric are introduced.  They are all complete, and indeed they are equivalent to each other.  The Teichm\"{u}ller metric even coincides with the Kobayashi metric, which is a result from Royden\cite{bib27}.  The equivalence of the Kobayashi metric to the Carath\'{e}odory metric and to the Bergman metric  (Theorem 3.1.1) rely on the results of Kobayashi\cite{bib28} (Lemma 3.1.1).  Section 3.2 consists of the description of the Weil-Petersson metric.  Since the harmonic lifts $v_1, ..., v_n$ are important in the computation of the curvature formulae of the Weil-Petersson metric and the Ricci metric, some results related to them, like the formulae for the Lie derivatives of the harmonic Beltrami differentials $B_i$ with respect to them (Lemma 3.2.4), the harmonicity of these Lie derivatives (Lemma 3.2.5) and the commutator of $v_k$ and $\overline{v_l}$ (Lemma 3.2.6), are included as well.  The definition of Maass operators $K_p$ and $L_p$ are also given.  Moreover, the curvature formula of the Weil-Petersson metric is quoted (Theorem 3.2.1).

Chapter 4 is the longest chapter of the thesis because it contains numerous important and useful results of the four "younger" K\"{a}hler metrics, namely, the McMullen metric, the K\"{a}hler-Einstein metric, the Ricci metric and the perturbed Ricci metric.  The first section is an overview of the McMullen metric, which was constructed by McMullen\cite{bib3}.  A sketch of the proof of the K\"{a}hler hyperbolicity of this metric is given (Theorem 4.1.1).  Also, its equivalence to the Teichm\"{u}ller metric is proved (Theorem 4.1.3).  Section 4.2 is a short review on the existence of the K\"{a}hler-Einstein metric (Theorem 4.2.1) and Yau's conjecture.  The last three sections are the main characters of this chapter.  In section 4.3 an intensive study on the Ricci metric due to Liu-Sun-Yau is included.  The commutator operators $\xi_k$ are defined so as to simplify the computation of the curvature formula of the Ricci metric.  Some formulae related to these operators are given (Lemma 4.3.1). The operator $Q_{k\overline{l}}$, which is indeed the commutator of $\xi_k$ and $\overline{v_l}$ (Lemma 4.3.2), is also defined for the purpose of simplification.  The curvature formula of the Ricci metric is then explicitly worked out (Theorem 4.3.2).

Plenty of estimates on the asymptotics and the curvatures of the Ricci metric are found in section 4.4.  Some $C^{\infty}$ norms on the sections of $(\kappa\otimes\overline{\kappa}^{-1})^{\frac{p}{2}}$ for $p\geq 0$ are defined as Wolpert\cite{bib11} did.  Then the Trapani's expression\cite{bib10} of the Masur's holomorphic quadratic differentials $\psi_1, ..., \psi_n$ \cite{bib9} on the genuine collars using rs-coordinate provides a tool for the estimates of the Weil-Petersson metric near a boundary point (Corollary 4.4.1) and the rs-coordinate expression of the harmonic Beltrami differential $B_i$ on the collars (Lemma 4.4.1).  Using this expression of $B_i$ the estimates on the norms of $A_i$ and $f_{i\overline{j}}=A_i\overline{A_j}$ are easily obtained (Lemma 4.4.2).  The approximation functions $\tilde{e_{i\overline{j}}}$ of $e_{i\overline{j}}=(\Box+1)^{-1}f_{i\overline{j}}$ are defined and the norm estimates on them are worked out (Lemma 4.4.3) so that the estimates on the asymptotics of the Ricci metric can be done (Corollary 4.4.2).  As a preparation for the curvature estimates, the estimates on different norms of $K_0f_{i\overline{j}}$ are given (Lemma 4.4.6).  A formula for the term $\displaystyle \int_X Q_{k\overline{l}}(e_{i\overline{j}})e_{\alpha\overline{\beta}}dv$ in the curvature formula of the Ricci metric is provided as well (Lemma 4.4.7).  Finally the estimates on the curvatures, especially the holomorphic sectional curvature, of the Ricci metric are proved (Theorem 4.4.3).

The last section in this chapter is about the perturbed Ricci metric.  Using the curvature formula of the Ricci metric, the curvature formula of the perturbed Ricci metric is easily obtained (Theorem 4.5.1).  By the estimates on the inverse matrix $(\tilde{\tau}^{i\overline{j}})$ near a boundary point (Lemma 4.5.2), the estimates on the curvatures of the perturbed Ricci metric can be done with nearly the same method as that of the Ricci metric (Theorem 4.5.2).  The equivalence of the Ricci metric and the perturbed Ricci metric follows from the estimates on the Weil-Petersson metric and the Ricci metric (Theorem 4.5.3).

In the last chapter, two essential equivalences of the canonical metrics are examined.  One is the equivalence of the Ricci metric to the K\"{a}hler-Einstein metric proved in the section 5.1.  This equivalence (Theorem 5.1.1) follows from a simple but important linear algebra fact (Lemma 5.1.1) and the Schwarz-Yau lemma.  Another is the equivalence of the Ricci metric to the McMullen metric in the section 5.2.  Using the first derivatives of the length functions of short geodesics $\partial_i l_j$ (Lemma 5.2.1) and the Schwarz-Yau lemma, this equivalence (Theorem 5.2.1) is proved as well.  These two equivalence give a positive answer to the Yau's conjecture.  This remark gives an end to this thesis.

\chapter{Background Knowledge}
In this chapter some background knowledge about the Riemann surface theory and the Teichm\"{u}ller theory will be presented.  In addition, we will also take a look at the Schwarz-Yau Lemma, whose importance will be revealed in chapter 5.

\section{Results from Riemann Surface Theory and Quasiconformal Mappings}
Riemann surface theory is a prerequisite for the Teichm\"{u}ller theory, so we should begin by a briefing about some essential results related to the Riemann surfaces.

\subsection{Riemann Surfaces and the Uniformization Theorem}
Let us first review the definition of a Riemann surface.
\begin{defn}
A Riemann surface $R$ is a connected one dimensional complex manifold, i.e., $R$ is a connected Hausdorff topological space with a family $\{(U_{\alpha},z_{\alpha})\}_{\alpha\in A}$ ($A$ is an index set) satisfying the following three conditions:
\begin{itemize}
\item[(i)] $\{U_{\alpha}\}_{\alpha\in A}$ is an open covering of $R$.
\item[(ii)] Every $z_{\alpha}$ is a homeomorphism from $U_{\alpha}$ onto an open subset $D_{\alpha}$ of $\mathbb{C}$.
\item[(iii)] If $U_{\alpha}\cap U_{\beta}\neq \phi$, the transition mappings $z_{\alpha\beta}:=z_{\alpha}\circ z_{\beta}^{-1}$ from $z_{\beta}(U_{\alpha}\cap U_{\beta})$ to $z_{\alpha}(U_{\alpha}\cap U_{\beta})$ are biholomorphic.
\end{itemize}
\end{defn}

\bigskip
The pairs $(U_{\alpha},z_{\alpha})$ are called charts.  Two systems of charts $(U_{\alpha},z_{\alpha})$ and $(V_{\beta},w_{\beta})$ are called compatible if $w_{\beta}\circ z_{\alpha}^{-1}$ is a holomorphic mapping from $z_{\alpha}(U_{\alpha}\cap V_{\beta})$ to $w_{\beta}(U_{\alpha}\cap V_{\beta})$ for any nonempty $U_{\alpha}\cap V_{\beta}$.  Since the composition of holomorphic mappings is holomorphic and the inverse of holomorphic bijection is also holomorphic, we can see that compatibility is an equivalence relation.  An equivalent class of compatible systems of charts on $R$ is called a complex structure or a Riemann surface structure of $R$.

A system of charts on $R$ determines an orientation of $R$ by pulling back the usual orientation of $D_{\alpha}$ to $U_{\alpha}$ through the mapping $z_{\alpha}$.  On $U_{\alpha}\cap U_{\beta}$ the orientation is consistently determined as the Jacobian of $z_{\alpha\beta}$ is positive.

Traditionally a compact Riemann surface is called closed while a noncompact Riemann surface is called open.

A continuous mapping $f$ between two Riemann surfaces $R_1$ and $R_2$ is called holomorphic if for any charts $z_{\alpha}$ and $w_{\beta}$ on $R_1$ and $R_2$ respectively, the mapping $w_{\beta}\circ f\circ z_{\alpha}^{-1}$ is holomorphic.  $R_1$ and $R_2$ are analytically equivalent if there exists a biholomorphism between $R_1$ and $R_2$.

\bigskip
The uniformization theorem can be stated in many different ways.  The deepest part of it is called the Koebe's planarity theorem, which says that a noncompact planar Riemann surface is analytically equivalent to a domain in the plane (a surface $R$ is planar if any simple closed curve on $R$ divides $R$ into two connected components).  This result can be combined with the topological theory of covering surfaces.  In particular, $R$ is analytically equivalent to its universal covering $\tilde{R}$ factored by a deck transformation group $\Gamma$ which is isomorphic to the fundamental group of $R$.  We are then led to the following theorem, which is also called the uniformization theorem.
\begin{thm}
Let $R$ be a Riemann surface.  Then $R$ is analytically equivalent to one of the following:
\begin{itemize}
\item[(i)] The Riemann sphere $\hat{\mathbb{C}}=\mathbb{C}\cup\{\infty\}=\mathbb{C}P^1$.
\item[(ii)] The complex plane $\mathbb{C}$.
\item[(iii)] The punctured plane $\mathbb{C}-\{0\}$
\item[(iv)] The plane modulo a lattice $\mathbb{C}/L$, where $L$ is isomorphic to $\mathbb{Z}\times\mathbb{Z}$ and is spanned by two $\mathbb{R}$-linearly independent vectors in $\mathbb{C}$.
\item[(v)] The upper half plane $\mathbb{H}=\{z|\textrm{Im}z>0\}$ modulo a properly discontinuous, torsion free group $\Gamma$ of biholomorphisms of $\mathbb{H}$ (here we allow $\Gamma$ to be a trivial group, in which case the quotient space is $\mathbb{H}$ itself).
\end{itemize}
\end{thm}

\bigskip
Proofs of this theorem can be found in Ahlfors\cite{bib4}, Ahlfors and Sario\cite{bib5}, Farkas and Kra\cite{bib6}, and Springer\cite{bib7}.  Notice that any two of the five types of surfaces above are not analytically equivalent.

\subsection{Fuchsian Groups}
Due to the uniformization theorem, groups of holomorphic homeomorphisms of the upper half plane $\mathbb{H}$ which are torsion free and act properly discontinuously on $\mathbb{H}$ play a central role in Riemann surface theory.  Let's drop the torsion free condition and consider the notion of a Fuchsian group:
\begin{defn}
A group of holomorphic homeomorphisms of the upper half plane $\mathbb{H}$ which acts properly discontinuously on $\mathbb{H}$ is called a Fuchsian group.
\end{defn}

\bigskip
The structure of the Fuchsian groups is indeed known.  More precisely, we have the following facts:
\begin{fact}
Suppose $A:\mathbb{H}\rightarrow\mathbb{H}$ is s bijective holomorphic mapping.  Then there exist real numbers $a, b, c, d$ such that $ad-bc=1$ and
\[A(z)=\dfrac{az+b}{cz+d}.\]
The numbers $a, b, c, d$ are uniquely determined up to a $-1$ multiple to all of them.  Hence a Fuchsian group is a subgroup of $\textrm{PSL}(2,\mathbb{R})$.
\end{fact}
\begin{fact}
A subgroup of $\textrm{PSL}(2,\mathbb{R})$ is discrete if and only if it acts properly discontinuously on $\mathbb{H}$.  Thus any discrete subgroup of $\textrm{PSL}(2,\mathbb{R})$ is a Fuchsian group.
\end{fact}

\bigskip
For the proofs of these facts, one can consult Gardiner\cite{bib25}.

Note that for $B\in\textrm{PSL}(2,\mathbb{R})$ and $\Gamma$ being a subgroup of $\textrm{PSL}(2,\mathbb{R})$, we also call the group $B\circ\Gamma\circ B^{-1}$ a Fuchsian group.  This group would instead have $B(\mathbb{H})$ as the invariant domain.

\begin{defn}
An elliptic point $p$ for $G$ acting on $X$ is a point on $X$ whose isotropy group is nontrivial.
\end{defn}

\bigskip
We then have a theorem about the elliptic points for a Fuchsian group and the quotient space $\mathbb{H}/\Gamma$.
\begin{thm}
Let $\Gamma$ be a Fuchsian group acting on $\mathbb{H}$.  Then $\mathbb{H}/\Gamma$ is a complete metric space with the quotient metric given by
\[\overline{d}(\overline{p},\overline{r})=\inf_{A\in G}d(A(p),r)\]
where $d$ is the Poincar\'{e} metric on $\mathbb{H}$.  The set of elliptic points for $\Gamma$ is a discrete set.  Moreover, there is a unique complex structure on $\mathbb{H}/\Gamma$ making $\mathbb{H}/\Gamma$ into a Riemann surface and the mapping $\pi:\mathbb{H}\rightarrow \mathbb{H}/\Gamma$ into a holomorphic mapping.  At nonelliptic points $\pi$ is a local homeomorphism, and at elliptic points $\pi$ is locally an $n$-to-$1$ mapping, where $n$ is the order of the isotropy group of the elliptic point.
\end{thm}

\bigskip
Its proof also appears in Gardiner\cite{bib25}.

\subsection{Quasiconformal Mappings and the Beltrami Equation}
A Jordan region on a Riemann surface $R$ is a connected and simply connected open subset of $R$ whose boundary is a simple closed curve in $R$. This motivates us to consider the notion of a generalized quadrilateral.
\begin{defn}
A generalized quadrilateral $Q$ on $R$ is a Jordan region on $R$ together with two disjoint closed arcs $\beta_1$ and $\beta_2$ on the boundary of $Q$.  The module of $Q$, denoted as $m(Q)$, is determined by the conformal mapping of $Q$ onto a rectangle which takes $\beta_1$ and $\beta_2$ onto the vertical sides of the rectangle in this way: if the rectangle has width $a$ and height $b$, then $m(Q)=\dfrac{a}{b}$.
\end{defn}

\bigskip
Notice that $m(Q)$ possesses the following property: if $Q^*$ is the same Jordan region as $Q$ but with two disjoint closed arcs $\alpha_1$ and $\alpha_2$ complementary in the boundary of $Q$ (except for common endpoints), then $m(Q^*)=\dfrac{1}{m(Q)}$.

Now we are ready to introduce the geometric definition of a quasiconformal mapping.
\begin{defn}
Let $f$ be an orientation preserving homeomorphism from a region $\Omega$ to a region $\Omega'$.  Then $f$ is $K$-quasiconformal if for every quadrilateral $Q$ in $\Omega$, $m(f(Q))\leq Km(Q)$.  The smallest possible value of $K$ for which the inequality holds for all quadrilateral $Q$ is called the dilatation of $f$.
\end{defn}

\bigskip
It is obvious that $m(f(Q^*))\leq Km(Q^*)$ and $f(Q)^*=f(Q^*)$, so by definition we know that
\[K^{-1}m(Q)\leq m(f(Q))\leq Km(Q)\]
for all quadrilateral $Q$.  Using this definition we can easily obtain the property that if $f_1$ and $f_2$ are $K_1$- and $K_2$-quasiconformal respectively, then $f_1\circ f_2$ is $K_1K_2$-quasiconformal.

\bigskip
The analytic definition of a quasiconformal mapping require the notion of absolute continuity on lines (ACL).  A function $f(z)=u(x,y)+iv(x,y)$ is ACL if for every rectangle in $\Omega$ with sides parallel to the $x$- and $y$-axes, both $u(x,y)$ and $v(x,y)$ are absolutely continuous on almost every horizontal and almost every vertical line of that rectangle.  The function $u$ and $v$ will thus have partial derivatives $u_x$, $u_y$, $v_x$, $v_y$ almost everywhere in $\Omega$.  Note that the complex partial derivatives are defined by $f_z=\dfrac{1}{2}(f_x-if_y)$ and $f_{\overline{z}}=\dfrac{1}{2}(f_x+if_y)$.
\begin{defn}
Let $f$ be a homeomorphism from a domain $\Omega$ to a domain $\Omega'$.  Then $f$ is $K$-quasiconformal if $f$ is ACL in $\Omega$ and $|f_{\overline{z}}|\leq k|f_z|$ almost everywhere for $k=\dfrac{K-1}{K+1}<1$.  The minimal possible value of $K$ for which the latter condition is satisfied is called the dilatation of $f$.
\end{defn}

\bigskip
The next theorem provides a bridge connecting two definitions.
\begin{thm}
The geometric and the analytic definitions of $K$-quasiconformality are equivalent.  Moreover, for a quasiconformal mapping $f$, $f_{\overline{z}}$ and $f_z$ are locally square integrable.
\end{thm}

\bigskip
The proof of this theorem can be found in Ahlfors\cite{bib22} or Lehto and Virtanen\cite{bib23}.

\bigskip
Notice that for a quasiconformal mapping $f$, $K(f)$ is denoted to be the dilatation of $f$.

\bigskip
We have seen that if $f$ is a topological ACL mapping and if $\Big|\dfrac{f_{\overline{z}}}{f_z}\Big|\leq k<1$ almost everywhere, then $f$ is quasiconformal.  Let $\mu(z)$ be a measurable complex valued function defined in a domain $\Omega$ for which $\left\|\mu\right\|_{\infty}=k<1$.  The Beltrami equation is
\[f_{\overline{z}}(z)=\mu(z)f_z(z)\]
where the partial derivatives are assumed to be locally square integrable and taken in the sense of distributions.  The function $\mu$ is called the Beltrami coefficient of the mapping $f$.

In the theory of Beltrami equation (Ahlfors\cite{bib22}, Lehto and Virtanen\cite{bib23}, Ahlfors and Bers\cite{bib24}), a solution $f$ can be expressed as a power series in $\mu$, where the power series is made up by taking compositions of singular integral operators.  The following theorem asserts the existence of normalized global solutions to the Beltrami equation on $\hat{\mathbb{C}}$.  It also expresses in a very particular way the analytic dependence of the solution $f$ on the Beltrami coefficient $\mu$.  The analyticity of this dependence is important in the construction of complex structure for the Teichm\"{u}ller spaces.
\begin{thm} (Theorem 5, \cite{bib25})
The Beltrami equation gives a one-to-one correspondence between the set of quasiconformal homeomorphisms of $\hat{\mathbb{C}}$ which fix the points $0$, $1$ and $\infty$ and the set of measurable complex valued functions $\mu$ on $\hat{\mathbb{C}}$ for which $\left\|\mu\right\|_{\infty}<1$.

Furthermore, the normalized solution $f^{\mu}$ to the equation depends holomorphically on $\mu$ and for $r>0$ there are $\delta>0$ and $C(r)>0$ such that
\[|f^{\mu}-z-tF(z)|\leq C(r)t^2\]
for $|z|<r$ and $|t|<\delta$, where
\[F(z)=-\dfrac{z(z-1)}{\pi}\iint_{\mathbb{C}}\dfrac{\mu(\xi)d\xi d\eta}{\zeta(\zeta-1)(\zeta-z)}\]
and $\zeta=\xi+i\eta$.
\end{thm}

\bigskip
Usually we need solutions $f$ which map the upper half plane $\mathbb{H}$ to itself and preserve $\hat{\mathbb{R}}$ for arbitrary Beltrami coefficient $\mu$ with support in $\mathbb{H}$.  Let $M(\mathbb{H})$ be the space of complex valued $L_{\infty}$-functions $\mu$ with support in $\mathbb{H}$ and $\left\|\mu\right\|_{\infty}<1$.  For $\mu$ in $M(\mathbb{H})$, let $\hat{\mu}$ be identical to $\mu$ on $\mathbb{H}$ and equal to $\overline{\mu(\overline{z})}$ on the lower half plane $\mathbb{H}^*$.  Then we solve the Beltrami equation with $\mu$ replaced by $\hat{\mu}$, and let $f$ be the desired normalized solution.  Since $\overline{f(\overline{z})}$ has the same Beltrami coefficient as $f$ and also fixes the same three points on $\hat{\mathbb{R}}$ as $f$, by uniqueness of the solutions to the Beltrami equation we know that $\overline{f(\overline{z})}=f(z)$, and $f$ is a quasiconformal mapping which preserves $\hat{\mathbb{R}}$ in an orientation-preserving manner.  In addition, $f$ preserves $\mathbb{H}$ and $\mathbb{H}^*$.

We hence obtain a corollary to the previous theorem.
\begin{cor}
For every $\mu\in M(\mathbb{H})$, there exists a unique quasiconformal self-mapping $f$ of $\mathbb{H}$ satisfying the Beltrami equation on $\mathbb{H}$ which extends continuously to the closure of $\mathbb{H}$ and is normalized to fix $0$, $1$ and $\infty$.
\end{cor}

\subsection{Holomorphic Quadratic Differentials}
We begin by stating the definition of a holomorphic quadratic differential.
\begin{defn}
A holomorphic quadratic differential $\varphi$ on a Riemann surface $R$ is an assignment of a holomorphic function $\varphi_1(z_1)$ to each local coordinate $z_1$ satisfying the compatibility condition, i.e., if $z_2$ is another local coordinate overlapping $z_1$, then $\varphi_1(z_1)=\varphi_2(z_2)\Big(\dfrac{dz_2}{dz_1}\Big)^2$.
\end{defn}

\bigskip
We usually impose some restrictions on the holomorphic quadratic differentials.  Along a border arc $\alpha$ in the border of $R$ we require that $\varphi_1(\alpha)$ be real if $z_1$ is a local coordinate with $z_1(\alpha)$ real.  If there exists an isolated boundary point $p$ on $R$, then we require that $\varphi$ have at most a simple pole at $p$.  We do not permit poles of $\varphi$ to occur on the boundary curves of $R$.  The vector space of holomorphic quadratic differentials with the above conditions is denoted as $Q(R)$.

Now if the genus of $R$ is $g$ and $R$ is obtained form a compact surface by deleting $m$ disjoint closed disks and $n$ isolated points outsides the disks, then we have the following theorem:
\begin{thm}
\[dim_{\mathbb{R}}Q(R)=6g-6+3m+2n\]
except for cases shown below:
\begin{center}
\begin{tabular}{|c|c|c|c|}
\hline
$g$ & $m$ & $n$ & $dim_{\mathbb{R}}Q(R)$\\
\hline
$0$ & $0$ & $0,1,2$ & $0$\\
$0$ & $1$ & $0,1$ & $0$\\
$1$ & $0$ & $0$ & $2$\\
\hline
\end{tabular}
\end{center}
In the case when $m=0$, $Q(R)$ is a complex vector space and the same formula yields
\[dim_{\mathbb{C}}Q(R)=3g-3+n\]
for any $g\geq 2$, for $g=1$ and $n\geq 1$, and for $g=0$ and $n\geq 3$.
\end{thm}

\bigskip
This theorem can be easily deduced from the Riemann-Roch theorem.  It can also be seen as a consequence of the Teichm\"{u}ller theorem.  The proof can be found in Farkas and Kra\cite{bib6}.

\subsection{Nodal Riemann Surfaces}
For the study of the boundary of the Teichm\"{u}ller spaces or the moduli spaces, we require the notion of a nodal Riemann surface.
\begin{defn}
A nodal Riemann surface $X_0$ is a (singular) Riemann surface such that for every point $p\in X_0$, there is a neighborhood of $p$ which is isometric to $\{z||z|<1\}\subset\mathbb{C}$ or $\{(z,w)|zw=0, |z|<1, |w|<1\}\subset\mathbb{C}^2$.  In the second case $p$ is called a node.  A connected component of the complement of the nodes is called a part of $X_0$.
\end{defn}

\bigskip
Suppose $X_0$ has nodes $p_1, ..., p_s$ and $r\geq 1$ parts $X_1, ..., X_r$, where each $X_i$ is a Riemann surface of genus $g_i$ with $n_i$ punctures.  Notice there are exactly two punctures $a_j$ and $b_j$ (which may lie on different $X_i$) corresponding to the same node $p_j$, and we call them paired punctures.  Then we have $\displaystyle 2s=\sum_{i=1}^rn_i$ and $\displaystyle g=\sum_{i=1}^rg_i+p+1-r$.  In this thesis we will assume for all $i$, $X_i$ is not a sphere with one or two punctures.  Thus each $X_i$ is a hyperbolic Riemann surface carrying a Poincar\'{e} metric $\rho$.

\bigskip
Now let us talk about the rs-coordinate on a Riemann surface.  This coordinate was defined by Wolpert\cite{bib11}, and there are two cases: the puncture case and the short geodesic case.

For the puncture case, let $X$ be a nodal Riemann surface and $p\in X$ be a node.  Also, let $a$ and $b$ be the paired punctures corresponding to $p$.
\begin{defn}
A local coordinate chart $(U, u)$ near $a$ is called rs-coordinate if $u(a)=0$ where $u$ maps $U$ to the punctured disk $\{u|0<|u|<c\}$ with $c>0$, and the restriction to $U$ of the K\"{a}hler-Einstein metric on $X$ can be written as
\[\dfrac{1}{2|u|^2(\log|u|)^2}|du|^2.\]
The rs-coordinate $(V, v)$ near $b$ can be defined similarly.
\end{defn}

\bigskip
For the short geodesic case, let $X$ be a closed Riemann surface, and let $\gamma$ be a closed geodesic on $X$ with length $l<c_*$, where $c_*$ is the collar constant introduced by Keen\cite{bib26}.
\begin{defn}
A local coordinate chart $(U, z)$ is called rs-coordinate at $\gamma$ if $\gamma\subset U$ where $z$ maps $U$ to the annulus $\{z|c^{-1}|t|^{\frac{1}{2}}<|z|<c|t|^{\frac{1}{2}}\}$ with $c, |t|>0$, and the restriction to $U$ of the K\"{a}hler-Einstein metric on $X$ can be written as
\[\dfrac{1}{2}\bigg(\dfrac{\pi}{\log|t|}\dfrac{1}{|z|}\csc\dfrac{\pi\log|z|}{\log|t|}\bigg)^2|dz|^2.\]
\end{defn}

\bigskip
By Keen's collar theorem \cite{bib26}, we obtain the following lemma:
\begin{lem}
Let $X$ be a closed Riemann surface and let $\gamma$ be a closed geodesic on $X$ with length $l<c_*$.  Then there is a collar $\Omega$ on $X$ with holomorphic coordinate $z$ defined on $\Omega$ such that
\begin{itemize}
\item[(i)] $z$ maps $\Omega$ to the annulus $\{z|c^{-1}e^{-\dfrac{2\pi^2}{l}}<|z|<c\}$ for $c>0$;
\item[(ii)] the K\"{a}hler-Einstein metric on $X$ restricted to $\Omega$ is given by
\[\dfrac{1}{2}u^2r^{-2}\csc^2\tau|dz|^2\]
where $u=\dfrac{l}{2\pi}$, $r=|z|$ and $\tau=u\log r$;
\item[(iii)] the geodesic $\gamma$ is given by the equation $|z|=e^{-\dfrac{\pi^2}{l}}$.
\end{itemize}
We call such a collar $\Omega$ a genuine collar.
\end{lem}

\bigskip
Notice that the constant $c$ in the lemma has a lower bound such that the area of $\Omega$ is bounded from below.  Moreover, the coordinate $z$ in the lemma is indeed rs-coordinate.

\newpage

\section{Teichm\"{u}ller Theory}
In this section we will go through some classical results of the Teichm\"{u}ller theory, like the Teichm\"{u}ller's distance and the Bers embedding.  The infinitesimal theory of Teichm\"{u}ller spaces will also be discussed.

\subsection{Teichm\"{u}ller Spaces}
There are basically two kinds of Teichm\"{u}ller spaces, namely, the Teichm\"{u}ller space of a Riemann surface $T(R_0)$ and the Teichm\"{u}ller space of a Fuchsian group $T(\Gamma_0)$.  Now let us see the definition of the Teichm\"{u}ller space of a Riemann surface first.

Let $R_0$ be a fixed Riemann surface.  In fact we can define $T(R_0)$ in two different ways.  One of them is to define $T(R_0)$ as a deformation space.
\begin{defn}
Let $\textrm{Def}(R_0)$ be the set of all pairs $(f, R)$, where $f$ is a quasiconformal homeomorphism from $R_0$ onto a surface $R$.  Two pairs $(f_1, R_1)$ and $(f_2, R_2)$ in $\textrm{Def}(R_0)$ are equivalent if there is a conformal map $c:R_1 \rightarrow R_2$ such that $c\circ f_1$ is homotopic to $f_2$.  Then $T(R_0)$ is defined as the quotient of $\textrm{Def}(R_0)$ by this equivalence relation.
\end{defn}

\bigskip
Another method is to define $T(R_0)$ as an orbit space.  Let $M(R_0)$ be the open unit ball in the Banach space $L_{\infty}(R_0)$, where $L_{\infty}(R_0)$ is the space of all Beltrami differentials $\mu$ on $R_0$.  Note that such a differential $\mu$ is an assignment of a measurable complex valued function $\mu^z$ to each local coordinate $z$ on $R_0$ such that $\left\|\mu\right\|_{\infty}=\sup\{\left\|\mu^z(z)\right\|_{\infty}| z \textrm{ is a local coordinate}\}<\infty$, and if $\zeta$ is another local coordinate overlapping with $z$, then $\mu^z(z)\dfrac{d\overline{z}}{dz}=\mu^{\zeta}(\zeta)\dfrac{d\overline{\zeta}}{d\zeta}$.

Form the theory of Beltrami equation we know that for any $\mu\in M(R_0)$, there exists a quasiconformal homeomorphism $w$ from $R_0$ onto another Riemann surface $R^{\mu}$ satisfying $w_{\overline{z}}=\mu w_z$.  Let $D_0(R_0)$ be the group of quasiconformal homeomorphic self-mappings of $R_0$ which are homotopic to the identity.  Now we can define $T(R_0)$:
\begin{defn}
Consider the group action $D_0(R_0)\times M(R_0)\rightarrow M(R_0)$ given by $(h, \mu)\mapsto h^*(\mu)$, where
\[h^*(\mu)=\dfrac{(w\circ h)_{\overline{z}}}{(w\circ h)_z}=\dfrac{\nu(z)+\mu(h(z))\theta(z)}{1+\overline{\nu(z)}\mu(h(z))\theta(z)}\]
for which $\nu$ is the Beltrami coefficient of $h$ and $\theta(z)=\dfrac{\overline{h_z}}{h_z}$.  Then $T(R_0)$ is defined as $M(R_0)/D_0(R_0)$.
\end{defn}

\bigskip
Gardiner\cite{bib25} showed that in fact the above two methods define the same $T(R_0)$.

\bigskip
For the definition of the Teichm\"{u}ller space of a Fuchsian group, let $\Gamma_0$ be a Fuchsian group, and define $B(\Gamma_0)$ to be the set of measurable complex valued Beltrami coefficients $\mu$ with support in $\mathbb{H}$ and which satisfy $\mu(B(z))\overline{B'(z)}=\mu(z)B'(z)$ for all $B\in\Gamma_0$.  Also, define $M(\Gamma_0)\subset B(\Gamma_0)$ such that $\left\|\mu\right\|_{\infty}<1$ for all $\mu\in M(\Gamma_0)$.  Let $w_{\mu}$ be the unique $qc$-homeomorphism of $\mathbb{H}$ whose extension to $\hat{\mathbb{R}}$ fixes $0$, $1$ and $\infty$ and satisfies $w_{\overline{z}}=\mu w_z$.
\begin{defn}
$T(\Gamma_0)$ is defined as $M(\Gamma_0)/\sim$, where $\sim$ is an equivalence relation such that $\mu\sim\nu$ if $w_{\mu}(x)=w_{\nu}(x)$ for all $x\in\hat{\mathbb{R}}$.
\end{defn}

\bigskip
Again, Gardiner\cite{bib25} showed that if $R_0=\mathbb{H}/\Gamma_0$, then $T(R_0)$ is canonically isomorphic to $T(\Gamma_0)$.  Notice that by Imayoshi and Taniguchi\cite{bib32} we can see that $T(R_0)$ indeed depends on the genus $g$ of the (closed) Riemann surface $R_0$ instead of the surface itself.  In this thesis, the identification $T_g=T(R_0)=T(\Gamma_0)$ is taken.

\subsection{Teichm\"{u}ller's Distance}
Let $\Gamma$ be a Fuchsian group.  We can then define a Teichm\"{u}ller's distance on $T(\Gamma)$:
\begin{defn}
For $[\mu], [\nu]\in T(\Gamma)$, the Teichm\"{u}ller's distance $d$ on $T(\Gamma)$ is defined as
\[d([\mu],[\nu])=\dfrac{1}{2}\inf\log K(w_{\tilde{\mu}}\circ w_{\tilde{\mu}}^{-1})\]
where the infinmum is taken over all $\tilde{\mu}\sim\mu$ and $\tilde{\nu}\sim\nu$.
\end{defn}

\bigskip
The Teichm\"{u}ller's distance makes $T(\Gamma)$ a metric space and thus gives it a topology.

\subsection{The Bers Embedding}
Let us first define the Schwarzian derivative, which will be used to define the Bers embedding.
\begin{defn}
For a (quasi)conformal mapping $f$ on a domain in $\mathbb{C}$, the Schwarzian derivative $\{f,z\}$ of $f$ is given by
\[\{f,z\}=\dfrac{f'''(z)}{f'(z)}-\dfrac{3}{2}\bigg(\dfrac{f''(z)}{f'(z)}\bigg)^2.\]
\end{defn}

\bigskip
We can easily verify the Cayley identity: $\{f\circ g,z\}=\{f,g\}g'(z)^2+\{g,z\}$ and see that $\{f,z\}=0$ if and only if $f$ is a M\"{o}bius transformation.

\bigskip
Let $\Gamma$ be a Fuchsian group such that $\mathbb{H}/\Gamma$ is a closed Riemann surface with genus $g\geq 2$ and the extensions of all its elements fix $0$, $1$ and $\infty$.  Now for $\mu\in M(\Gamma)$, let $\varphi_{\mu}(z)=\{w^{\mu},z\}$ for $z\in\mathbb{H}^*$, where $w^{\mu}$ is the unique $qc$-homeomorphism of $\hat{\mathbb{C}}$ normalized to fix $0$, $1$ and $\infty$ with Beltrami coefficient $\mu$ on $\mathbb{H}$ and $0$ on $\mathbb{H}^*$.  Then we have the following lemma (Lemma 6.4, \cite{bib32}):
\begin{lem}
Let $\gamma\in\Gamma$, then
\[\varphi_{\mu}(\gamma(z))\gamma'(z)^2=\varphi_{\mu}(z)\]
for $z\in\mathbb{H}^*$.  Furthermore, $\varphi_{\mu}$ can be regarded as a holomorphic quadratic differential on the Riemann surface $R^*=\mathbb{H}^*/\Gamma$, and for $\mu, \nu\in M(\Gamma)$, $\mu\sim\nu$ if and only if $\varphi_{\mu}=\varphi_{\nu}$ on $\mathbb{H}^*$.
\end{lem}

\bigskip
Now we can define the Bers embedding.
\begin{defn}
The Bers embedding is the mapping $\beta: T(\Gamma)\rightarrow Q(R^*)$ given by $\beta([\mu])=\varphi_{\mu}$.  Moreover, the Bers projection $\Phi:M(\Gamma)\rightarrow Q(R^*)$ is defined by $\Phi(\mu)=\varphi_{\mu}$.
\end{defn}

\bigskip
Notice that $Q(R^*)$ is equipped with the hyperbolic $L^{\infty}$-norm:
\[\left\|\varphi\right\|_{\infty, \rho}=\sup_{\mathbb{H}^*}\rho^{-2}|\varphi(z)|\]
for $\varphi\in Q(R^*)$, where $ds^2=\rho^2(z)|dz|^2$ is the Poincar\'{e} metric on $\mathbb{H}^*$.  This makes $Q(R^*)$ a complex Banach space.  Since the Bers embedding is a continuous injection, $T(\Gamma)$ inherits the complex structure of $Q(R^*)$.

\subsection{Teichm\"{u}ller Modular Groups and Moduli Spaces of Riemann Surfaces}
Let $R$ be a closed Riemann surface with genus $g\geq 2$, and let $D_+(R)$ be the group of quasiconformal homeomorphic self-mappings of $R$.  Then we can define the Teichm\"{u}ller modular group of $R$:
\begin{defn}
The Teichm\"{u}ller modular group of $R$, denoted as $Mod(R)$, is given by
\[Mod(R)=D_+(R)/D_0(R).\]
\end{defn}

\bigskip
Now we consider the moduli space of $R$:
\begin{defn}
The moduli space of $R$ is the quotient space of $T(R)$ by $Mod(R)$.  It is denoted as $\mathcal{M}(R)$.
\end{defn}

\bigskip
Note that by Imayoshi and Taniguchi\cite{bib32} again we know that the above definitions depend on $g$ instead of $R$ itself, so in this thesis the notations $Mod(R)$ and $\mathcal{M}(R)$ are replaced by $Mod_g$ and $\mathcal{M}_g$ respectively.

\subsection{Infinitesimal Theory of Teichm\"{u}ller Spaces}
For $\mu\in M(\Gamma)$ and $\nu\in B(\Gamma)$, we consider the derivative $\dot{\Phi}_{\mu}[\nu]$ of $\Phi$ in the direction $\nu$ at $\mu$ given by
\[\dot{\Phi}_{\mu}[\nu]=\lim_{t\rightarrow\infty}\dfrac{1}{t}(\Phi(\mu_t)-\Phi(\mu)),\]
where the limit is taken with respect to $\left\|\cdot\right\|_{\infty, \rho}$ and $\mu_t\in M(\Gamma)$ is defined as $\mu_t=\mu+t\nu+t\epsilon(t)$ with $\left\|\epsilon(t)\right\|_{\infty, \rho}\rightarrow$ as $t\rightarrow 0$.  Let $HB(\Gamma^{\mu})$ be the space of harmonic Beltrami differentials for $\Gamma^{\mu}=w_{\mu}\circ\Gamma\circ w_{\mu}^{-1}$ on $\mathbb{H}$.  Then the following theorem (Proposition 7.8, \cite{bib32}) gives us a representation of $T_p T(\Gamma)$ for $p=[\mu]\in T(\Gamma)$:
\begin{thm}
There exists an isomorphism between $M(\Gamma)/\textrm{Ker}\dot{\Phi}_{\mu}\cong T_p T(\Gamma)$ and $HB(\Gamma^{\mu})\cong T_0 T(\Gamma^{\mu})$.
\end{thm}

\bigskip
Now for $[X]\in\mathcal{M}_g$, the identification $T_X\mathcal{M}_g\cong HB(X)=H^1(X, TX)$ can also be obtained from Kodaira-Spencer theory and Hodge theory.  Moreover, we also have $T_X^*\mathcal{M}_g\cong Q(X)$.  Pick $\mu\in HB(X)$ and $\varphi\in Q(X)$, we can write $\mu$ and $\varphi$ as $\mu(z)\dfrac{\partial}{\partial z}\otimes d\overline{z}$ and $\varphi(z)dz^2$ in a local coordinate $z$ on $X$.  The duality between $T_X\mathcal{M}_g$ and $T_X^*\mathcal{M}_g$ is given by
\[[\mu;\varphi]=\int_X\mu(z)\varphi(z)dzd\overline{z}.\]

\subsection{Boundary of Moduli Spaces of Riemann Surfaces}
The boundary of $\mathcal{M}_g$ is the divisor $\overline{\mathcal{M}_g}-\mathcal{M}_g$ in $\overline{\mathcal{M}_g}$, where $\overline{\mathcal{M}_g}$ is the Deligne-Mumford compactification of $\mathcal{M}_g$ in \cite{bib33}.

For the local theory of the boundary, we recall the construction of Wolpert\cite{bib11}.  Let $X_{0,0}$ be a nodal Riemann surface corresponding to a codimension $m$ boundary point, i.e., $X_{0,0}$ has $m$ nodes $p_1, ..., p_m$.  Let $X_0=X_{0,0}-\{p_1, ..., p_m\}=\cup_{k=1}^r X_k$.  Fix rs-coordinate charts $(U_i, \eta_i)$ and $(V_i, \zeta_i)$ at $p_i$ for $i=1, ..., m$ such that all $U_i$ and $V_i$ are mutually disjoint.  Also, pick an open set $U_0\subset X_0$ such that $X_k\cap U_0\neq\phi$ is relatively compact and $U_0\cap(U_i\cup V_i)$ is empty for all $i$.  Now let $\nu_{m+1}, ..., \nu_n$ be Beltrami differentials supported in $U_0$ and span the tangent space at $X_0$ of the deformation space of $X_0$.  For $s=(s_{m+1}, ..., s_n)$, let $\displaystyle \nu(s)=\sum_{i=m+1}^ns_i\nu_i$.  For small enough $|s|$, we know that $|\nu(s)|<1$.  Let $X_{0,s}$ be the nodal Riemann surface obtained by solving the Beltrami equation $\overline{\partial}w=\nu(s)\partial w$.  Notice that there are constants $\delta, c>0$ such that when $|s|<\delta$, $(U_i, \eta_i)$ and $(V_i, \zeta_i)$ are holomorphic coordinates on $X_{0,s}$ with $0<|\eta_i|<c$ and $0<|\zeta_i|<c$.  Now we assume $(t_1, ..., t_m)$ has small norm.  Then we do the plumbing construction on $X_{0,s}$ to obtain $X_{t,s}$.  We remove the disks $\{\eta_i|0<|\eta_i|<\dfrac{|t_i|}{c}\}$ and $\{\zeta_i|0<|\zeta_i|<\dfrac{|t_i|}{c}\}$ for each $i=1, ..., m$ from $X_{0,s}$ and identify $\{\eta_i|\dfrac{|t_i|}{c}<|\eta_i|<c\}$ with $\{\zeta_i|\dfrac{|t_i|}{c}<|\zeta_i|<c\}$ by $\eta_i\zeta_i=t_i$.  The surface obtained is $X_{t,s}$.  The tuples $(t_1, ..., t_m, s_{m+1}, ..., s_n)$ are called the (local) pinching coordinates.  Also, the coordinates $\eta_i$ (or $\zeta_i$) are called the plumbing coordinates on $X_{t,s}$ and the collar $\{\eta_i|\dfrac{|t_i|}{c}<|\eta_i|<c\}$ (or $\{\zeta_i|\dfrac{|t_i|}{c}<|\zeta_i|<c\}$) is called the plumbing collar.

\newpage

\section{Schwarz-Yau Lemma}
This lemma was proved by Yau\cite{bib16} as a generalization of the Schwarz lemma.  In chapter 5 we will use this lemma as a main tool for proving the equivalences between the canonical metrics.
\begin{thm}
Let $M$ be a complete K\"{a}hler manifold with Ricci curvature bounded from below by a constant $K_1$, and let $N$ be another K\"{a}hler manifold with holomorphic sectional curvature bounded from above by a negative constant $K_2$. Then if there is a non-constant holomorphic mapping $f$ from $M$ into $N$, we have $K_1\leq 0$ and
\[f^*dS_N^2\leq\dfrac{K_1}{K_2}dS_M^2.\]
\end{thm}

\bigskip
This lemma also has a version for the comparison between volume forms instead of metrics:
\begin{thm}
Let $M$ be a complete K\"{a}hler manifold with scalar curvature bounded from below by a constant $K_1$, and let $N$ be another Hermitian manifold with Ricci curvature bounded from above by a negative constant $K_2$. Suppose the Ricci curvature of $M$ is bounded form below and $\textrm{dim}(M)=\textrm{dim}(N)$.  Then if there is a non-constant holomorphic mapping $f$ from $M$ into $N$, we have $K_1\leq 0$ and
\[f^*dV_N\leq\dfrac{K_1}{K_2}dV_M.\]
\end{thm}

\chapter{Classical Canonical Metrics on the Teichnm\"{u}ller Spaces and the Moduli Spaces of Riemann Surfaces}
There are a few renowned classical metrics on the Teichnm\"{u}ller spaces $T_g$ and the moduli spaces $\mathcal{M}_g$ which have been studied intensively before.  These metrics are important in the study of the geometry and topology of $T_g$ and $\mathcal{M}_g$.  In this chapter we will briefly review three (complete) Finsler metrics: the Teichm\"{u}ller metric, the Kobayashi metric and the Carath\'{e}odory metric, and two K\"{a}hler metrics: the Bergman metric and the Weil-Petersson metric.

\section{Finsler Metrics and Bergman Metric}
The definitions and some basic properties of the Teichm\"{u}ller metric $\omega_T$, the Kobayashi metric $\omega_K$, the Carath\'{e}odory metric $\omega_C$ and the Bergman metric $\omega_B$ will be briefly discussed in this section.

\subsection{Definitions and Properties of the Metrics}
The Teichm\"{u}ller metric can be defined like this:
\begin{defn}
Let $X\in\mathcal{M}_g$, and let $\varphi=\varphi(z)dz^2\in Q(X)\cong T_X^*\mathcal{M}_g$.  Then the Teichm\"{u}ller $L^1$ norm on $T_X^*\mathcal{M}_g$ is given by
\[\left\|\varphi\right\|_T=\int_X|\varphi(z)|dzd\overline{z}.\]
The Teichm\"{u}ller norm on $T_X\mathcal{M}_g$ is thus defined by the duality as
\[\left\|\mu\right\|_T=\sup\{\textrm{Re}[\mu;\varphi]|\left\|\varphi\right\|_T=1\}\]
for $\mu\in HB(X)\cong T_X\mathcal{M}_g$.
\end{defn}

\bigskip
Note that $\omega_T$ has constant holomorphic sectional curvature equal to $-1$.

\bigskip
The Kobayashi metric and the Carath\'{e}odory metric can indeed be defined for any complex space.  Let $Y$ be a complex manifold of dimension $n$, and let $\Delta_R$ be the disk in $\mathbb{C}$ with radius $R$.  Now take $\Delta=\Delta_1$, and suppose $\rho$ is the Poincar\'{e} metric on $\Delta$.  Also, define $\textrm{Hol}(Y, \Delta_R)$ and $\textrm{Hol}(\Delta_R, Y)$ as the spaces of holomorphic maps from $Y$ to $\Delta_R$ and from $\Delta_R$ to $Y$ respectively.
\begin{defn}
Let $p\in Y$ and $v\in T_p Y$ be a holomorphic tangent vector.  The Carath\'{e}odory norm on $T_p Y$ is defined to be
\[\left\|v\right\|_C=\sup_{f\in\textrm{Hol}(Y, \Delta)}\left\|f_*v\right\|_{\Delta,\rho},\]
and the Kobayashi norm on $T_p Y$ is defined as
\[\left\|v\right\|_K=\inf_{f\in\textrm{Hol}(\Delta_R, Y), f(0)=p, f'(0)=v}\dfrac{2}{R}.\]
\end{defn}

\bigskip
The Bergman metric can be also be defined for any complex space $Y$, provided that the Bergman kernel is positive.  Let $K_Y$ be the canonical bundle of $Y$ and let $W$ be the space of $L^2$ holomorphic sections of $K_Y$ in the sense that if $\sigma\in W$, then
\[\left\|\sigma\right\|_{L^2}^2=\int_Y(\sqrt{-1})^{n^2}\sigma\wedge\overline{\sigma}<\infty.\]
The inner product on $W$ is defined by
\[<\sigma,\vartheta>=\int_Y(\sqrt{-1})^{n^2}\sigma\wedge\overline{\vartheta}\]
for $\sigma, \vartheta\in W$.
\begin{defn}
Let $\sigma_1$, $\sigma_2$, ... be an orthonormal basis on $W$.  The Bergman kernel form is the non-negative $(n, n)$-form
\[B_Y=\sum_{j=1}^{\infty}(\sqrt{-1})^{n^2}\sigma_j\wedge\overline{\sigma_j}.\]
In local coordinates $z_1$, ..., $z_n$ on $Y$ one can express $B_Y$ as
\[B_Y=BE_Y(z, \overline{z})(\sqrt{-1})^{n^2}dz_1\wedge...\wedge dz_n\wedge d\overline{z_1}\wedge...\wedge d\overline{z_n}\]
where $BE_Y(z, \overline{z})$ is called the Bergman kernel function.  If $B_Y$ is positive, then the Bergman metric can be defined by
\[B_{i\overline{j}}=\dfrac{\partial^2\log BE_Y(z, \overline{z})}{\partial z_i\partial \overline{z_j}}.\]
\end{defn}

\bigskip
Note that the Bergman metric is well-defined and is non-degenerate if the elements in $W$ separate points and the first jet of Y, and it becomes a K\"{a}hler metric in this case.  Also note that the Bergman metric $\omega_B'$ on $\mathcal{M}_g$ and the induced Bergman metric $\omega_B$ on $\mathcal{M}_g$ from the Bergman metric on $T_g$ (by abuse of notation, let's still denote it by $\omega_B$) are different in the sense that $\omega_B$ is complete while $\omega_B'$ is incomplete.

\subsection{Equivalences of the Metrics}
Royden\cite{bib27} showed that the Teichm\"{u}ller metric and the Kobayashi metric in fact coincide in $T_g$ (and in $\mathcal{M}_g$ as well).  Now we go into the equivalences between the Kobayashi metric, the Carath\'{e}odory metric and the Bergman metric.

\bigskip
Here are some facts about $\omega_K$, $\omega_C$ and $\omega_B$ which can be used to prove the equivalence of the Kobayashi metric and the Carath\'{e}odory metric and the equivalence of the Kobayashi metric and the Bergman metric.
\begin{lem} (Lemma 2.1, \cite{bib2})
Let $X$ be a complex space.  Then
\begin{itemize}
\item[(i)] $\left\|\cdot\right\|_{C, X}\leq \left\|\cdot\right\|_{K, X}$;
\item[(ii)] Let $Y$ be another complex space and $f:X\rightarrow Y$ be a holomorphic map.  Let $p\in X$ and $v\in T_p X$.  Then
$\left\|f_*(v)\right\|_{C, Y, f(p)}\leq \left\|v\right\|_{C, X, p}$ and $\left\|f_*(v)\right\|_{K, Y, f(p)}\leq \left\|v\right\|_{K, X, p}$;
\item[(iii)] If $X\subset Y$ is a connected open subset, then $BE_Y(z)\leq BE_X(z)$ for any local coordinates around $z\in X$;
\item[(iv)] If the Bergman kernel of $X$ is positive, then at each $z\in X$, a peak section $\sigma$ at $z$ exists.  This existence is unique up to a constant factor $c$ with norm 1.  Furthermore, we have $BE_X(z)=|\sigma(z)|^2$ for any local coordinates around $z$;
\item[(v)] If the Bergman kernel of $X$ is positive, then $\left\|\cdot\right\|_{C, X}\leq 2\left\|\cdot\right\|_{B, X}$;
\item[(vi)] If $X$ is a bounded convex domain in $\mathbb{C}^n$, then $\left\|\cdot\right\|_{C, X}=\left\|\cdot\right\|_{K, X}$;
\item[(vii)] Let $|\cdot|$ be the Euclidean norm and let $B_r$ be the open ball with center $0$ and radius $r$ in $\mathbb{C}^n$.  Then for any holomorphic tangent vector $v$ at $0$,
\[\left\|v\right\|_{C, B_r, 0}=\left\|v\right\|_{K, B_r, 0}=\dfrac{2}{r}|v|.\]
\end{itemize}
\end{lem}
\begin{proof}
(i)-(vi) are in fact results of Kobayashi (Proposition 4.2.4, Proposition 4.2.3, Proposition 3.5.18, Proposition 4.10.4, Proposition 4.10.3, Theorem 4.10.18, Theorem 4.8.13, \cite{bib28}).  The last claim follows form (ii) and (vi) by considering the maps $i:\Delta_r\rightarrow B_r$ and $j:B_r\rightarrow \Delta_r$ given by $i(z)=(z, 0, ..., 0)$ and $j(z_1, ..., z_n)=z_1$.
\end{proof}

\bigskip
The following theorem proved by Liu-Sun-Yau (Theorem 2.1, \cite{bib2}) tells us the equivalence of the Kobayashi metric and the Carath\'{e}odory metric and the equivalence of the Kobayashi metric and the Bergman metric on $T_g$.  Its proof is quite a direct comparison between the mentioned metrics, so it is omitted here.
\begin{thm}
There is a positive constant $C$ only depending on $g$ such that for each $X\in T_g$ and each $v\in T_X T_g$, we have
\[C^{-1}\left\|v\right\|_{K, T_g, X}\leq \left\|v\right\|_{C, T_g, X}\leq C\left\|v\right\|_{K, T_g, X}\]
and
\[C^{-1}\left\|v\right\|_{K, T_g, X}\leq \left\|v\right\|_{B, T_g, X}\leq C\left\|v\right\|_{K, T_g, X}.\]
\end{thm}

\newpage

\section{Weil-Petersson Metric}
In this section the Weil-Petersson metric $\omega_{WP}$ will be introduced.  Some of the basic properties and the curvature formula of $\omega_{WP}$ will also be given.

\subsection{Definition and Properties of the Weil-Petersson Metric}
Let $\mathcal{M}_g$ be the moduli space of a Riemann surface of genus $g$ with $g\geq 2$.  Also, let $\mathcal{X}$ be the total space and $\pi:\mathcal{X}\rightarrow\mathcal{M}_g$ be the projection map.  Before defining the Weil-Petersson metric, we have to fix some notations first.

Let $s_1, ..., s_n$ be holomorphic coordinates near a regular point $s\in\mathcal{M}_g$ and assume that $z$ is a holomorphic local coordinate on the fiber $X_s=\pi^{-1}(s)$.  For the holomorphic vector fields $\dfrac{\partial}{\partial s_1}, ..., \dfrac{\partial}{\partial s_n}$, there are vector fields $v_1, ..., v_n$ on $\mathcal{X}$ such that
\begin{itemize}
\item[(i)] $\pi_*(v_i)=\dfrac{\partial}{\partial s_i}$ for $i=1, ..., n$;
\item[(ii)] $\overline{\partial} v_i$ are harmonic $TX_s$-valued $(0,1)$-forms for $i=1, ..., n$.
\end{itemize}
The vector fields $v_1, ..., v_n$ are called the harmonic lifts of $\dfrac{\partial}{\partial s_1}, ..., \dfrac{\partial}{\partial s_n}$, whose existence was proved by Siu\cite{bib31}.  In fact, Schumacher\cite{bib30} gave the following explicit construction of the harmonic lift:

Let $\lambda=\dfrac{\sqrt{-1}}{2}\lambda(z,s)dz\wedge d\overline{z}$ be the K\"{a}hler-Einstein metric on the fiber $X_s$.  For convenience, let $\partial_z=\dfrac{\partial}{\partial z}$ and $\partial_i=\dfrac{\partial}{\partial s_i}$.  Notice that the K\"{a}hler-Einstein condition gives us $\lambda=\partial_z\partial_{\overline{z}}\log\lambda$.  Now let $a_i=-\lambda^{-1}\partial_i\partial_{\overline{z}}\log\lambda$ and $A_i=\partial_{\overline{z}}a_i$.  Then we have
\begin{lem} (Lemma 2.1, \cite{bib1})
The harmonic horizontal lift of $\partial_i$ is $v_i=\partial_i+a_i\partial_z$.  In particular,
\[B_i=A_i\partial_z\otimes d\overline{z}\in H^1(X_s, TX_s)\]
 is harmonic.  Furthermore, the lift $\partial_i\mapsto B_i$ gives the Kodaira-Spencer map $T_s\mathcal{M}_g\rightarrow H^1(X_s, TX_s)$.
\end{lem}

\bigskip
Now we are ready to define $\omega_{WP}$.
\begin{defn}
The Weil-Petersson metric on $\mathcal{M}_g$ is defined as
\[h_{i\overline{j}}=\int_{X_s}B_i\cdot\overline{B_j}dv=\int_{X_s}A_i\overline{A_j}dv\]
where $dv=\dfrac{\sqrt{-1}}{2}\lambda dz\wedge d\overline{z}$ is the volume form on the fiber $X_s$.
\end{defn}

\bigskip
Notice that $\omega_{WP}$ is incomplete and the Ricci curvature of $\omega_{WP}$ is negative.

\subsection{Results about Harmonic Lifts}
The harmonic lifts $v_1, ..., v_n$ will paly an essential role in calculating the curvature formulae of the Weil-Petersson metric and the Ricci metric (which will be discussed in the next chapter), so in this section some useful results related to $v_1, ..., v_n$ will be provided.

Let us review several operators which will be used for the simplifications of the computations afterwards in this thesis.

Define an $(1, 1)$-form on the total space $\mathcal{X}$ by
\[G=\dfrac{\sqrt{-1}}{2}\partial\overline{\partial}\log\lambda=\dfrac{\sqrt{-1}}{2}(G_{i\overline{j}}ds_i\wedge d\overline{s_j}-\lambda a_ids_i\wedge d\overline{z})-\lambda\overline{a_i}dz\wedge d\overline{s_i}+\lambda dz\wedge d\overline{z}.\]
Then introduce
\[e_{i\overline{j}}=\dfrac{2}{\sqrt{-1}}G(v_i, \overline{v_j})=G_{i\overline{j}}-\lambda a_i\overline{a_j}\]
to be a global function.  Also, let $f_{i\overline{j}}=A_i\overline{A_j}$.  Schumacher showed the following lemma:
\begin{lem}
We have
\[f_{i\overline{j}}=(\Box+1)e_{i\overline{j}}\]
where $\Box=-\lambda^{-1}\partial_z\partial_{\overline{z}}$.
\end{lem}

\bigskip
It will be useful to study the action of $v_i$ on $e_{i\overline{j}}$.
\begin{lem} (Lemma 3.2, \cite{bib1})
We have
\[v_k(e_{i\overline{j}})=v_i(e_{k\overline{j}}).\]
\end{lem}
\begin{proof}
By applying $dG$ to $v_i$, $v_k$ and $\overline{v_j}$, and noting that $dG=0$, $g(v_i, v_k)=0$ and all $[v_k,\overline{v_j}]$, $[v_i,\overline{v_j}]$, $[v_i,v_k]$ are tangent to $X_s$ and perpendicular to $v_i$, $v_k$, $\overline{v_j}$ with respect to $G$ respectively, the lemma follows.
\end{proof}

Let us define a new operator which is related to the Maass operators:
\begin{defn}
Define
\[P:C^{\infty}(X_s)\rightarrow \Gamma(\Lambda^{1,0}(T^{0,1}X_s))\]
by $P(f)=\partial_z(\lambda^{-1}\partial_z f)$.  Notice that its dual operator
\[P^*:\Gamma(\Lambda^{1,0}(T^{0,1}X_s))\rightarrow C^{\infty}(X_s)\]
is given by $P^*(B)=\lambda^{-1}\partial_z(\lambda^{-1}\partial_z(\lambda B))$.
\end{defn}

\bigskip
We also recall the definitions of Maass operators by Wolpert\cite{bib11}.  Let $X$ be a Riemann surface and $\kappa$ be its canonical bundle.  For any integer $p$, let $S(p)$ be the space of smooth sections of $(\kappa\otimes\overline{\kappa}^{-1})^{\frac{p}{2}}$.  Also, let $ds^2=\rho^2(z)|dz|^2$ be a conformal metric.
\begin{defn}
The Maass operators $K_p$ and $L_p$ are defined to be the metric derivative $K_p:S(p)\rightarrow S(p+1)$ and $L_p:S(p)\rightarrow S(p-1)$ given by
\[K_p(\sigma)=\rho^{p-1}\partial_z(\rho^{-p}\sigma)\]
and
\[L_p(\sigma)=\rho^{-p-1}\partial_z(\rho^p\sigma)\]
for $\sigma\in S(p)$.
\end{defn}

\bigskip
By definition we have $P=K_1K_0$.

\bigskip
Since $\omega_{WP}$ is defined using the integral along fibers, the following formula will help us to calculate the curvature formula:
\[\partial_i\int_{X_s}\eta=\int_{X_s}L_{v_i}\eta\]
where $\eta$ is a relative $(1,1)$-form on $\mathcal{X}$.  Notice that the Lie derivative is defined by
\[L_{v_i}\eta=\lim_{t\rightarrow 0}\dfrac{1}{t}((\varphi_t)^*\eta-\eta)\]
for a form $\sigma$ on $X_s$ and by
\[L_{v_i}\xi=\lim_{t\rightarrow 0}\dfrac{1}{t}((\varphi_{-t})_*\xi-\xi)\]
for a vector field $\xi$ on $X_s$, where $\varphi_t$ is the one parameter group generated by $v_i$.  For convenience, denote $L_{v_i}$ by $L_i$.  Then we have
\begin{prop} (Proposition 3.1, \cite{bib1})
\[L_i\sigma=i(v_i)d_1\sigma+d_1i(v_1)\sigma\]
where $d_1$ is the differential operator along the fiber $X_s$ and
\[L_i\xi=[v_i,\xi].\]
\end{prop}

\bigskip
Now we have the following lemmas (Lemma 3.3-3.4, \cite{bib1}) concerning the Lie derivative $L_i$ of $B_j$:
\begin{lem}
We have
\begin{itemize}
\item[(i)] $L_idv=0$;
\item[(ii)] $L_{\overline{l}}(B_i)=-\overline{P}(e_{i\overline{l}})-f_{i\overline{l}}\partial_{\overline{z}}\otimes d\overline{z}+f_{i\overline{l}}\partial_z\otimes dz$;
\item[(iii)] $L_k(\overline{B_i})=-P(e_{k\overline{j}})-f_{k\overline{j}}\partial_z\otimes dz+f_{k\overline{j}}\partial_{\overline{z}}\otimes d\overline{z}$;
\item[(iv)] $L_k(B_i)=(v_k(A_i)-A_i\partial_za_k)\partial_z\otimes d\overline{z}$;
\item[(v)] $L_{\overline{l}}(\overline{B_j})=(\overline{v_l}(\overline{A_j})-\overline{A_j}\partial_{\overline{z}}\overline{a_l})\partial_{\overline{z}}\otimes dz$.
\end{itemize}
\end{lem}
\begin{lem}
$L_k(B_i)\in H^1(X_s,TX_s)$ is harmonic.
\end{lem}

\bigskip
The harmonicity of $L_kB_i$ is helpful in simplifying the curvature formula if we consider the normal coordinates by noting the following corollary (Corollary 3.1, \cite{bib1}):
\begin{cor}
Let $s_1,...s_n$ be normal coordinates at $s\in\mathcal{M}_g$ with respect to $\omega_{WP}$.  Then at $s$ we have
\[L_kB_i=0\]
for all $i,k$.
\end{cor}
\begin{proof}
Note that $B_1, ..., B_n$ is a basis of $T_s\mathcal{M}_g$, by harmonicity we can express $L_kB_i$ as
\[L_kB_i=h^{p\overline{q}}(\int_{X_s}L_kB_i\cdot \overline{B_q}dv)B_p=h^{p\overline{q}}\partial_kh_{i\overline{q}}B_p\]
where the last equality follows form normality.
\end{proof}

The following formula for the commutator of $v_k$ and $\overline{v_l}$ is quite useful:
\begin{lem}(Lemma 3.5, \cite{bib1})
\[[\overline{v_l},v_k]=-\lambda^{-1}\partial_{\overline{z}}e_{k\overline{l}}\partial_z+\lambda{-1}\partial_ze_{k\overline{l}}\partial_{\overline{z}}.\]
\end{lem}
\begin{proof}
Since we have
\begin{align*}
[\overline{v_l},v_k]&=\overline{v_l}v_k-v_k\overline{v_l}\\
&=(\partial_{\overline{l}}+\overline{a_l}\partial_{\overline{z}})(\partial_k+a_k\partial_z)\\
&\quad-(\partial_k+a_k\partial_z)(\partial_{\overline{l}}+\overline{a_l}\partial_{\overline{z}})\\
&=\partial_{\overline{l}}\partial_k+\overline{a_l}\partial_{\overline{z}}\partial_k+\overline{v_l}(a_k)\partial_z+a_k\partial_{\overline{l}}\partial_z+\overline{a_l}a_k\partial_{\overline{z}}\partial_z\\
&\quad-\partial_k\partial_{\overline{l}}-a_k\partial_z\partial_{\overline{l}}-v_k(\overline{a_l})\partial_{\overline{z}}-\overline{a_l}\partial_k\partial_{\overline{z}}-a_k\overline{a_l}\partial_z\partial_{\overline{z}}\\
&=\overline{v_l}(a_k)\partial_z-v_k(\overline{a_l})\partial_{\overline{z}},
\end{align*}
and by Lemma 3.2.4 we know that
\[\overline{v_l}(a_k)=\partial_{\overline{l}}a_k+\overline{a_l}\partial_{\overline{z}}a_k=-\lambda^{-1}\partial_{\overline{z}}e_{k\overline{l}}-\overline{a_l}A_k+\overline{a_l}A_k=-\lambda^{-1}\partial_{\overline{z}}e_{k\overline{l}}\]
and
\[v_k(\overline{a_l})=\partial_k\overline{a_l}+a_k\partial_z\overline{a_l}=-\lambda{-1}\partial_ze_{k\overline{l}}-a_k\overline{A_l}+a_k\overline{A_l}=-\lambda{-1}\partial_ze_{k\overline{l}},\]
we obtain the desired result.
\end{proof}

\subsection{Curvature Formula for the Weil-Petersson Metric}
Let us fix some notation for curvature.  Let $(M, g)$ be a K\"{a}hler manifold.  Then the curvature tensor is given by
\[R_{i\overline{j}k\overline{l}}=\dfrac{\partial^2 g_{i\overline{j}}}{\partial z_k\partial\overline{z_l}}-g^{p\overline{q}}\dfrac{\partial g_{i\overline{q}}}{\partial z_k}\dfrac{\partial g_{p\overline{j}}}{\partial\overline{z_l}}.\]
Hence the Ricci curvature is given by
\[R_{i\overline{j}}=-g^{k\overline{l}}R_{i\overline{j}k\overline{l}}.\]

\bigskip
The following curvature formula for $\omega_{WP}$ was proved by Wolpert\cite{bib29}, Schumacher\cite{bib30} and Siu\cite{bib31}, and its proof can be found in \cite{bib1}:
\begin{thm}
The curvature of $\omega_{WP}$ is given by
\[R_{i\overline{j}k\overline{l}}=\int_{X_s}(e_{i\overline{j}}f_{k\overline{l}}+e_{i\overline{l}}f_{k\overline{j}})dv.\]
\end{thm}

\chapter{K\"{a}hler Metrics on the Teichm\"{u}ller Spaces and the Moduli Spaces of Riemann Surfaces}

\section{McMullen Metric}
The McMullen Metric is a new K\"{a}hler metric on $\mathcal{M}_g$ constructed by McMullen\cite{bib3}.  This metric is equivalent to the Teichm\"{u}ller metric and is K\"{a}hler hyperbolic, which were also proved by McMullen.

\subsection{Definition of the McMullen Metric}
For a closed geodesic $\gamma$ on a Riemann surface of genus $g$, let $l_{\gamma}(X)$ be the length of the corresponding hyperbolic geodesic on $X\in T_g=T(R)$.  Mumford\cite{bib14} showed that a sequence $X_n\in \mathcal{M}_g$ tends to infinity if and only if $\inf_{\gamma}l_{\gamma}(X_n)\rightarrow 0$.  This behavior motivates McMullen to use the reciprocal length functions $\dfrac{1}{l_{\gamma}}$ to define the McMullen metric $g_{1/l}$.
\begin{defn}
The McMullen metric $g_{1/l}$ on the Teichm\"{u}ller space $T_g$ is defined by $g_{1/l}(v,w)=\omega_{1/l}(v,\sqrt{-1}w)$ for $v, w\in T_XT_g$, where $\omega_{1/l}$ is the K\"{a}hler form of $g_{1/l}$ given by
\[\omega_{1/l}=\omega_{WP}-\sqrt{-1}\delta\sum_{l_{\gamma}(X)<\epsilon}\partial\overline{\partial}Log\dfrac{\epsilon}{l_{\gamma}}\]
for suitable choices of small constants $\delta, \epsilon>0$ and $Log: \mathbb{R}_+\rightarrow [0,\infty)$ being a smooth function such that
\[Log(x)=\left\{
\begin{array}{ll}
\log(x) & \text{if } x\geq 2;\\
0 & \text{if } x\leq 1.
\end{array}\right.\]
The sum above is taken over primitive short geodesics $\gamma$ on $X$, which has at most $\dfrac{3g}{2}$ terms.
\end{defn}
\begin{rem}
Notice that $g_{1/l}$ is invariant under the action of $Mod_g$, so $g_{1/l}$ is also defined on the moduli space $\mathcal{M}_g$.
\end{rem}

\subsection{Properties of the McMullen Metric}
As mentioned before, McMullen metric is K\"{a}hler hyperbolic.  Before we go into the definition of K\"{a}hler hyperbolic metric, we should first understand what $d(bounded)$ is.
\begin{defn}
An $n$-form $\alpha$ on a K\"{a}hler manifold is $d(bounded)$ if there exists a bounded $(n-1)$-form $\beta$ such that $\alpha=d\beta$
\end{defn}

\bigskip
Now we can define the K\"{a}hler hyperbolic metric.
\begin{defn}
Let $(M, g)$ be a K\"{a}hler manifold.  $(M, g)$ is K\"{a}hler hyperbolic if the following conditions are satisfied:
\begin{itemize}
\item[(i)] On the universal covering manifold $(\tilde{M}, \tilde{g})$, the K\"{a}hler form $\tilde{\omega}$ of the pulled-back metric $\tilde{g}$ is $d(bounded)$;
\item[(ii)] $(M, g)$ is complete and of finite volume;
\item[(iii)] The sectional curvatures of $(M, g)$ are bounded from above and below;
\item[(iv)] The injectivity radius of $(\tilde{M}, \tilde{g})$ is bounded from below.
\end{itemize}
Note that the last three conditions would be automatic when $M$ is compact.  A K\"{a}hler metric $g$ on $M$ is said to be K\"{a}hler hyperbolic if $(M, g)$ is K\"{a}hler hyperbolic.
\end{defn}

\bigskip
The following theorem is a collection of the McMullen's results (Corollary 5.2, Corollary 7.3 and Theorem 8.1-8.2, \cite{bib3}) which were used to prove that $g_{1/l}$ is K\"{a}hler hyperbolic.
\begin{thm}
The McMullen metric $g_{1/l}$ satisfies the following properties:
\begin{itemize}
\item[(i)] $g_{1/l}$ is complete;
\item[(ii)] The K\"{a}hler form $\omega_{1/l}$ of the McMullen metric $g_{1/l}$ on $T_g$ is $d(bounded)$ with primitive
\[\theta_{1/l}=\theta_{WP}-\delta\sum_{l_{\gamma}(X)<\epsilon}\partial Log\dfrac{\epsilon}{l_{\gamma}}\]
satisfying $d(\sqrt{-1}\theta_{1/l})=\omega_{1/l}$.  Note that $\sqrt{-1}\theta_{WP}$ is a primitive of the Weil-Petersson K\"{a}hler form defined by $\theta_{WP}(X)=-\beta_X(Y)$ for a fixed $Y$ in the conjugate Teichm\"{u}ller space $T_g^{conj}$, where $\beta_X: T_g^{conj}\rightarrow Q(X)$ is the Bers embedding;
\item[(iii)] $(M_g, g_{1/l})$ has finite volume;
\item[(iv)] The sectional curvatures of $g_{1/l}$ are bounded form below and above over $T_g$, and the injectivity radius of $g_{1/l}$ is uniformly bounded from below.
\end{itemize}
Hence $g_{1/l}$ is K\"{a}hler hyperbolic.
\end{thm}
\begin{sproof}
(i) follows from the fact that $g_{1/l}$ is equivalent to the Teichm\"{u}ller metric (which will be discussed in the next subsection) and the Teichm\"{u}ller metric is complete.

To prove that $g_{1/l}$ on $T_g$ is $d(bounded)$ in (ii), it suffices to show that $\theta_{1/l}$ is bounded with respect to $\left\|\cdot\right\|_T$.  The first term $\theta_{WP}$ is bounded using Nehari's bound (Theorem 2.1, \cite{bib3}) and Gauss-Bonnet theorem.  The remaining terms are bounded by the estimate for the gradient of geodesic length: $\left\|\partial l_{\gamma}\right\|_T\leq 2l_{\gamma}$.

For (iii), we consider the Deligne-Mumford compactification $\overline{\mathcal{M}_g}$ of $\mathcal{M}_g$.  Then by the existence of a covering $\{V_{\alpha}\}$ with $\textrm{vol}(V_{\alpha}\cap\mathcal{M}_g)<\infty$ in the Kobayashi metric on $M_g$ and the coincidence of the Teichm\"{u}ller metric and the Kobayashi metric on $M_g$, we know that the Teichm\"{u}ller volume of $M_g$ is finite.  Again using the equivalence of $g_{1/l}$ and the Teichm\"{u}ller metric we can prove (iii).

Finally for (iv), we first extend $l_{\gamma}$ and $\theta_{WP}$ to holomorphic functions on the complexification $QF_g=T_g\times T_g^{conj}$ of $T_g$.  Then we can use the local uniform bounds on these holomorphic functions to control all of their derivatives, and thus the derivatives of $g_{1/l}$ are also controlled, which implies the sectional curvatures are bounded.  The lower bound of the injectivity radius hence follows from the curvature bounds.
\end{sproof}

\subsection{Equivalence of the McMullen Metric and the Teichm\"{u}ller Metric}
Let's fix some notations and present some useful facts before the discussion of the equivalence of the McMullen metric and the Teichm\"{u}ller metric.  Let $N=\dfrac{3g}{2}+1$ be a bound on the number of terms of $\omega_{1/l}$, and let
\[\psi_{\gamma}=\partial\log l_\gamma=\dfrac{\partial l_{\gamma}}{l_{\gamma}}\]
for a given $X\in T_g$.  Note that we have
\[|\psi_{\gamma}(v)|^2=\dfrac{\sqrt{-1}}{2}\dfrac{\partial l_{\gamma}\wedge\overline{\partial}l_{\gamma}}{l_{\gamma}^2}(v, \sqrt{-1}v).\]
Now let $\pi: X_{\gamma}\rightarrow X$ be the covering space corresponding to $<\gamma>\subset\pi_1(R)$.  Notice that we can identify $X_{\gamma}$ with an annulus $A(r)=\{z|r^{-1}<|z|<r\}$ up to rotations by requiring that the orientation of $\tilde{\gamma}\subset X_{\gamma}$ agrees with that of $\mathbb{S}^1$.  Define $\phi_{\gamma}\in Q(X)$ by
\[\phi_{\gamma}=\pi_*(\dfrac{dz^2}{z}).\]
Wolpert (Theorem 3.1, \cite{bib12}) showed the following important relation between $\phi_{\gamma}$ and $\psi_{\gamma}$:
\[\psi_{\gamma}=-\dfrac{l_{\gamma}(X)}{2\pi^3}\phi_{\gamma}.\]
Here are the facts about the quadratic differentials which will be used in the proof of the equivalence of $g_{1/l}$ and the Teichm\"{u}ller metric (Theorem 4.1-4.4, \cite{bib3}):
\begin{thm}
For $\epsilon>0$ sufficiently small, any $\phi\in Q(X)$ can be uniquely expressed as
\[\phi=\phi_0+\sum_{l_{\gamma}(X)<\epsilon}a_{\gamma}\phi_{\gamma}\]
with $\textrm{Res}_{\gamma}(\phi_0)=0$ for all $\gamma$ in the sum above such that
\begin{itemize}
\item[(i)] $\left\|\phi_0\right\|_T=O(\left\|\phi\right\|_T)$ and $\left\|a_{\gamma}\phi_{\gamma}\right\|_T=O(\left\|\phi\right\|_T)$;
\item[(ii)] $\left\|\phi_0\right\|_{WP}\leq C(\epsilon)\left\|\phi_0\right\|_T$ for some constant $C(\epsilon)$ depending on $\epsilon$.
\end{itemize}
If $\phi=\phi_{\delta}$ with $\epsilon<l_{\delta}(X)<2\epsilon$, then we have $\left\|a_{\gamma}\phi_{\gamma}\right\|_T=O(l_{\delta}(X)\left\|\phi\right\|_T)$.  Moreover, $\psi_{\gamma}$ is proportional to $\phi_{\gamma}$ such that $\left\|\psi_{\gamma}\right\|_T\leq 2$ and $\left\|\psi_{\gamma}\right\|_T\rightarrow 2$ as $l_{\gamma}\rightarrow 0$.
\end{thm}

\bigskip
The following two lemmas will play the main roles in the proof of the equivalence as well.
\begin{lem}
There is a Hermitian metric $g$ of the form
\[g(v,v)=A(\epsilon)\left\|v\right\|_{WP}^2+B\sum_{l_{\gamma}(X)<\epsilon}|\psi_{\gamma}(v)|^2\]
such that $\left\|v\right\|_T^2\leq g(v,v)\leq O(\left\|v\right\|_T^2)$ for all sufficiently small $\epsilon>0$.  Hence we have
\[\left\|v\right\|_T^2\sim\left\|v\right\|_{WP}^2+\sum_{l_{\gamma}(X)<\epsilon}|\partial\log l_{\gamma}(v)|^2\]
for $\epsilon>0$ sufficiently small and for all $v\in T_X T_g$.
\end{lem}
\begin{proof}
By Cauchy-Schwarz inequality and the duality between $Q(X)=T^*_X T_g$ and $M(X)=T_X T_g$, we know that $\left\|v\right\|_{WP}=O(\left\|v\right\|_T)$ (Proposition 2.4, \cite{bib3}).  Now by Theorem 4.1.2 we have $|\psi_{\gamma}(v)|\leq 2\left\|v\right\|_T$, and since there are at most $N$ terms in the expression of $g$, we obtain $g(v,v)\leq O(\left\|v\right\|_T^2)$.

To prove $\left\|v\right\|_T^2\leq g(v,v)$, we first fix $v\in T_X T_g$ and pick $\phi\in Q(X)$ such that $\left\|\phi\right\|_T=1$ and $\phi(v)=\left\|v\right\|_T$.  Then as long as $\epsilon>0$ is small enough, by Theorem 4.1.2 we can obtain
\[\phi=\phi_0+\sum_{l_{\gamma}(X)<\epsilon}a_{\gamma}\psi_{\gamma}\]
with $\textrm{Res}_{\gamma}(\phi_0)=0$ and $\left\|\psi_{\gamma}\right\|_T\geq1$.  Moreover we have $\left\|\phi_0\right\|_T=O(1)$ and $\left\|a_{\gamma}\psi_{\gamma}\right\|_T=O(1)$ (as $\left\|\phi\right\|_T=1$).  Then by duality and Theorem 4.1.2 again we have $|\phi_0(v)|\leq D(\epsilon)\left\|v\right\|_{WP}$ for some constant $D(\epsilon)$ depending on $\epsilon$.  Since $\left\|\psi_{\gamma}\right\|_T\geq1$ and $\left\|a_{\gamma}\psi_{\gamma}\right\|_T=O(1)$, we know that $|a_{\gamma}|\leq E$, where $E$ is a constant independent of $\epsilon$.  Hence we have
\[\left\|v\right\|_T=\phi(v)\leq D(\epsilon)\left\|v\right\|_{WP}+E\sum_{l_{\gamma}(X)<\epsilon}|\psi_{\gamma}|.\]
Notice that $N$ is a bound for the number of terms in the sum above, so by Cauchy-Schwarz inequality we obtain
\[\left\|v\right\|_T^2\leq ND(\epsilon)^2\left\|v\right\|_{WP}^2+NE^2\sum_{l_{\gamma}(X)<\epsilon}|\psi_{\gamma}|^2.\]
Thus taking $A(\epsilon)=ND(\epsilon)$ and $B=NE^2$, the result follows.
\end{proof}
\begin{lem}
For $\epsilon<l_{\delta}(X)<2\epsilon$, we have
\[|\psi_{\delta}(v)|^2\leq D(\epsilon)\left\|v\right\|_{WP}^2+O(\epsilon\sum_{l_{\gamma}(X)<\epsilon}|\psi_{\gamma}(v)|^2).\]
\end{lem}
\begin{proof}
By Theorem 4.1.2 we can derive that $\displaystyle \psi_{\delta}=\psi_0+\epsilon\sum_{l_{\gamma}(X)<\epsilon}a_{\gamma}\psi_{\gamma}$ with $a_{\gamma}=O(l_{\delta}(X))=O(\epsilon)$ and $|\psi_0(v)|\leq D(epsilon)\left\|v\right\|_{WP}$.  Then using similar argument as in Lemma 4.1.1 we obtain the result.
\end{proof}

Now we are ready to prove the equivalence.
\begin{thm}
For all sufficiently small $\epsilon>0$, there exists a $\delta>0$ such that $\omega_{1/l}$ defines a K\"{a}hler metric $g_{1/l}$ on $T_g$ which is equivalent to the Teichm\"{u}ller metric.
\end{thm}
\begin{proof}
Consider the $(1,1)$-form
\[\omega_1=(F(\epsilon)+A(\epsilon))\omega_{WP}-B\sum_{l_{\delta}(X)<2\epsilon}\dfrac{\sqrt{-1}}{2}\partial\overline{\partial}Log\dfrac{2\epsilon}{l_{\delta}}\]
where
\[F(\epsilon)=4NBD(\epsilon)\sup_{[1,2]}|Log''(x)|.\]
Observe that $\omega_1$ and $\omega_{1/l}$ are indeed of the same form (up to scaling and replacing $\epsilon$ by $\dfrac{\epsilon}{2}$).  Therefore, to prove the theorem it suffices to show that $g_1(v,v)=\omega_1(v,\sqrt{-1}v)$ is equivalent to the Teichm\"{u}ller metric.

Now let $v\in T_X T_g$ have Teichm\"{u}ller norm $\left\|v\right\|_T=1$.  We first compute
\[\partial\overline{\partial}Log\dfrac{2\epsilon}{l_{\delta}}=2\epsilon\bigg(Log'\dfrac{2\epsilon}{l_{\delta}}\bigg)\partial\overline{\partial}\bigg(\dfrac{1}{l_{\delta}}\bigg)+\dfrac{4\epsilon^2}{l_{\delta}^2}\bigg(Log''\dfrac{2\epsilon}{l_{\delta}}\bigg)\dfrac{\partial l_{\delta}\wedge\overline{\partial}l_{\delta}}{l_{\delta}^2}.\]
Using the fact that the function $\dfrac{1}{l_{\delta}}$ is pluriharmonic (Theorem 3.1, \cite{bib3}), we have
\[\bigg|\partial\overline{\partial}\bigg(\dfrac{1}{l_{\delta}}\bigg)(v,\sqrt{-1}v)\bigg|=O(1).\]
Since $Log'(x)$ is bounded and $|\psi_{\gamma}(v)|^2=\dfrac{\sqrt{-1}}{2}\dfrac{\partial l_{\gamma}\wedge\overline{\partial}l_{\gamma}}{l_{\gamma}^2}(v, \sqrt{-1}v)$, we know that
\[\dfrac{\sqrt{-1}}{2}\partial\overline{\partial}Log\dfrac{2\epsilon}{l_{\delta}}(v,\sqrt{-1}v)=\dfrac{4\epsilon^2}{l_{\delta}^2}\bigg(Log''\dfrac{2\epsilon}{l_{\delta}}\bigg)|\psi_{\delta}(v)|^2+O(\epsilon).\]
Now if $l_{\delta}<\epsilon$, then $Log''\dfrac{2\epsilon}{l_{\delta}}=\log''\dfrac{2\epsilon}{l_{\delta}}=-\dfrac{l_{\delta}^2}{4\epsilon^2}$, and thus
\[-\dfrac{\sqrt{-1}}{2}\partial\overline{\partial}Log\dfrac{2\epsilon}{l_{\delta}}(v,\sqrt{-1}v)=|\psi_{\delta}(v)|^2+O(\epsilon).\]
If $\epsilon\leq l_{\delta}<2\epsilon$, by Lemma 4.1.2 we have
\begin{align*}
\bigg|\dfrac{\sqrt{-1}}{2}\partial\overline{\partial}Log\dfrac{2\epsilon}{l_{\delta}}(v,\sqrt{-1}v)\bigg|&\leq 4\psi_{\delta}(v)|^2\sup_{[1,2]}|Log''(x)|+O(\epsilon)\\
&\leq \dfrac{F(\epsilon)}{NB}\left\|v\right\|_{WP}^2+O(\epsilon\sum_{l_{\gamma}(X)<\epsilon}|\psi_{\gamma}(v)|^2)=O(\epsilon).
\end{align*}
Hence using these two estimates we obtain
\begin{align*}
g_1(v,v)&\geq A(\epsilon)\left\|v\right\|_{WP}^2+B\sum_{l_{\delta}(X)<\epsilon}|\psi_{\delta}(v)|^2\\
&\quad+B\sum_{\epsilon\leq l_{\delta}(X)<2\epsilon}\bigg(\dfrac{F(\epsilon)}{NB}\left\|v\right\|_{WP}^2-\bigg|\dfrac{\sqrt{-1}}{2}\partial\overline{\partial}Log\dfrac{2\epsilon}{l_{\delta}}(v,\sqrt{-1}v)\bigg|\bigg)+O(\epsilon)\\
&\geq A(\epsilon)\left\|v\right\|_{WP}^2+B(1+O(\epsilon))\sum_{l_{\delta}(X)<\epsilon}|\psi_{\delta}(v)|^2+O(\epsilon).
\end{align*}
Then by Lemma 4.1.1 we have
\[g_1(v,v)\geq \left\|v\right\|_T^2+O(\epsilon)=1+O(\epsilon),\]
which implies that $\left\|v\right\|_T^2=O(g_1(v,v))$ when $\epsilon$ is sufficiently small.  By similar argument in Lemma 4.1.1 we can also prove $g_1(v,v)=O(\left\|v\right\|_T^2)$.  This finishes the proof.
\end{proof}
\newpage

\section{K\"{a}hler-Einstein Metric}

\subsection{Existence of the K\"{a}hler-Einstein Metric}
A K\"{a}hler-Einstein metric $g$ is a K\"{a}hler metric whose Ricci tensor is proportional to itself, i.e., $R_{i\overline{j}}=cg_{i\overline{j}}$.  Cheng-Yau\cite{bib17} and Mok-Yau\cite{bib18} showed the existence of the K\"{a}hler-Einstein Metric $g_{KE}$ on $T_g$ with Ricci curvature equal to $-1$.  More precisely, the existence of $g_{KE}$ is ensured by the fact that $T_g$ is a bounded domain of holomorphy (note that Bers-Ehrenpreis\cite{bib19} showed that $T_g$ is a domain of holomorphy and Bers\cite{bib20}\cite{bib21} showed that $T_g$ is a bounded domain in $\mathbb{C}^{3g-3}$) and the following theorem (Main Theorem, \cite{bib18}):
\begin{thm}
Any bounded domain of holomorphy admits a complete K\"{a}hler-Einstein metric.
\end{thm}

\bigskip
Notice that $g_{KE}$ descends to $\mathcal{M}_g$ as a canonical K\"{a}hler metric as well.

\subsection{A Conjecture of Yau}
In \cite{bib15}, Yau proposed the following conjecture:
\begin{conj}
The K\"{a}hler-Einstein metric is equivalent to the Teichm\"{u}ller metric.
\end{conj}

\bigskip
Finally it was confirmed by Liu-Sun-Yau\cite{bib1}, and we will review the answer to this conjecture in chapter 5.

\newpage

\section{Ricci Metric}
The Ricci metric is one of the new complete K\"{a}hler metrics (another one being the perturbed Ricci metric, which will be examined in the next section) introduced by Liu-Sun-Yau\cite{bib1}.  Their main purpose of introducing these two metrics was to build up a bridge for proving the equivalence of the K\"{a}hler-Einstein metric to the Teichm\"{u}ller metric (and McMullen metric).

\subsection{Definition of the Ricci Metric}
Ahlfors\cite{bib8} proved that the Ricci curvature of the Weil-Petersson metric is negative, and Trapani\cite{bib10} further showed that the negative of the Ricci curvature of the Weil-Petersson metric defines a (complete) K\"{a}hler metric on the moduli space $\mathcal{M}_g$.  This is just the Ricci metric brought by Yau et. al..
\begin{defn}
The Ricci metric $\tau$ on the moduli space $\mathcal{M}_g$ is defined by
\[\tau_{i\overline{j}}=-R_{i\overline{j}}=h^{\alpha\overline{\beta}}R_{i\overline{j}\alpha\overline{\beta}}.\]
\end{defn}

\bigskip
Before we step into the next subsection, let's go through some useful definitions and a lemma. Here we first define a new operator (which will play a main role in the curvature formula) acting on functions on the fibers.
\begin{defn}
For each $1\leq k\leq n$ and for any $f\in C^{\infty}(X_s)$, we define the commutator operator $\xi_k$ which acts on $C^{\infty}(X_s)$ by
\[\xi_k(f)=\overline{\partial}^*(i(B_k)\partial f)=-\lambda^{-1}\partial_z(A_k\partial_zf).\]
\end{defn}

\bigskip
$\xi_k$ is called the commutator operator because $\xi_k$ is indeed the commutator of $(\Box+1)$ and $v_k$, which we can see from the following lemma (Lemma 3.6, \cite{bib1}):
\begin{lem}
As operators acting on $C^{\infty}(X_s)$, we have
\begin{itemize}
\item[(i)] $(\Box+1)v_k-v_k(\Box+1)=\Box v_k-v_k\Box=\xi_k;$
\item[(ii)] $(\Box+1)\overline{v_l}-\overline{v_l}(\Box+1)=\Box \overline{v_l}-\overline{v_l}\Box=\overline{\xi}_l;$
\item[(iii)] $\xi_k(f)=-A_k\partial_z(\lambda^{-1}\partial_zf)=-A_kP(f)=-A_kK_1K_0(f).$
\end{itemize}
Furthermore, we have
\[(\Box+1)v_k(e_{i\overline{j}})=\xi_k(e_{i\overline{j}})+\xi_i(e_{k\overline{j}})+L_kB_i\cdot \overline{B_j}\]
\end{lem}
\begin{proof}
Notice that (ii) can be derived from (i) by taking conjugation.  To prove (i), we consider
\begin{align*}
(\Box+1)v_k-v_k(\Box+1)&=\Box v_k+v_k-v_k\Box-v_k=\Box v_k-v_k\Box\\
&=-\lambda^{-1}\partial_z\partial_{\overline{z}}(a_k\partial_z+\partial_k)-(a_k\partial_z+\partial_k)(-\lambda^{-1}\partial_z\partial_{\overline{z}})\\
&=-\lambda^{-1}\partial_z(\partial_{\overline{z}}a_k\partial_z+a_k\partial_{\overline{z}}\partial_z+\partial_{\overline{z}}\partial_k)+a_k\partial_z(\lambda^{-1})\partial_z\partial_{\overline{z}}\\
&\quad+\lambda^{-1}a_k\partial_z\partial_z\partial_{\overline{z}}+\partial_k(\lambda^{-1})\partial_z\partial_{\overline{z}}+\lambda^{-1}\partial_k\partial_z\partial_{\overline{z}}\\
&=-\lambda^{-1}\partial_z(A_k\partial_z)-\lambda^{-1}\partial_za_k\partial_{\overline{z}}\partial_z-\lambda^{-1}a_k\partial_z\partial_{\overline{z}}\partial_z-\lambda^{-1}\partial_z\partial_{\overline{z}}\partial_k\\
&-\lambda^{-2}\partial_z\lambda a_k\partial_z\partial_{\overline{z}}+\lambda^{-1}a_k\partial_z\partial_z\partial_{\overline{z}}-\lambda^{-2}\partial_k\lambda\partial_z\partial_{\overline{z}}+\lambda^{-1}\partial_k\partial_z\partial_{\overline{z}}\\
&=\xi_k-\lambda^{-1}(\partial_za_k+\lambda^{-1}\partial_z\lambda a_k+\lambda^{-1}\partial_k\lambda)\partial_z\partial_{\overline{z}}=\xi_k.
\end{align*}
Note that we used the fact that $\partial_za_k=-\lambda^{-1}\partial_z\lambda a_k-\lambda^{-1}\partial_k\lambda$ in the last equality.  Now we can prove (iii).  Using the harmonicity of $B_k$, we know that $\overline{\partial}^*B_k=0$, which implies that $\partial_z(\lambda A_k)=0$.  Hence
\begin{align*}
\xi_k(f)&=-\lambda^{-1}\partial_z(A_k\partial_zf)=-\lambda^{-1}\partial_z(\lambda A_k\lambda^{-1}\partial_zf)\\
&=-\lambda^{-1}\partial_z(\lambda A_k)\lambda^{-1}\partial_zf-\lambda^{-1}\lambda A_k\partial_z(\lambda^{-1}\partial_zf)\\
&=-A_kP(f)=-A_kK_1K_0(f).
\end{align*}
The last part can be proved by (i) in this way:
\begin{align*}
(\Box+1)v_k(e_{i\overline{j}})&=v_k((\Box+1)(e_{i\overline{j}}))+\xi_k(e_{i\overline{j}})=v_k(f_{i\overline{j}})+\xi_k(e_{i\overline{j}})\\
&=L_k(B_i\cdot\overline{B_j})+\xi_k(e_{i\overline{j}})=L_k(B_i)\cdot\overline{B_j}+B_i\cdot L_k(\overline{B_j})+\xi_k(e_{i\overline{j}})\\
&=L_k(B_i)\cdot\overline{B_j}-A_iP(e_{k\overline{j}})+\xi_k(e_{i\overline{j}})=L_k(B_i)\cdot\overline{B_j}+\xi_i(e_{k\overline{j}})+\xi_k(e_{i\overline{j}}).
\end{align*}
\end{proof}

Since the terms in the curvature formula will be highly symmetric with respect to indices, it is more convenient to introduce the symmetrization operators.
\begin{defn}
Let $U$ be any quantity which depends on indices $i, k, \alpha, \overline{j}, \overline{l}, \overline{\beta}$. The symmetrization operator $\sigma_1$ is defined by taking the summation of all orders of the triple $(i, k, \alpha)$, i.e.,
\begin{align*}
\sigma_1(U(i, k, \alpha, \overline{j}, \overline{l}, \overline{\beta}))&=U(i, k, \alpha, \overline{j}, \overline{l}, \overline{\beta})+U(i, \alpha, k, \overline{j}, \overline{l}, \overline{\beta})+U(k, i, \alpha, \overline{j}, \overline{l}, \overline{\beta})\\
&\quad+U(k, \alpha, i, \overline{j}, \overline{l}, \overline{\beta})+U(\alpha, i, k, \overline{j}, \overline{l}, \overline{\beta})+U(\alpha, k, i, \overline{j}, \overline{l}, \overline{\beta}).
\end{align*}
In a similar fashion, we define $\sigma_2$ to be the symmetrization operator of $\overline{j}$ and $\overline{\beta}$ and $\tilde{\sigma_1}$ to be the symmetrization operator of $\overline{j}, \overline{l}, \overline{\beta}$.
\end{defn}

\subsection{Curvature Formula of the Ricci Metric}
As to compute the curvature of the Ricci metric, we need to compute the first and second order derivatives of $\tau_{i\overline{j}}$.  Let's compute the first order derivatives first.

\begin{thm} (Theorem 3.2, \cite{bib1})
\[\partial_k\tau_{i\overline{j}}=h^{\alpha\overline{\beta}}\bigg\{\sigma_1\int_{X_s}(\xi_k(e_{i\overline{j}})e_{\alpha\overline{\beta}})dv\bigg\}+\tau_{p\overline{j}}\Gamma^p_{ik}\]
where $\Gamma^p_{ik}$ is the Christoffel symbol of the Weil-Petersson metric.
\end{thm}
\begin{proof}
Notice that $\Box+1$ is self-adjoint and by Lemma 3.2.2 we know that $(\Box+1)e_{i\overline{j}}=f_{i\overline{j}}$.  Then using Lemma 4.3.1 and Theorem 3.2.1 we have
\begin{align*}
\partial_kR_{i\overline{j}\alpha\overline{\beta}}&=\partial_k\int_{X_s}(e_{i\overline{j}}f_{\alpha\overline{\beta}}+e_{i\overline{\beta}}f_{\alpha\overline{j}})dv\\
&=\int_{X_s}L_k\{(e_{i\overline{j}}f_{\alpha\overline{\beta}}+e_{i\overline{\beta}}f_{\alpha\overline{j}})dv\}\\
&=\int_{X_s}(v_k(e_{i\overline{j}})f_{\alpha\overline{\beta}}+e_{i\overline{j}}v_k(f_{\alpha\overline{\beta}})+v_k(e_{i\overline{\beta}})f_{\alpha\overline{j}}+e_{i\overline{\beta}}v_k(f_{\alpha\overline{j}}))dv\qquad\qquad\qquad\
\end{align*}
\begin{align*}
\quad\quad\quad\ &=\int_{X_s}(v_k(e_{i\overline{j}})(\Box+1)e_{\alpha\overline{\beta}}+e_{i\overline{j}}v_k(f_{\alpha\overline{\beta}})+v_k(e_{i\overline{\beta}})(\Box+1)e_{\alpha\overline{j}}+e_{i\overline{\beta}}v_k(f_{\alpha\overline{j}}))dv\\
&=\int_{X_s}((\Box+1)v_k(e_{i\overline{j}})e_{\alpha\overline{\beta}}+e_{i\overline{j}}v_k(f_{\alpha\overline{\beta}})+(\Box+1)v_k(e_{i\overline{\beta}})e_{\alpha\overline{j}}+e_{i\overline{\beta}}v_k(f_{\alpha\overline{j}}))dv\\
&=\int_{X_s}((v_k(\Box+1)+\xi_k)(e_{i\overline{j}})e_{\alpha\overline{\beta}}+e_{i\overline{j}}v_k(f_{\alpha\overline{\beta}})+(v_k(\Box+1)+\xi_k)(e_{i\overline{\beta}})e_{\alpha\overline{j}}\\
&\quad+e_{i\overline{\beta}}v_k(f_{\alpha\overline{j}}))dv\\
&=\int_{X_s}(v_k(f_{i\overline{j}})e_{\alpha\overline{\beta}}+e_{i\overline{j}}v_k(f_{\alpha\overline{\beta}})+v_k(f_{i\overline{\beta}})e_{\alpha\overline{j}}+e_{i\overline{\beta}}v_k(f_{\alpha\overline{j}}))dv\\
&\quad+\int_{X_s}(\xi_k(e_{i\overline{j}})e_{\alpha\overline{\beta}}+\xi_k(e_{i\overline{\beta}})e_{\alpha\overline{j}})dv\\
&=\int_{X_s}((L_kB_i\cdot\overline{B_j})e_{\alpha\overline{\beta}}+(L_kB_{\alpha}\cdot\overline{B_{\beta}})e_{i\overline{j}}+(L_kB_i\cdot\overline{B_{\beta}})e_{\alpha\overline{j}}+(L_kB_{\alpha}\cdot\overline{B_j})e_{i\overline{\beta}})dv\\
&\quad+\int_{X_s}((B_i\cdot L_k\overline{B_j})e_{\alpha\overline{\beta}}+(B_{\alpha}\cdot L_k\overline{B_{\beta}})e_{i\overline{j}}+(B_i\cdot L_k\overline{B_{\beta}})e_{\alpha\overline{j}}+(B_{\alpha}\cdot L_k\overline{B_j})e_{i\overline{\beta}})dv\\
&\quad+\int_{X_s}(\xi_k(e_{i\overline{j}})e_{\alpha\overline{\beta}}+\xi_k(e_{i\overline{\beta}})e_{\alpha\overline{j}})dv=I+II+III.
\end{align*}
Notice that $B_1, ..., B_n$ is a basis of $T_sM_g$, so we can simplify $I$ in this way:
\begin{align*}
I&=\int_{X_s}(L_kB_i\cdot(\overline{B_j}e_{\alpha\overline{\beta}}+\overline{B_{\beta}}e_{\alpha\overline{j}})+L_kB_{\alpha}\cdot(\overline{B_{\beta}}e_{i\overline{j}}+\overline{B_j}e_{i\overline{\beta}})dv\\
&=<L_kB_i,\overline{B_j}e_{\alpha\overline{\beta}}+\overline{B_{\beta}}e_{\alpha\overline{j}}>_{WP}+<L_kB_{\alpha},\overline{B_{\beta}}e_{i\overline{j}}+\overline{B_j}e_{i\overline{\beta}}>_{WP}\\
&=h^{p\overline{q}}<L_kB_i,\overline{B_q}>_{WP}<B_p,\overline{B_j}e_{\alpha\overline{\beta}}+\overline{B_{\beta}}e_{\alpha\overline{j}}>_{WP}\\
&\quad+h^{p\overline{q}}<L_kB_{\alpha},\overline{B_q}>_{WP}<B_p,\overline{B_{\beta}}e_{i\overline{j}}+\overline{B_j}e_{i\overline{\beta}}>_{WP}\\
&=h^{p\overline{q}}\int_{X_s}(L_kB_i\cdot\overline{B_q})dv\int_{X_s}(B_p\cdot(\overline{B_j}e_{\alpha\overline{\beta}}+\overline{B_{\beta}}e_{\alpha\overline{j}}))dv\\
&\quad+h^{p\overline{q}}\int_{X_s}(L_kB_{\alpha}\cdot\overline{B_q})dv\int_{X_s}(B_p\cdot(\overline{B_{\beta}}e_{i\overline{j}}+\overline{B_j}e_{i\overline{\beta}}))dv\\
&=h^{p\overline{q}}\int_{X_s}(L_kB_i\cdot\overline{B_q})dv\int_{X_s}(f_{p\overline{j}}e_{\alpha\overline{\beta}}+f_{p\overline{\beta}}e_{\alpha\overline{j}})dv\\
&\quad+h^{p\overline{q}}\int_{X_s}(L_kB_{\alpha}\cdot\overline{B_q})dv\int_{X_s}(f_{p\overline{\beta}}e_{i\overline{j}}+f_{p\overline{j}}e_{i\overline{\beta}})dv\\
&=h^{p\overline{q}}\partial_kh_{i\overline{q}}R_{p\overline{j}\alpha\overline{\beta}}+h^{p\overline{q}}\partial_kh_{\alpha\overline{q}}R_{i\overline{j}p\overline{\beta}}=\Gamma^p_{ik}R_{p\overline{j}\alpha\overline{\beta}}+\Gamma^p_{\alpha k}R_{i\overline{j}p\overline{\beta}}.
\end{align*}
By Lemma 3.2.4 and Lemma 4.3.1, we have
\begin{align*}
B_i\cdot L_k\overline{B_j}&=(A_i\partial_z\otimes d\overline{z})\cdot(-\partial_z(\lambda^{-1}\partial_ze_{k\overline{j}})\partial_z\otimes d\overline{z}-f_{k\overline{j}}\partial_z\otimes dz+f_{k\overline{j}}\partial_{\overline{z}}\otimes d\overline{z})\\
&=-A_i\partial_z(\lambda^{-1}\partial_ze_{k\overline{j}})=\xi_i(e_{k\overline{j}}).
\end{align*}
Hence $II$ can be simplified as well: \[II=\int_{X_s}(\xi_i(e_{k\overline{j}})e_{\alpha\overline{\beta}}+\xi_{\alpha}(e_{k\overline{\beta}})e_{i\overline{j}}+\xi_i(e_{k\overline{\beta}})e_{\alpha\overline{j}}+\xi_{\alpha}(e_{k\overline{j}})e_{i\overline{\beta}})dv.\]
Then using the fact that $\xi_k$ is a real symmetric operator, i.e., $\displaystyle \int_{X_s}\xi_k(e_{i\overline{\beta}})e_{\alpha\overline{j}}dv=\int_{X_s}\xi_k(e_{\alpha\overline{j}})e_{i\overline{\beta}}dv$, we know that
\begin{align*}
\partial_kR_{i\overline{j}\alpha\overline{\beta}}&=\Gamma^p_{ik}R_{p\overline{j}\alpha\overline{\beta}}+\Gamma^p_{\alpha k}R_{i\overline{j}p\overline{\beta}}\\
&\quad+\int_{X_s}(\xi_i(e_{k\overline{j}})e_{\alpha\overline{\beta}}+\xi_{\alpha}(e_{k\overline{\beta}})e_{i\overline{j}}+\xi_i(e_{k\overline{\beta}})e_{\alpha\overline{j}}+\xi_{\alpha}(e_{k\overline{j}})e_{i\overline{\beta}})dv\\
&\quad+\int_{X_s}(\xi_k(e_{i\overline{j}})e_{\alpha\overline{\beta}}+\xi_k(e_{i\overline{\beta}})e_{\alpha\overline{j}})dv\\
&=\Gamma^p_{ik}R_{p\overline{j}\alpha\overline{\beta}}+\Gamma^p_{\alpha k}R_{i\overline{j}p\overline{\beta}}+\int_{X_s}(\xi_i(e_{k\overline{j}})e_{\alpha\overline{\beta}}+\xi_i(e_{\alpha\overline{j}})e_{k\overline{\beta}})dv\\
&\quad+\int_{X_s}(\xi_{\alpha}(e_{k\overline{\beta}})e_{i\overline{j}}+\xi_{\alpha}(e_{i\overline{\beta}})e_{k\overline{j}})dv+\int_{X_s}(\xi_k(e_{i\overline{j}})e_{\alpha\overline{\beta}}+\xi_k(e_{\alpha\overline{j}})e_{i\overline{\beta}})dv\\
&=\Gamma^p_{ik}R_{p\overline{j}\alpha\overline{\beta}}+\Gamma^p_{\alpha k}R_{i\overline{j}p\overline{\beta}}+\sigma_1\int_{X_s}(\xi_k(e_{i\overline{j}})e_{\alpha\overline{\beta}})dv.
\end{align*}
Finally by the definition of $\tau_{i\overline{j}}$ we obtain
\begin{align*}
\partial_k\tau_{i\overline{j}}&=h^{\alpha\overline{\beta}}\partial_kR_{i\overline{j}\alpha\overline{\beta}}+\partial_kh^{\alpha\overline{\beta}}R_{i\overline{j}\alpha\overline{\beta}}\\
&=h^{\alpha\overline{\beta}}(\partial_kR_{i\overline{j}\alpha\overline{\beta}}-R_{i\overline{j}p\overline{\beta}}\Gamma^p_{k\alpha})\\
&=h^{\alpha\overline{\beta}}(\sigma_1\int_{X_s}(\xi_k(e_{i\overline{j}})e_{\alpha\overline{\beta}})dv+\Gamma^p_{ik}R_{p\overline{j}\alpha\overline{\beta}})\\
&=h^{\alpha\overline{\beta}}\bigg\{\sigma_1\int_{X_s}(\xi_k(e_{i\overline{j}})e_{\alpha\overline{\beta}})dv\bigg\}+\tau_{p\overline{j}}\Gamma^p_{ik}.
\end{align*}
\end{proof}

Before we go into the computation of the second order derivatives of $\tau_{i\overline{j}}$, let's first compute the commutator of $\xi_k$ and $\overline{v_l}$.
\begin{lem} (Lemma 3.7, \cite{bib1})
For any smooth function $f\in C^{\infty}(X_s)$,
\[\overline{v_l}(\xi_kf)-\xi_k(\overline{v_l}f)=\overline{P}(e_{k\overline{l}})-2f_{k\overline{l}}\Box f+\lambda^{-1}\partial_zf_{k\overline{l}}\partial_{\overline{z}}f.\]
\end{lem}
\begin{proof}
Let's fix local holomorphic coordinates and compute locally.  We first consider the commutator of $\overline{v_l}$ and $\partial_z$:
\begin{align*}
\overline{v_l}\partial_z-\partial_z\overline{v_l}&=\partial_{\overline{l}}\partial_z+\overline{a_l}\partial_{\overline{z}}\partial_z-\partial_z\partial_{\overline{l}}-\partial_z\overline{a_l}\partial_{\overline{z}}-\overline{a_l}\partial_z\partial_{\overline{z}}\\
&=-\partial_z\overline{a_l}\partial_{\overline{z}}=-\overline{A_l}\partial_{\overline{z}}.
\end{align*}
Then we know that the commutator of $\overline{v_l}$ and $\lambda^{-1}\partial_z$ is
\begin{align*}
\overline{v_l}(\lambda^{-1}\partial_z)-\lambda^{-1}\partial_z(\overline{v_l})&=\overline{v_l}(\lambda^{-1})\partial_z+\lambda^{-1}(\overline{v_l}\partial_z-\partial_z\overline{v_l})\\
&=-\lambda^{-2}(\partial_{\overline{l}}\lambda+\overline{a_l}\partial_{\overline{z}}\lambda)\partial_z-\lambda^{-1}\overline{A_l}\partial_{\overline{z}}\\
&=\lambda^{-1}\partial_{\overline{z}}\overline{a_l}\partial_z-\lambda^{-1}\overline{A_l}\partial_{\overline{z}}.
\end{align*}
These two formulae imply
\begin{align*}
\overline{v_l}P-P\overline{v_l}&=-\overline{v_l}(\partial_z(\lambda^{-1}\partial_z))+\partial_z(\lambda^{-1}\partial_z)\overline{v_l}\\
&=(\overline{A_l}\partial_{\overline{z}}-\partial_z\overline{v_l})(\lambda^{-1}\partial_z)+\partial_z(\overline{v_l}(\lambda^{-1}\partial_z)-\lambda^{-1}\partial_{\overline{z}}\overline{a_l}\partial_z+\lambda^{-1}\overline{A_l}\partial_{\overline{z}})\\
&=\overline{A_l}\partial_{\overline{z}}(\lambda^{-1}\partial_z)-\partial_z(\lambda^{-1}\partial_{\overline{z}}\overline{a_l}\partial_z)+\partial_z(\lambda^{-1}\overline{A_l}\partial_{\overline{z}})\\
&=-\lambda^{-2}\partial_{\overline{z}}\lambda\overline{A_l}\partial_z+\lambda^{-1}\overline{A_l}\partial_{\overline{z}}\partial_z+\lambda^{-2}\partial_z\lambda\partial_{\overline{z}}\overline{a_l}\partial_z-\lambda^{-1}\partial_{\overline{z}}\overline{A_l}\partial_z-\lambda^{-1}\partial_{\overline{z}}\overline{a_l}\partial_z\partial_z\\
&-\lambda^{-2}\partial_z\lambda\overline{A_l}\partial_{\overline{z}}+\lambda^{-1}\partial_z\overline{A_l}\partial_{\overline{z}}+\lambda^{-1}\overline{A_l}\partial_z\partial_{\overline{z}}.
\end{align*}
Note that the harmonicity gives us $\partial_{\overline{z}}(\lambda\overline{A_l})=0$, so we have $\partial_{\overline{z}}\overline{A_l}=-\lambda^{-1}\partial_{\overline{z}}\lambda\overline{A_l}$.  Thus we obtain
\begin{align*}
\overline{v_l}P-P\overline{v_l}&=\lambda^{-1}\partial_{\overline{z}}\overline{A_l}\partial_z-\overline{A_l}\Box-\partial_{\overline{z}}\overline{a_l}\partial_z(\lambda^{-1}\partial_z)-\lambda^{-1}\partial_{\overline{z}}\overline{a_l}\partial_z\partial_z\\
&-\lambda^{-2}\partial_z\lambda\overline{A_l}\partial_{\overline{z}}+\lambda^{-1}\partial_z\overline{A_l}\partial_{\overline{z}}-\overline{A_l}\Box\\
&=-2\overline{A_l}\Box-\partial_{\overline{z}}\overline{a_l}P-\lambda^{-2}\partial_z\lambda\overline{A_l}\partial_{\overline{z}}+\lambda^{-1}\partial_z\overline{A_l}\partial_{\overline{z}}.
\end{align*}
Then since by Lemma 4.3.1 we know that $\xi_k=-A_kP$, we have
\begin{align*}
\overline{v_l}(\xi_kf)-\xi_k(\overline{v_l}f)&=-\overline{v_l}(A_k)P(f)+A_k(\overline{v_l}P(f)-P\overline{v_l}(f))\\
&=-\overline{v_l}(A_k)P(f)-2A_k\overline{A_l}\Box f-A_k\partial_{\overline{z}}\overline{a_l}P(f)-\lambda^{-2}\partial_z\lambda A_k\overline{A_l}\partial_{\overline{z}}f\\
&\quad+\lambda^{-1}A_k\partial_z\overline{A_l}\partial_{\overline{z}}f\\
&=-(\overline{v_l}(A_k)+A_k\partial_{\overline{z}}\overline{a_l})P(f)-2f_{k\overline{l}}\Box f-\lambda^{-2}\partial_z\lambda A_k\overline{A_l}\partial_{\overline{z}}f\\
&\quad+\lambda^{-1}A_k\partial_z\overline{A_l}\partial_{\overline{z}}f.
\end{align*}
Now we use the fact that $\overline{v_l}(A_k)+A_k\partial_{\overline{z}}\overline{a_l}=-\overline{P}(e_{k\overline{l}})$ and the following result from harmonicity: $-\lambda^{-1}\partial_z\lambda A_k=\partial_zA_k$, so we finish the proof:
\begin{align*}
\overline{v_l}(\xi_kf)-\xi_k(\overline{v_l}f)&=\overline{P}(e_{k\overline{l}})P(f)-2f_{k\overline{l}}\Box f+\lambda^{-1}\partial_zA_k\overline{A_l}\partial_{\overline{z}}f+\lambda^{-1}A_k\partial_z\overline{A_l}\partial_{\overline{z}}f\\
&=\overline{P}(e_{k\overline{l}})P(f)-2f_{k\overline{l}}\Box f+\lambda^{-1}\partial_zf_{k\overline{l}}\partial_{\overline{z}}f.
\end{align*}
\end{proof}

For convenience we may define the commutator of $\xi_k$ and $\overline{v_l}$ as an operator.
\begin{defn}
For each $k, l$, we define the operator $Q_{k\overline{l}}: C^{\infty}(X_s)\longrightarrow C^{\infty}(X_s)$ by
\[Q_{k\overline{l}}(f)=\overline{P}(e_{k\overline{l}})P(f)-2f_{k\overline{l}}\Box f+\lambda^{-1}\partial_zf_{k\overline{l}}\partial_{\overline{z}}f\]
\end{defn}

\bigskip
We can now compute the curvature formula of the Ricci metric (Theorem 3.3, \cite{bib1}).
\begin{thm}
Let $s_1, ..., s_n$ be local holomorphic coordinates at $s\in M_g$.  Then at $s$, we have
\begin{align*}
\tilde{R}_{i\overline{j}k\overline{l}}&=h^{\alpha\overline{\beta}}\bigg\{\sigma_1\sigma_2\int_{X_s}\big\{(\Box+1)^{-1}(\xi_k(e_{i\overline{j}}))\overline{\xi}_l(e_{\alpha\overline{\beta}})+(\Box+1)^{-1}(\xi_k(e_{i\overline{j}}))\overline{\xi}_{\beta}(e_{\alpha\overline{l}})\big\}dv\bigg\}\\
&\quad+h^{\alpha\overline{\beta}}\bigg\{\sigma_1\int_{X_s}Q_{k\overline{l}}(e_{i\overline{j}})e_{\alpha\overline{\beta}}dv\bigg\}\\
&\quad-\tau^{p\overline{q}}h^{\alpha\overline{\beta}}h^{\gamma\overline{\delta}}\bigg\{\sigma_1\int_{X_s}\xi_k(e_{i\overline{q}})e_{\alpha\overline{\beta}}dv\bigg\}\bigg\{\tilde{\sigma}_1\int_{X_s}\overline{\xi}_l(e_{p\overline{j}})e_{\gamma\overline{\delta}}dv\bigg\}\\
&\quad+\tau_{p\overline{j}}h^{p\overline{q}}R_{i\overline{q}k\overline{l}},
\end{align*}
where $\tilde{R}_{i\overline{j}k\overline{l}}$ is the curvature tensor of the Ricci metric.
\end{thm}
\begin{proof}
Let's compute the second derivatives of $\tau_{i\overline{j}}$ first.  By Lemma 4.3.1 and Lemma 4.3.2 we obtain
\begin{align*}
\partial_{\overline{l}}\int_{X_s}\xi_k(e_{i\overline{j}})e_{\alpha\overline{\beta}}dv&=\int_{X_s}(\overline{v_l}(\xi_k(e_{i\overline{j}}))e_{\alpha\overline{\beta}}+\xi_k(e_{i\overline{j}})\overline{v_l}(e_{\alpha\overline{\beta}}))dv\\
&=\int_{X_s}(\xi_k(\overline{v_l}(e_{i\overline{j}}))e_{\alpha\overline{\beta}}+\xi_k(e_{i\overline{j}})\overline{v_l}(e_{\alpha\overline{\beta}})+Q_{k\overline{l}}(e_{i\overline{j}})e_{\alpha\overline{\beta}})dv\\
&=\int_{X_s}(\xi_k(e_{\alpha\overline{\beta}})\overline{v_l}(e_{i\overline{j}})+\xi_k(e_{i\overline{j}})\overline{v_l}(e_{\alpha\overline{\beta}})+Q_{k\overline{l}}(e_{i\overline{j}})e_{\alpha\overline{\beta}})dv\\
&=\int_{X_s}(\Box+1)^{-1}(\xi_k(e_{\alpha\overline{\beta}}))(\Box+1)(\overline{v_l}(e_{i\overline{j}}))dv\\
&\quad+\int_{X_s}(\Box+1)^{-1}(\xi_k(e_{i\overline{j}}))(\Box+1)(\overline{v_l}(e_{\alpha\overline{\beta}}))dv\\
&\quad+\int_{X_s}Q_{k\overline{l}}(e_{i\overline{j}})e_{\alpha\overline{\beta}}dv\\
&=\int_{X_s}(\Box+1)^{-1}(\xi_k(e_{\alpha\overline{\beta}}))(\overline{\xi}_l(e_{i\overline{j}})+\overline{\xi}_j(e_{i\overline{l}})+L_{\overline{l}}\overline{B_j}\cdot B_{i})dv\\
&\quad+\int_{X_s}(\Box+1)^{-1}(\xi_k(e_{i\overline{j}}))(\overline{\xi}_l(e_{\alpha\overline{\beta}})+\overline{\xi}_{\beta}(e_{\alpha\overline{l}})+L_{\overline{l}}\overline{B_{\beta}}\cdot B_{\alpha})dv\\
&\quad+\int_{X_s}Q_{k\overline{l}}(e_{i\overline{j}})e_{\alpha\overline{\beta}}dv.
\end{align*}
From Lemma 3.2.5 we know that $L_kB_i$ is harmonic.  As $B_1, ..., B_n$ is a basis of the space of harmonic Beltrami differentials, and notice that $\displaystyle \int_{X_s} L_kB_i\cdot \overline{B_q}dv=\partial_kh_{i\overline{q}}$, we have $L_kB_i=\Gamma^s_{ik}B_s$.  Hence we can derive that
\begin{align*}
&\quad\int_{X_s}((\Box+1)^{-1}(\xi_k(e_{\alpha\overline{\beta}}))(L_{\overline{l}}\overline{B_j}\cdot B_{i})+(\Box+1)^{-1}(\xi_k(e_{i\overline{j}}))(L_{\overline{l}}\overline{B_{\beta}}\cdot B_{\alpha}))dv\\
&=\int_{X_s}((\Box+1)^{-1}(\xi_k(e_{\alpha\overline{\beta}}))(\overline{\Gamma^t_{jl}}\overline{A_t}A_i)+(\Box+1)^{-1}(\xi_k(e_{i\overline{j}}))(\overline{\Gamma^t_{\beta l}}\overline{A_t}A_{\alpha}))dv\\
&=\overline{\Gamma^t_{jl}}\int_{X_s}\xi_k(e_{\alpha\overline{\beta}})(\Box+1)^{-1}(A_i\overline{A_t})dv+\overline{\Gamma^t_{\beta l}}\int_{X_s}\xi_k(e_{i\overline{j}})(\Box+1)^{-1}(A_{\alpha}\overline{A_t})dv\\
&=\overline{\Gamma^t_{jl}}\int_{X_s}\xi_k(e_{\alpha\overline{\beta}})e_{i\overline{t}}dv+\overline{\Gamma^t_{\beta l}}\int_{X_s}\xi_k(e_{i\overline{j}})e_{\alpha\overline{t}}dv.
\end{align*}
Therefore we have
\begin{align*}
\partial_{\overline{l}}\int_{X_s}\xi_k(e_{i\overline{j}})e_{\alpha\overline{\beta}}dv&=\int_{X_s}(\Box+1)^{-1}(\xi_k(e_{\alpha\overline{\beta}}))(\overline{\xi}_l(e_{i\overline{j}})+\overline{\xi}_j(e_{i\overline{l}}))dv\\
&\quad+\int_{X_s}(\Box+1)^{-1}(\xi_k(e_{i\overline{j}}))(\overline{\xi}_l(e_{\alpha\overline{\beta}})+\overline{\xi}_{\beta}(e_{\alpha\overline{l}}))dv\\
&\quad+\overline{\Gamma^t_{jl}}\int_{X_s}\xi_k(e_{\alpha\overline{\beta}})e_{i\overline{t}}dv+\overline{\Gamma^t_{\beta l}}\int_{X_s}\xi_k(e_{i\overline{j}})e_{\alpha\overline{t}}dv\\
&\quad+\int_{X_s}Q_{k\overline{l}}(e_{i\overline{j}})e_{\alpha\overline{\beta}}dv.
\end{align*}
Moreover, we know that
\begin{align*}
\partial_{\overline{l}}\Gamma^p_{ik}&=\partial_{\overline{l}}(h^{p\overline{q}}\partial_kh_{i\overline{q}})\\
&=-h^{p\overline{\beta}}h^{\alpha\overline{q}}\partial_{\overline{l}}h_{\alpha\overline{\beta}}\partial_kh_{i\overline{q}}+h^{p\overline{q}}\partial_{\overline{l}}\partial_kh_{i\overline{q}}\\
&=h^{p\overline{q}}(\partial_{\overline{l}}\partial_kh_{i\overline{q}}-h^{\alpha\overline{\beta}}\partial_{\overline{l}}h_{\alpha\overline{q}}\partial_kh_{i\overline{\beta}})\\
&=h^{p\overline{q}}R_{i\overline{q}k\overline{l}}.
\end{align*}
Then we can use Theorem 4.3.1 to get the formula for the second derivatives of $\tau_{i\overline{j}}$:
\begin{align*}
\partial_{\overline{l}}\partial_k\tau_{i\overline{j}}&=(\partial_{\overline{l}}h^{\alpha\overline{\beta}})\bigg\{\sigma_1\int_{X_s}(\xi_k(e_{i\overline{j}})e_{\alpha\overline{\beta}})dv\bigg\}+h^{\alpha\overline{\beta}}\bigg\{\sigma_1\partial_{\overline{l}}\int_{X_s}(\xi_k(e_{i\overline{j}})e_{\alpha\overline{\beta}})dv\bigg\}\\
&\quad+\partial_{\overline{l}}\tau_{p\overline{j}}\Gamma^p_{ik}+\tau_{p\overline{j}}\partial_{\overline{l}}\Gamma^p_{ik}\\
&=-h^{\alpha\overline{t}}\overline{\Gamma^{\beta}_{lt}}\bigg\{\sigma_1\int_{X_s}(\xi_k(e_{i\overline{j}})e_{\alpha\overline{\beta}})dv\bigg\}\\
&\quad+h^{\alpha\overline{\beta}}\bigg\{\sigma_1\sigma_2\int_{X_s}(\Box+1)^{-1}(\xi_k(e_{i\overline{j}}))(\overline{\xi}_l(e_{\alpha\overline{\beta}})+\overline{\xi}_{\beta}(e_{\alpha\overline{l}}))dv\bigg\}\\
&\quad+h^{\alpha\overline{\beta}}\bigg\{\sigma_1\int_{X_s}Q_{k\overline{l}}(e_{i\overline{j}})e_{\alpha\overline{\beta}}dv\bigg\}+h^{\alpha\overline{\beta}}\overline{\Gamma^t_{jl}}\bigg\{\sigma_1\int_{X_s}\xi_k(e_{\alpha\overline{\beta}})e_{i\overline{t}}dv\bigg\}\\
&\quad+h^{\alpha\overline{\beta}}\overline{\Gamma^t_{\beta l}}\bigg\{\sigma_1\int_{X_s}\xi_k(e_{i\overline{j}})e_{\alpha\overline{t}}dv\bigg\}+h^{\gamma\overline{\delta}}\bigg\{\tilde{\sigma}_1\int_{X_s}(\overline{\xi}_l(e_{p\overline{j}})e_{\gamma\overline{\delta}})dv\bigg\}\Gamma^p_{ik}\\
&\quad+\tau_{p\overline{q}}\overline{\Gamma^q_{jl}}\Gamma^p_{ik}+\tau_{p\overline{j}}h^{p\overline{q}}R_{i\overline{q}k\overline{l}}.
\end{align*}
Hence we obtain the desired result:
\begin{align*}
\tilde{R}_{i\overline{j}k\overline{l}}&=-h^{\alpha\overline{\beta}}\overline{\Gamma^{t}_{\beta l}}\bigg\{\sigma_1\int_{X_s}(\xi_k(e_{i\overline{j}})e_{\alpha\overline{t}})dv\bigg\}\\
&\quad+h^{\alpha\overline{\beta}}\bigg\{\sigma_1\sigma_2\int_{X_s}(\Box+1)^{-1}(\xi_k(e_{i\overline{j}}))(\overline{\xi}_l(e_{\alpha\overline{\beta}})+\overline{\xi}_{\beta}(e_{\alpha\overline{l}}))dv\bigg\}\\
&\quad+h^{\alpha\overline{\beta}}\bigg\{\sigma_1\int_{X_s}Q_{k\overline{l}}(e_{i\overline{j}})e_{\alpha\overline{\beta}}dv\bigg\}+h^{\alpha\overline{\beta}}\overline{\Gamma^t_{jl}}\bigg\{\sigma_1\int_{X_s}\xi_k(e_{\alpha\overline{\beta}})e_{i\overline{t}}dv\bigg\}\\
&\quad+h^{\alpha\overline{\beta}}\overline{\Gamma^t_{\beta l}}\bigg\{\sigma_1\int_{X_s}\xi_k(e_{i\overline{j}})e_{\alpha\overline{t}}dv\bigg\}+h^{\gamma\overline{\delta}}\bigg\{\tilde{\sigma}_1\int_{X_s}(\overline{\xi}_l(e_{p\overline{j}})e_{\gamma\overline{\delta}})dv\bigg\}\Gamma^p_{ik}\\
&\quad+\tau_{p\overline{q}}\overline{\Gamma^q_{jl}}\Gamma^p_{ik}+\tau_{p\overline{j}}h^{p\overline{q}}R_{i\overline{q}k\overline{l}}\\
&-\tau^{p\overline{q}}\bigg(h^{\alpha\overline{\beta}}\bigg\{\sigma_1\int_{X_s}(\xi_k(e_{i\overline{q}})e_{\alpha\overline{\beta}})dv\bigg\}+\tau_{p\overline{q}}\Gamma^p_{ik}\bigg)\bigg(h^{\gamma\overline{\delta}}\bigg\{\tilde{\sigma}_1\int_{X_s}(\overline{\xi}_l(e_{p\overline{j}})e_{\gamma\overline{\delta}})dv\bigg\}+\tau_{p\overline{q}}\overline{\Gamma^q_{jl}}\bigg)\\
&=h^{\alpha\overline{\beta}}\bigg\{\sigma_1\sigma_2\int_{X_s}\big\{(\Box+1)^{-1}(\xi_k(e_{i\overline{j}}))\overline{\xi}_l(e_{\alpha\overline{\beta}})+(\Box+1)^{-1}(\xi_k(e_{i\overline{j}}))\overline{\xi}_{\beta}(e_{\alpha\overline{l}})\big\}dv\bigg\}\\
&\quad+h^{\alpha\overline{\beta}}\bigg\{\sigma_1\int_{X_s}Q_{k\overline{l}}(e_{i\overline{j}})e_{\alpha\overline{\beta}}dv\bigg\}\\
&-\tau^{p\overline{q}}h^{\alpha\overline{\beta}}h^{\gamma\overline{\delta}}\bigg\{\sigma_1\int_{X_s}\xi_k(e_{i\overline{q}})e_{\alpha\overline{\beta}}dv\bigg\}\bigg\{\tilde{\sigma}_1\int_{X_s}\overline{\xi}_l(e_{p\overline{j}})e_{\gamma\overline{\delta}}dv\bigg\}\\
&\quad+\tau_{p\overline{j}}h^{p\overline{q}}R_{i\overline{q}k\overline{l}}.
\end{align*}
\end{proof}

In fact the curvature formula of the Ricci metric can be simplified using the normal coordinates, but for the sake of easier estimates on the asymptotic behavior of the curvature (since normal coordinates near boundary points are difficult to be described), the general formula will be used in the upcoming computations.

\newpage

\section{The Asymptotic Behavior of the Ricci Metric}
Using the formula of $R_{i\overline{j}k\overline{l}}$ (Theorem 3.2.1) we can obtain the sign of the curvature of the Weil-Petersson metric very easily, but the formula of $\tilde{R}_{i\overline{j}k\overline{l}}$ (Theorem 4.3.2) does not help us find out the sign of the curvature of the Ricci metric directly.  Hence we need to do some estimates on the asymptotics of the Ricci metric so that we can estimate the curvature of the Ricci metric.

\subsection{Estimates on the Asymptotics of the Ricci Metric}
Let's define the $C^k$ norm of $\sigma\in S(p)$ as Wolpert\cite{bib11} did:
\begin{defn}
Let $Q$ be an operator which is a composition of operators $K_*$ and $L_*$.  Denote $|Q|$ to be the number of the factors in the composition.  For any $\sigma\in S(p)$, define
\[\left\|\sigma\right\|_0=\sup_X|\sigma|\]
and
\[\left\|\sigma\right\|_k=\sum_{|Q|\leq k}\left\|Q\sigma\right\|_0.\]
We can also localize the norm on a subset of $X$.  Let $\Omega\subset X$ be a domain.  Then we can define
\[\left\|\sigma\right\|_{0,\Omega}=\sup_{\Omega}|\sigma|\]
and
\[\left\|\sigma\right\|_{k,\Omega}=\sum_{|Q|\leq k}\left\|Q\sigma\right\|_{0,\Omega}.\]
\end{defn}

\bigskip
Note that both of the definitions depend on the choice of conformal metric on $X$.  In this section, $X$ will always be equipped with the K\"{a}hler-Einstein metric unless otherwise specified.

We also need the following version of the Schauder estimate proved by Wolpert (Lemma A.4.2, \cite{bib11}):
\begin{thm}
Let $X$ be a closed Riemann surface.  Let $f$ and $g$ be smooth functions on $X$ such that $(\Box+1)g=f$.  Then for any integer $k\geq 0$, there is a constant $c_k$ such that $\left\|g\right\|_{k+1}\leq c_k\left\|f\right\|_k$.
\end{thm}

\bigskip
Before estimating the asymptotics of the Ricci metric in the pinching coordinates, we need to specify some notations first.  Let $(t, s)=(t_1, ..., t_m, s_{m+1}, ..., s_n)$ be the pinching coordinates near $X_{0,0}$.  For $|(t,s)|<\delta$, let $\Omega^j_c$ be the $j$-th genuine collar on $X_{t,s}$ which contains a short geodesic $\gamma_j$ with length $l_j$.  Let $u_j=\dfrac{l_j}{2\pi}$, $\displaystyle u_0=\sum^m_{j=1}u_j+\sum^n_{j=m+1}|s_j|$, $r_j=|z_j|$ and $\tau_j=u_j\log r_j$, where $z_j$ is the properly normalized rs-coordinate on $\Omega^j_c$ such that
\[\Omega^j_c=\{z_j:c^{-1}e^{-\frac{2\pi^2}{l_j}}<|z_j|<c\}.\]
Hence we know that the K\"{a}hler-Einstein metric $\lambda$ on $X_{t,s}$ restricted to the collar $\Omega^j_c$ is given by
\[\lambda=\dfrac{1}{2}u_j^2r_j^{-2}\csc^2\tau_j.\]
In addition, let $\displaystyle \Omega_c=\cup^m_{j=1}\Omega^j_c$ and $R_c=X_{t,s}-\Omega_c$.  Note that the constant c may be changed finitely many times in the following estimates, but this will not affect the result.  Also note that by Wolpert\cite{bib11}, we know that $u_j=\dfrac{l_j}{2\pi}\sim-\dfrac{\pi}{\log|t_j|}$.

\bigskip
The first step we have to do for the estimates is to find all the harmonic Beltrami differentials $B_1, ..., B_n$ corresponding to the tangent vectors $\frac{\partial}{\partial t_1}, ..., \frac{\partial}{\partial s_n}$. Notice that Masur\cite{bib9} constructed $3g-3$ regular holomorphic quadratic differentials $\psi_1, ..., \psi_n$ corresponding to the cotangent $dt_1, ..., ds_n$ on the plumbing collars using plumbing coordinate.  However, we need to find them in rs-coordinates so that we can have an easier estimate of the curvature.  Let us collect Trapani's results (Lemma 6.2-6.5, \cite{bib10}) in one theorem using Masur's notations and formulation.
\begin{thm}
Let $(t, s)$ be the pinching coordinates on $\overline{\mathcal{M}_g}$ near $X_{0,0}$ which corresponds to a codimension $m$ boundary point of $\overline{\mathcal{M}_g}$.  Then there exist constants $M$, $delta>0$ and $0<c<1$ such that if $(t, s)<delta$, then the $j$-th plumbing collar on $X_{t,s}$ contains the genuine collar $\Omega^j_c$.  Furthermore, a rs-coordinate $z_j$ can be chosen on the collar $\Omega^j_c$ properly such that the holomorphic quadratic differentials $\psi_1, ..., \psi_n$ corresponding to the cotangent vectors $dt_1, ..., ds_n$ have the form $\psi_i=\varphi_i(z_j)dz_j^2$ on the genuine collar $\Omega^j_c$ for $1\leq j\leq m$, where
\begin{itemize}
\item[(i)] $\varphi_i(z_j)=\dfrac{1}{z_j^2}q_i^j(z_j)$ if $i\geq m+1$;
\item[(ii)] $\varphi_i(z_j)=(-\dfrac{t_j}{\pi})\dfrac{1}{z_j^2}(q_j(z_j)+1)$ if $i=j$;
\item[(iii)] $\varphi_i(z_j)=(-\dfrac{t_i}{\pi})\dfrac{1}{z_j^2}q_i^j(z_j)$ if $1\leq i\leq m$ and $i\neq j$,
\end{itemize}
where $\beta_i^j$ and $\beta_j$ are functions of $(t, s)$, $q_i^j$ and $q_j$ are functions of $(t, s, z_j)$ given by
\[q_i^j(z_j)=\sum_{k<0}\alpha^j_{ik}(t,s)t_j^{-k}z_j^k+\sum_{k>0}\alpha^j_{ik}(t,s)z_j^k\]
and
\[q_j(z_j)=\sum_{k<0}\alpha_{jk}(t,s)t_j^{-k}z_j^k+\sum_{k>0}\alpha_{jk}(t,s)z_j^k\]
such that
\begin{itemize}
\item[(1)] $\displaystyle \sum_{k<0}|\alpha^j_{ik}|c^{-k}\leq M$ and $\displaystyle \sum_{k>0}|\alpha^j_{ik}|c^k\leq M$ if $i\neq j$;
\item[(2)] $\displaystyle \sum_{k<0}|\alpha_{jk}|c^{-k}\leq M$ and $\displaystyle \sum_{k>0}|\alpha_{jk}|c^k\leq M$.
\end{itemize}
\end{thm}

\bigskip
Using this theorem we can derive the following refined version of Masur's estimates of the Weil-Petersson metric (Theorem 1, \cite{bib9}).  Note that in the remaining section, $(t, s)$ will be assumed to have small norm and $X=X_{t,s}$.
\begin{cor}
Let $(t, s)$ be the pinching coordinates.  Then
\begin{itemize}
\item[(i)] $h^{i\overline{i}}=2\dfrac{|t_i|^2}{u_i^3}(1+O(u_0))$ and $h_{i\overline{i}}=\dfrac{1}{2}\dfrac{u_i^3}{|t_i|^2}(1+O(u_0))$ if $i\leq m$;
\item[(ii)] $h^{i\overline{j}}=O(|t_it_j|)$ and $h_{i\overline{j}}=O(\dfrac{u_i^3u_j^3}{|t_it_j|})$ if $\leq i, j\leq m$ and $i\neq j$;
\item[(iii)] $h^{i\overline{j}}=O(|t_i|)$ and $h_{i\overline{j}}=O(\dfrac{u_i^3}{|t_i|})$ if $i\leq m$ and $j\geq m+1$;
\item[(iv)] $h^{i\overline{j}}=O(1)$ and $h_{i\overline{j}}=O(1)$ if $i, j\geq m+1$.
\end{itemize}
\end{cor}
\begin{proof}
In fact the last three claims coincide with the Schumacher's formulation of the Masur's result (Theorem [M], \cite{bib13}), so here we just prove (i).  However, we should note that in fact all the claims can be proved using the same method.

Let's make a useful claim about some integral estimates:
\begin{clm}
We have
\begin{itemize}
\item[(1)]$\displaystyle\int_{c^{-1}e^{-\frac{\pi}{u_j}}}^c\dfrac{1}{r_j}\sin^2\tau_jdr_j=\dfrac{1}{u_j}(\dfrac{\pi}{2}+O(u_0));$
\item[(2)] $\displaystyle\int_{c^{-1}e^{-\frac{\pi}{u_j}}}^cr_j^{k-1}\sin^2\tau_jdr_j=O(u_j^2)c^k$ for $k\geq 1$;
\item[(3)] $\displaystyle\int_{c^{-1}e^{-\frac{\pi}{u_j}}}^cr_j^{k-1}\sin^2\tau_jdr_j=O(u_j^2)c^{-k}\Big(e^{-\frac{\pi}{u_j}}\Big)^k$ for $k\leq -1$.
\end{itemize}
\end{clm}
\begin{proof}
For (1), we consider
\begin{align*}
\bigg|\int_{c^{-1}e^{-\frac{\pi}{u_j}}}^c\dfrac{1}{r_j}\sin^2\tau_jdr_j\bigg|&=\dfrac{1}{u_j}\bigg|\int_{-u_j\log c-\pi}^{u_j\log c}\dfrac{1-\cos 2\tau_j}{2}d\tau_j\bigg|\\
&=\dfrac{1}{u_j}\bigg|\bigg[\dfrac{\tau_j}{2}-\dfrac{\sin2\tau_j}{4}\bigg]_{-u_j\log c-\pi}^{u_j\log c}\bigg|\\
&\leq\dfrac{1}{u_j}\bigg(\dfrac{\pi+|2u_j\log c|+|\sin(2u_j\log c)|}{2}\bigg)\\
&=\dfrac{1}{u_j}(\dfrac{\pi}{2}+O(u_0)).
\end{align*}
Then for (2), we have
\begin{align*}
\bigg|\int_{c^{-1}e^{-\frac{\pi}{u_j}}}^cr_j^{k-1}\sin^2\tau_jdr_j\bigg|&\leq\dfrac{1}{k}\bigg|\bigg[r_j^k\sin^2\tau_j\bigg]_{c^{-1}e^{-\frac{\pi}{u_j}}}^c\bigg|\\
&\quad+\dfrac{1}{k}\bigg|\int_{c^{-1}e^{-\frac{\pi}{u_j}}}^cr_j^k\cdot2\sin\tau_j\cos\tau_j\cdot\dfrac{u_j}{r_j}dr_j\bigg|\\
\end{align*}
\begin{align*}
\qquad\qquad\qquad\qquad\qquad\qquad\qquad&=\dfrac{1}{k}(c^k-c^{-k}\Big(e^{-\frac{\pi}{u_j}}\Big)^k)\sin^2(u_j\log c)\\
&\quad+\dfrac{u_j}{k}\bigg|\int_{c^{-1}e^{-\frac{\pi}{u_j}}}^cr_j^{k-1}\sin2\tau_jdr_j\bigg|\\
&\leq O(u_j^2)c^k+\dfrac{u_j}{k^2}\bigg(\bigg|\bigg[r_j^k\sin2\tau_j\bigg]_{c^{-1}e^{-\frac{\pi}{u_j}}}^c\bigg|\\
&\quad+\bigg|\int_{c^{-1}e^{-\frac{\pi}{u_j}}}^cr_j^k\cos2\tau_j\cdot\dfrac{2u_j}{r_j}dr_j\bigg|\bigg)\\
&\leq O(u_j^2)c^k+\dfrac{u_j}{k^2}(c^k-c^{-k}\Big(e^{-\frac{\pi}{u_j}}\Big)^k)|\sin(2u_j\log c)|\\
&\quad+\dfrac{2u_j^2}{k^2}\int_{c^{-1}e^{-\frac{\pi}{u_j}}}^cr_j^{k-1}dr_j\\
&=O(u_j^2)c^k+\dfrac{2u_j^2}{k^3}(c^k-c^{-k}\Big(e^{-\frac{\pi}{u_j}}\Big)^k)\\
&=O(u_j^2)c^k.
\end{align*}
(3) can be done similarly.
\end{proof}
Now we come to the computation of $h^{i\overline{i}}$ for $i\leq m$.  Since $h^{i\overline{j}}$ is given by the formula
\[
h^{i\overline{j}}=\int_X\psi_i\overline{\psi_j}\lambda^{-2}dv,
\]
we know that
\[
h^{i\overline{i}}=\int_X|\psi_i|^2\lambda^{-2}dv=\sum_{l=1}^m\int_{\Omega_c^l}|\varphi_i|^2\lambda^{-2}dv+\int_{R_c}|\psi_i|^2\lambda^{-2}dv.
\]
For $l=i$, we have
\begin{align*}
\int_{\Omega_c^i}|\varphi_i|^2\lambda^{-2}dv&=\int_0^{2\pi}\int_{c^{-1}e^{-\frac{\pi}{u_i}}}^c\dfrac{|t_i|^2}{\pi^2r_i^4}|q_i+\beta_i|^2\cdot\dfrac{2r_i^2}{u_i^2}\sin^2\tau_ir_idr_id\theta_i\\
&=\dfrac{2|t_i|^2}{\pi^2u_i^2}\int_0^{2\pi}\int_{c^{-1}e^{-\frac{\pi}{u_i}}}^c\dfrac{1}{r_i}\sin^2\tau_i|q_i+1|^2dr_id\theta_i.
\end{align*}
Now using Wolpert's estimate on the length of the short geodesic $\gamma_i$ (Example 4.3, \cite{bib11}), we obtain $l_i=-\dfrac{2\pi^2}{\log|t_i|}(1+O(u_i))$, which implies $-\log|t_i|=\dfrac{\pi}{u_i}+O(1)$, i.e., $-k<-\log|t_i|-\dfrac{\pi}{u_i}<k$ for some $k>0$.  Thus $\mu<\dfrac{e^{-\frac{\pi}{u_i}}}{|t_i|}<\mu^{-1}$, so on the collar $\Omega_c^i$ we know that
\[
|q_i|\leq\sum_{k<0}|\alpha_{ik}(t,s)||t_i|^{-k}(c^{-1}e^{-\frac{\pi}{u_i}})^k+\sum_{k>0}|\alpha_{ik}(t,s)|c^k=O(1).
\]
Note that
\[
|q_i+1|^2=|q_i|^2+|q_i|+|\overline{q_i}|+1\leq O(1)(\sum_{k<0}|\alpha_{ik}(t,s)||t_i|^{-k}r_i^k+\sum_{k>0}|\alpha_{ik}(t,s)|r_i^k)+1.
\]
Hence along with the claim made before, we have
\begin{align*}
\int_{\Omega_c^i}|\varphi_i|^2\lambda^{-2}dv&\leq \dfrac{2|t_i|^2}{\pi^2u_i^2}\int_0^{2\pi}\bigg(O(1)\sum_{k<0}|\alpha_{ik}(t,s)||t_i|^{-k}\int_{c^{-1}e^{-\frac{\pi}{u_i}}}^cr_i^{k-1}\sin^2\tau_idr_i\\
&\quad+O(1)\sum_{k>0}|\alpha_{ik}(t,s)|\int_{c^{-1}e^{-\frac{\pi}{u_i}}}^cr_i^{k-1}\sin^2\tau_idr_i\\
&\quad+\int_{c^{-1}e^{-\frac{\pi}{u_i}}}^c\dfrac{1}{r_i}\sin^2\tau_idr_i\bigg)d\theta_i\\
&=\dfrac{4|t_i|^2}{\pi u_i^2}\bigg(O(u_i^2)\sum_{k<0}|\alpha_{ik}(t,s)||t_i|^{-k}c^{-k}\Big(e^{-\frac{\pi}{u_j}}\Big)^k\\
&\quad+O(u_i^2)\sum_{k>0}|\alpha_{ik}(t,s)|c^k+\dfrac{1}{u_j}(\dfrac{\pi}{2}+O(u_0))\bigg)\\
&=\dfrac{2|t_i|^2}{u_i^3}(O(u_i^3)+1+O(u_0))=\dfrac{2|t_i|^2}{u_i^3}(1+O(u_0)).
\end{align*}
Similarly for $l\neq i$, we have
\begin{align*}
\int_{\Omega_c^l}|\varphi_i|^2\lambda^{-2}dv&=\dfrac{2|t_i|^2}{\pi^2u_l^2}\int_0^{2\pi}\int_{c^{-1}e^{-\frac{\pi}{u_l}}}^c\dfrac{1}{r_l}\sin^2\tau_l|q_i^l|^2dr_ld\theta_l\\
&\leq\dfrac{2|t_i|^2}{\pi^2u_l^2}\int_0^{2\pi}\bigg(O(1)\sum_{k<0}|\alpha_{ik}^l(t,s)||t_l|^{-k}\int_{c^{-1}e^{-\frac{\pi}{u_l}}}^cr_l^{k-1}\sin^2\tau_ldr_l\\
&\quad+O(1)\sum_{k>0}|\alpha_{ik}^l(t,s)|\int_{c^{-1}e^{-\frac{\pi}{u_l}}}^cr_l^{k-1}\sin^2\tau_ldr_l
\end{align*}
\begin{align*}
&=\dfrac{4|t_i|^2}{\pi u_i^2}\bigg(O(u_i^2)\sum_{k<0}|\alpha_{ik}(t,s)||t_i|^{-k}c^{-k}\Big(e^{-\frac{\pi}{u_j}}\Big)^k\ \\
&\quad+O(u_i^2)\sum_{k>0}|\alpha_{ik}(t,s)|c^k\bigg)\\
&=O(|t_i^2|)).
\end{align*}
Then we use Masur's result ((iii) in the proof of Theorem 1, \cite{bib9}) to obtain $\displaystyle\int_{R_c}|\psi_i|^2\lambda^{-2}dv=O(|t_i^2|))$.  Therefore the first part of (i) is done.  The second part can be proved by the same argument as that in the proof of Corollary 4.4.2, which will be shown later.
\end{proof}

Then by Theorem 4.4.2 and Corollary 4.4.1, we can obtain the following expressions for the harmonic Beltrami differentials $B_i=A_i\partial_z\otimes d\overline{z}$:
\begin{lem} (Lemma 4.2, \cite{bib1})
On the genuine collar $\Omega_c^j$ for $c$ small enough, the coefficient functions $A_i$ of the harmonic Beltrami differentials $B_i$ have the form:
\begin{itemize}
\item[(i)] $A_i(z_j)=\dfrac{z_j}{\overline{z_j}}\sin^2\tau_j(\overline{p_i^j(z_j)}+\overline{b_i^j})$ if $i\neq j$;
\item[(ii)] $A_j(z_j)=\dfrac{z_j}{\overline{z_j}}\sin^2\tau_j(\overline{p_j(z_j)}+\overline{b_j})$
\end{itemize}
where
\begin{itemize}
\item[(a)] $\displaystyle p_i^j(z_j)=\sum_{k<0}a_{ik}^j\rho_j^{-k}z_j^k+\sum_{k>0}a_{ik}^jz_j^k$ if $i\neq j$;
\item[(b)] $\displaystyle p_j(z_j)=\sum_{k<0}a_{jk}\rho_j^{-k}z_j^k+\sum_{k>0}a_{jk}z_j^k$
\end{itemize}
such that $\rho_j=e^{-\frac{2\pi^2}{l_j}}$ and the coefficients satisfy the following conditions:
\begin{itemize}
\item[(1)] $\displaystyle \sum_{k<0}|a_{ik}^j|c^{-k}=O(u_j^{-2})$ and $\displaystyle \sum_{k>0}|a_{ik}^j|c^k=O(u_j^{-2})$ if $i\geq m+1$;
\item[(2)] $\displaystyle \sum_{k<0}|a_{ik}^j|c^{-k}=O(u_j^{-2})O(\dfrac{u_i^3}{|t_i|})$ and $\displaystyle \sum_{k>0}|a_{ik}^j|c^k=O(u_j^{-2})O(\dfrac{u_i^3}{|t_i|})$ if $i\leq m$ and $i\neq j$;
\item[(3)] $\displaystyle \sum_{k<0}|a_{jk}|c^{-k}=O(\dfrac{u_j}{|t_j|})$ and $\displaystyle \sum_{k>0}|a_{jk}|c^k=O(\dfrac{u_j}{|t_j|})$;
\item[(4)] $|b_i^j|=O(u_j)$ if $i\geq m+1$;
\item[(5)] $|b_i^j|=O(u_j)O(\dfrac{u_i^3}{|t_i|})$ if $i\leq m$ and $i\neq j$;
\item[(6)] $b_j=-\dfrac{u_j}{\pi\overline{t_j}}+O(\dfrac{u_ju_0}{|t_j|})$.
\end{itemize}
\end{lem}
\begin{proof}
From the duality formula between harmonic Beltrami differentials and the holomorphic quadratic differentials:
\[B_i=\lambda^{-1}\sum_{l=1}^nh_{i\overline{l}}\overline{\psi_l},\]
we know that $\displaystyle A_i=\lambda^{-1}\sum_{l=1}^nh_{i\overline{l}}\overline{\phi_l}$.  Also, note that by Wolpert's estimate of the length of the short geodesic $\gamma_j$ (Example 4.3, \cite{bib11}) we have $l_j=-\dfrac{2\pi^2}{\log|t_j|}(1+O(u_0))$, so we know that there exists a constant $0<\mu<1$ such that $\mu|t_j|<\rho_j<\mu^{-1}|t_j|$.  Then by Theorem 4.4.2 (but we need to replace $c$ by $\mu c$) and Corollary 4.4.1, we can compute the desired orders for each cases.
\end{proof}

In order to estimate the asymptotics of the Ricci metric, we need to estimate the norms of the harmonic Beltrami differentials, and the above expressions for the harmonic Beltrami differentials can help us to do so.
\begin{lem}
Let $\left\|\cdot\right\|_k$ be the $C^k$ norm defined before.  We have
\begin{itemize}
\item[(i)] $\left\|A_i\right\|_{0,\Omega_c^i}=O(\dfrac{u_i}{|t_i|})$ and $\left\|A_i\right\|_{0,X-\Omega_c^i}=O(\dfrac{u_i^3}{|t_i|})$ if $i\leq m$;
\item[(ii)] $\left\|A_i\right\|_0=O(1)$ if $i\geq m+1$;
\item[(iii)] $\left\|f_{i\overline{i}}\right\|_{0,\Omega_c^i}=O(\dfrac{u_i^2}{|t_i|^2})$ and $\left\|f_{i\overline{i}}\right\|_{0,X-\Omega_c^i}=O(\dfrac{u_i^6}{|t_i|^2})$ if $i\leq m$;
\item[(iv)] $\left\|f_{i\overline{j}}\right\|_{0,\Omega_c^i}=O(\dfrac{u_iu_j^3}{|t_it_j|})$, $\left\|f_{i\overline{j}}\right\|_{0,\Omega_c^j}=O(\dfrac{u_i^3u_j}{|t_it_j|})$ and $\left\|f_{i\overline{j}}\right\|_{0,X-(\Omega_c^i\cup\Omega_c^j)}=O(\dfrac{u_i^3u_j^3}{|t_it_j|})$ if $i, j\leq m$ and $i\neq j$;
\item[(v)] $\left\|f_{i\overline{j}}\right\|_{0,\Omega_c^i}=O(\dfrac{u_i}{|t_i|})$ and $\left\|f_{i\overline{j}}\right\|_{0,X-\Omega_c^i}=O(\dfrac{u_i^3}{|t_i|})$ if $i\leq m$ and $j\geq m+1$;
\item[(vi)] $\left\|f_{i\overline{j}}\right\|_0=O(1)$ if $i, j\geq m+1$;
\item[(vii)] $|f_{i\overline{i}}|_{L^1}=O(\dfrac{u_i^3}{|t_i|^2})$ if $i\leq m$;
\item[(viii)] $|f_{i\overline{j}}|_{L^1}=O(\dfrac{u_i^3u_j^3}{|t_it_j|})$ if $i, j\leq m$ and $i\neq j$;
\item[(ix)] $|f_{i\overline{j}}|_{L^1}=O(\dfrac{u_i^3}{|t_i|})$ if $i\leq m$ and $j\geq m+1$;
\item[(x)] $|f_{i\overline{j}}|_{L^1}=O(1)$ if $i, j\geq m+1$.
\end{itemize}
\end{lem}
\begin{rem}
Form the lemma, we can also know that $\left\|f_{i\overline{j}}\right\|_0=O(1)\left\|A_i\right\|_0\left\|A_j\right\|_0$ for any $i, j$.
\end{rem}
\begin{proof}
Choose $c$ sufficiently small such that for all $j\leq m$, $\tan(u_j\log c)<-10u_j$.  Then by some calculus argument we can show that when $1\leq p\leq 10$, on the collar $\Omega_c^j$,
\[|r_j^k\sin^p\tau_j|\leq c^k|\log c|^pu_j^p\]
if $k\geq 1$;
\[|r_j^k\sin^p\tau_j|\leq c^{-k}|\log c|^p\rho_j^ku_j^p\]
if $k\leq -1$.
The first two claims can be proved similarly, so let's work out (i) only.  Notice that on $\Omega_c^i$ we have
\begin{align*}
|A_i|&=\bigg|\dfrac{z_i}{\overline{z_i}}\bigg||\sin^2\tau_i(\overline{p_i}+\overline{b_i})|\\
&\leq \sum_{k<0}|a_{ik}|\rho_i^{-k}r_i^k\sin^2\tau_i+\sum_{k>0}|a_{ik}|r_i^k\sin^2\tau_i+|b_i|\\
&\leq (\log c)^2u_i^2(\sum_{k<0}|a_{ik}|c^{-k}+\sum_{k>0}|a_{ik}|c^k)+|b_i|\\
&=O(u_i^2)O(\dfrac{u_i}{|t_i|})+O(u_i^2)O(\dfrac{u_i}{|t_i|})+O(\dfrac{u_i}{|t_i|})=O(\dfrac{u_i}{|t_i|}).
\end{align*}
Thus $\left\|A_i\right\|_{0,\Omega_c^i}=O(\dfrac{u_i}{|t_i|})$.  In a similar way we can also show that $|A_i|=O(\dfrac{u_i^3}{|t_i|})$ on $\Omega_c^j$ with $j\neq i$.  Then by Masur's result\cite{bib9}, the duality formula, Theorem 4.4.2 and Corollary 4.4.1, we have $|A_i|=O(\dfrac{u_i^3}{|t_i|})$ on $R_c$, so we obtain $\left\|A_i\right\|_{0,X-\Omega_c^i}=O(\dfrac{u_i^3}{|t_i|})$.

(iii)-(vi) can be proved using (i)-(ii) and the fact that $f_{i\overline{j}}=A_i\overline{A_j}$.  For (vii), on $\Omega_c^i$ we know that
\begin{align*}
|f_{i\overline{i}}|&=\sin^4\tau_i|p_i+b_i|^2\leq\sin^4\tau_i|p_i|^2+\sin^4\tau_i|p_i\overline{b_i}|+\sin^4\tau_i|\overline{p_i}b_i|+\sin^4\tau_i|b_i|^2\\
&\leq O(\dfrac{u_i^6}{|t_i|^2})+O(\dfrac{u_i^2}{|t_i|^2})\sin^4\tau_i.
\end{align*}
Hence
\begin{align*}
|f_{i\overline{i}}|_{L^1(\Omega_c^i)}&\leq\int_{\Omega_c^i}\bigg(O(\dfrac{u_i^6}{|t_i|^2})+O(\dfrac{u_i^2}{|t_i|^2})\sin^4\tau_i\bigg)dv\\
&\leq \int_X O(\dfrac{u_i^6}{|t_i|^2})dv+O(\dfrac{u_i^2}{|t_i|^2})\int_{\Omega_c^i}\sin^4\tau_idv\\
&=O(\dfrac{u_i^6}{|t_i|^2})+O(\dfrac{u_i^3}{|t_i|^2})=O(\dfrac{u_i^3}{|t_i|^2}).
\end{align*}
Notice that by (iii) we have $|f_{i\overline{i}}|_{L^1(X-\Omega_c^i)}=O(\dfrac{u_i^6}{|t_i|^2})$, which implies that $|f_{i\overline{i}}|_{L^1}=O(\dfrac{u_i^3}{|t_i|^2})$.  (viii) and (ix) can be proved in the same way.  Lastly, we observe that (x) follows form (vi) and the Gauss-Bonnet theorem, and thus the proof is completed.

\end{proof}

We also need to estimate the functions $e_{i\overline{j}}$ for the estimate of the asymptotics of the Ricci metric, and thus we need to construct some approximation functions of them on the collars.

First we fix a constant $c_1<c$ and define the cut-off function $\eta\in C^{\infty}(\mathbb{R},[0,1])$ by
\[
\left\{
\begin{array}{ll}
\eta(x)=1 & \text{if } x\leq\log c_1;\\
\eta(x)=0 & \text{if } x\geq \log c;\\
0<\eta(x)<1 & \text{if } \log c_1<x<\log c.
\end{array}\right.
\]
Note that the derivatives of $\eta$ are bounded by some constants depending only on $c$ and $c_1$.

Let $\tilde{e}_{i\overline{j}}(z)$ be the function on $X$ (where $z=z_i$ on the collar $\Omega^i_c$) defined by:
\begin{itemize}
\item[(i)] if $i, j\leq m$ and $i=j$,
\[
\tilde{e}_{i\overline{j}}(z)=
\left\{
\begin{array}{ll}
\frac{1}{2}\sin^2\tau_i(|b_i|^2) & \text{when } z\in\Omega^i_{c_1};\\
\frac{1}{2}\sin^2\tau_i(|b_i|^2)\eta(\log r_i) & \text{when } z\in\Omega^i_c \text{ and } c_1<r_i<c;\\
\frac{1}{2}\sin^2\tau_i(|b_i|^2)\eta(\log\rho_i-\log r_i) & \text{when } z\in\Omega^i_c \text{ and } c^{-1}\rho_i<r_i<c^{-1}_1\rho_i;\\
0 & \text{when } z\in X-\Omega^i_{c};
\end{array}\right.
\]
\item[(ii)] if $i, j\leq m$ and $i\neq j$,
\[
\tilde{e}_{i\overline{j}}(z)=
\left\{
\begin{array}{ll}
\frac{1}{2}\sin^2\tau_i(\overline{b_i}b^i_j) & \text{when } z\in\Omega^i_{c_1};\\
\frac{1}{2}\sin^2\tau_i(\overline{b_i}b^i_j)\eta(\log r_i) & \text{when } z\in\Omega^i_c \text{ and } c_1<r_i<c;\\
\frac{1}{2}\sin^2\tau_i(\overline{b_i}b^i_j)\eta(\log\rho_i-\log r_i) & \text{when } z\in\Omega^i_c \text{ and } c^{-1}\rho_i<r_i<c^{-1}_1\rho_i;\\
\frac{1}{2}\sin^2\tau_j(\overline{b^j_i}b_j) & \text{when } z\in\Omega^j_{c_1};\\
\frac{1}{2}\sin^2\tau_j(\overline{b^j_i}b_j)\eta(\log r_j) & \text{when } z\in\Omega^j_c \text{ and } c_1<r_j<c;\\
\frac{1}{2}\sin^2\tau_j(\overline{b^j_i}b_j)\eta(\log\rho_j-\log r_j) & \text{when } z\in\Omega^j_c \text{ and } c^{-1}\rho_j<r_j<c^{-1}_1\rho_j;\\
0 & \text{when } z\in X-(\Omega^i_{c}\cup\Omega^j_{c});
\end{array}\right.
\]
\item[(iii)] if $i\leq m$ and $j\geq m+1$,
\[
\tilde{e}_{i\overline{j}}(z)=
\left\{
\begin{array}{ll}
\frac{1}{2}\sin^2\tau_i(\overline{b_i}b^i_j) & \text{when } z\in\Omega^i_{c_1};\\
\frac{1}{2}\sin^2\tau_i(\overline{b_i}b^i_j)\eta(\log r_i) & \text{when } z\in\Omega^i_c \text{ and } c_1<r_i<c;\\
\frac{1}{2}\sin^2\tau_i(\overline{b_i}b^i_j)\eta(\log\rho_i-\log r_i) & \text{when } z\in\Omega^i_c \text{ and } c^{-1}\rho_i<r_i<c^{-1}_1\rho_i;\\
0 & \text{when } z\in X-\Omega^i_{c};
\end{array}\right.
\]
\end{itemize}

Moveover, let $\tilde{f}_{i\overline{j}}=(\Box+1)\tilde{e}_{i\overline{j}}$.  Hence by direct computation we have

\begin{itemize}
\item[(i)] if $i, j\leq m$ and $i=j$,
\[\tilde{f}_{i\overline{j}}(z)=\sin^4\tau_i(|b_i|^2)\]
when $z\in\Omega^i_{c_1}$;
\begin{align*}
\tilde{f}_{i\overline{j}}(z)&=\sin^4\tau_i(|b_i|^2)\eta(\log r_i)-\frac{1}{2u_i}\sin^2\tau_i\sin2\tau_i(|b_i|^2)\eta'(\log r_i)\\
&-\frac{1}{4u_i^2}\sin^4\tau_i(|b_i|^2)\eta''(\log r_i)
\end{align*}
when $z\in\Omega^i_c$ and $c_1<r_i<c$;
\begin{align*}
\tilde{f}_{i\overline{j}}(z)&=\sin^4\tau_i(|b_i|^2)\eta(\log\rho_i-\log r_i)+\frac{1}{2u_i}\sin^2\tau_i\sin2\tau_i(|b_i|^2)\eta'(\log\rho_i-\log r_i)\\
&-\frac{1}{4u_i^2}\sin^4\tau_i(|b_i|^2)\eta''(\log\rho_i-\log r_i)
\end{align*}
when $z\in\Omega^i_c$ and $c^{-1}\rho_i<r_i<c^{-1}_1\rho_i$;
\[\tilde{f}_{i\overline{j}}(z)=0\]
when $z\in X-\Omega^i_{c}$.
\item[(ii)] if $i, j\leq m$ and $i\neq j$,
\[\tilde{f}_{i\overline{j}}(z)=\sin^4\tau_i(\overline{b_i}b^i_j)\]
when $z\in\Omega^i_{c_1}$;
\begin{align*}
\tilde{f}_{i\overline{j}}(z)&=\sin^4\tau_i(\overline{b_i}b^i_j)\eta(\log r_i)-\frac{1}{2u_i}\sin^2\tau_i\sin2\tau_i(\overline{b_i}b^i_j)\eta'(\log r_i)\\
&-\frac{1}{4u_i^2}\sin^4\tau_i(\overline{b_i}b^i_j)\eta''(\log r_i)
\end{align*}
when $z\in\Omega^i_c$ and $c_1<r_i<c$;
\begin{align*}
\tilde{f}_{i\overline{j}}(z)&=\sin^4\tau_i(\overline{b_i}b^i_j)\eta(\log\rho_i-\log r_i)+\frac{1}{2u_i}\sin^2\tau_i\sin2\tau_i(\overline{b_i}b^i_j)\eta'(\log\rho_i-\log r_i)\\
&-\frac{1}{4u_i^2}\sin^4\tau_i(\overline{b_i}b^i_j)\eta''(\log\rho_i-\log r_i)
\end{align*}
when $z\in\Omega^i_c$ and $c^{-1}\rho_i<r_i<c^{-1}_1\rho_i$;
\[\tilde{f}_{i\overline{j}}(z)=\sin^4\tau_j(\overline{b^j_i}b_j)\]
when $z\in\Omega^j_{c_1}$;
\begin{align*}
\tilde{f}_{i\overline{j}}(z)&=\sin^4\tau_j(\overline{b^j_i}b_j)\eta(\log r_j)-\frac{1}{2u_j}\sin^2\tau_j\sin2\tau_j(\overline{b^j_i}b_j)\eta'(\log r_j)\\
&-\frac{1}{4u_j^2}\sin^4\tau_j(\overline{b^j_i}b_j)\eta''(\log r_j)
\end{align*}
when $z\in\Omega^j_c$ and $c_1<r_j<c$;
\begin{align*}
\tilde{f}_{i\overline{j}}(z)&=\sin^4\tau_j(\overline{b^j_i}b_j)\eta(\log\rho_j-\log r_j)+\frac{1}{2u_j}\sin^2\tau_j\sin2\tau_j(\overline{b^j_i}b_j)\eta'(\log\rho_j-\log r_j)\\
&-\frac{1}{4u_j^2}\sin^4\tau_j(\overline{b^j_i}b_j)\eta''(\log\rho_j-\log r_j)
\end{align*}
when $z\in\Omega^j_c$ and $c^{-1}\rho_j<r_j<c^{-1}_1\rho_j$;
\[\tilde{f}_{i\overline{j}}(z)=0\]
when $z\in X-(\Omega^i_{c}\cup\Omega^j_{c})$.
\item[(iii)] if $i\leq m$ and $j\geq m+1$,
\[\tilde{f}_{i\overline{j}}(z)=\sin^4\tau_i(\overline{b_i}b^i_j)\]
when $z\in\Omega^i_{c_1}$;
\begin{align*}
\tilde{f}_{i\overline{j}}(z)&=\sin^4\tau_i(\overline{b_i}b^i_j)\eta(\log r_i)-\frac{1}{2u_i}\sin^2\tau_i\sin2\tau_i(\overline{b_i}b^i_j)\eta'(\log r_i)\\
&-\frac{1}{4u_i^2}\sin^4\tau_i(\overline{b_i}b^i_j)\eta''(\log r_i)
\end{align*}
when $z\in\Omega^i_c$ and $c_1<r_i<c$;
\begin{align*}
\tilde{f}_{i\overline{j}}(z)&=\sin^4\tau_i(\overline{b_i}b^i_j)\eta(\log\rho_i-\log r_i)+\frac{1}{2u_i}\sin^2\tau_i\sin2\tau_i(\overline{b_i}b^i_j)\eta'(\log\rho_i-\log r_i)\\
&-\frac{1}{4u_i^2}\sin^4\tau_i(\overline{b_i}b^i_j)\eta''(\log\rho_i-\log r_i)
\end{align*}
when $z\in\Omega^i_c$ and $c^{-1}\rho_i<r_i<c^{-1}_1\rho_i$;
\[\tilde{f}_{i\overline{j}}(z)=0\]
when $z\in X-\Omega^i_{c}$.
\end{itemize}

By construction we know that the supports of these approximation functions are contained in the corresponding collars.  Now we can estimate the functions $e_{i\overline{j}}$.

\begin{lem} We have the following estimates for $e_{i\overline{j}}$:
\begin{itemize}
\item[(i)] $e_{i\overline{i}}=\tilde{e}_{i\overline{i}}+O(\dfrac{u_i^4}{|t_i|^2})$ if $i\leq m$;
\item[(ii)] $e_{i\overline{j}}=\tilde{e}_{i\overline{j}}+O(\dfrac{u_i^3u_j^3}{|t_it_j|})$ if $i, j\leq m$ and $i\neq j$;
\item[(iii)] $e_{i\overline{j}}=\tilde{e}_{i\overline{j}}+O(\dfrac{u_i^3}{|t_i|})$ if $i\leq m$ and $j\geq m+1$;
\item[(iv)] $\left\|e_{i\overline{j}}\right\|_0=O(1)$ if $i, j\geq m+1$;
\item[(v)] $|\tilde{e}_{i\overline{i}}|_{L^1}=O(\dfrac{u_i^3}{|t_i|^2})$ if $i\leq m$;
\item[(vi)] $|\tilde{e}_{i\overline{j}}|_{L^1}=O(\dfrac{u_i^3u_j^3}{|t_it_j|})$ if $i, j\leq m$ and $i\neq j$;
\item[(vii)] $|\tilde{e}_{i\overline{j}}|_{L^1}=O(\dfrac{u_i^3}{|t_i|})$ if $i\leq m$ and $j\geq m+1$.
\end{itemize}
\end{lem}
\begin{proof}
For (i), since by the maximum principle we know that
\[
\left\|e_{i\overline{i}}-\tilde{e_{i\overline{i}}}\right\|_0\leq\left\|f_{i\overline{i}}-\tilde{f}_{i\overline{i}}\right\|_0,
\]
it suffices to show that $\left\|f_{i\overline{i}}-\tilde{f}_{i\overline{i}}\right\|_0=O(\dfrac{u_i^4}{|t_i|^2})$.  Using the fact that $\tilde{f}_{i\overline{i}}=0$ on $X-\Omega_c^i$, by Lemma 4.4.2 we have \[\left\|f_{i\overline{i}}-\tilde{f}_{i\overline{i}}\right\|_0=\left\|f_{i\overline{i}}\right\|_0=O(\dfrac{u_i^6}{|t_i|^2}).\]  As on $\Omega_{c_1}^i$ we have
\begin{align*}
|f_{i\overline{i}}-\tilde{f}_{i\overline{i}}|&\leq|\sin^4\tau_i(\overline{p_i}b_i)|+|\sin^4\tau_i(\overline{b_i}p_i)|+|\sin^4\tau_i(\overline{p_i}p_i)|\\
&=O(\dfrac{u_i^5}{|t_i|})O(\dfrac{u_i}{|t_i|})+O(\dfrac{u_i^5}{|t_i|})O(\dfrac{u_i}{|t_i|})+O(\dfrac{u_i^3}{|t_i|})O(\dfrac{u_i^3}{|t_i|})=O(\dfrac{u_i^6}{|t_i|^2}),
\end{align*}
we know that $\left\|f_{i\overline{i}}-\tilde{f}_{i\overline{i}}\right\|_{0, \Omega_{c_1}^i}=O(\dfrac{u_i^6}{|t_i|^2})$.  Now on $\Omega_c^i-\Omega_{c_1}^i$ with $c_1<r_i<c$ we have
\[|\sin\tau_i|\leq|\tau_i|\leq u_i|\log c_1|=O(u_i),\]
so we can conclude that
\begin{align*}
|f_{i\overline{i}}-\tilde{f}_{i\overline{i}}|&\leq(1-\eta)|b_i|^2\sin^4\tau_i+|\sin^4\tau_i(\overline{p_i}b_i)|+|\sin^4\tau_i(\overline{b_i}p_i)|+|\sin^4\tau_i(\overline{p_i}p_i)|\\
&\quad+\frac{|b_i|^2|\eta'|}{2u_i}\sin^2\tau_i|\sin2\tau_i|+\frac{|b_i|^2|\eta''|}{4u_i^2}\sin^4\tau_i\\
&=O(\dfrac{u_i^2}{|t_i|^2})O(u_i^4)+O(\dfrac{u_i^6}{|t_i|^2})+O(\dfrac{u_i^2}{|t_i|^2})O(\dfrac{1}{u_i})O(u_i^3)+O(\dfrac{u_i^2}{|t_i|^2})O(\dfrac{1}{u_i^2})O(u_i^4)\\
&=O(\dfrac{u_i^4}{|t_i|^2}).
\end{align*}
Similarly we also have $|f_{i\overline{i}}-\tilde{f}_{i\overline{i}}|\leq O(\dfrac{u_i^4}{|t_i|^2})$ on $\Omega^i_c-\Omega^i_{c_1}$ with $c^{-1}\rho_i<r_i<c^{-1}_1\rho_i$, which implies that $\left\|f_{i\overline{i}}-\tilde{f}_{i\overline{i}}\right\|_{0, \Omega_c^i-\Omega_{c_1}^i}=O(\dfrac{u_i^4}{|t_i|^2})$ and thus
$\left\|f_{i\overline{i}}-\tilde{f}_{i\overline{i}}\right\|_0=O(\dfrac{u_i^4}{|t_i|^2})$.  (ii) and (iii) follows nearly the same argument.

(iv) is easy since by the maximum principle and Lemma 4.4.2 we have $\left\|e_{i\overline{j}}\right\|_0\leq\left\|f_{i\overline{j}}\right\|_0=O(1)$.

The last three claims can be proved by direct computation.  Let's work out (v) explicitly as an example:
\begin{align*}
|\tilde{e}_{i\overline{i}}|_{L^1}&=\int_X|\tilde{e}_{i\overline{i}}|dv\\
&=\int_{\Omega_c^i}|\tilde{e}_{i\overline{i}}|dv\\
&\leq\int_{\Omega_c^i}\dfrac{1}{2}\sin^2\tau_i(|b_i|^2)dv\\
&=\int_0^{2\pi}\int_{c^{-1}e^{-\frac{\pi}{u_i}}}^c\dfrac{1}{2}\sin^2\tau_i(|b_i|^2)\cdot\dfrac{u_i^2}{2r_i^2}\csc^2\tau_i\cdot r_idr_id\theta_i\\
&=\dfrac{u_i^2|b_i|^2}{4}\int_0^{2\pi}\int_{c^{-1}e^{-\frac{\pi}{u_i}}}^c\dfrac{1}{r_i}dr_id\theta_i\\
&=O(\dfrac{u_i^4}{|t_i|^2})O(u_i^{-1})=O(\dfrac{u_i^3}{|t_i|^2}).
\end{align*}
\end{proof}

The following $L^2$ norm estimates will be used in the estimate of the Ricci metric.
\begin{lem}
Let $f\in C^{\infty}(X,\mathbb{C})$.  Then
\[
\int_X|(\Box+1)^{-1}f|^2dv\leq\int_X(\Box+1)^{-1}f\cdot\overline{f}dv\leq\int_X|f|^2dv.
\]
\end{lem}
\begin{proof}
We just need to note that the eigenvalues of $(\Box+1)^{-1}$ are all less than or equal to $1$.  Then the result would follow from the spectral decomposition of $(\Box+1)^{-1}$.
\end{proof}

Let's estimate the asymptotics of the Ricci metric in the following corollary (Corollary 4.2, \cite{bib1}).

\begin{cor}
Let $(t, s)$ be the pinching coordinates.  Then we have
\begin{itemize}
\item[(i)] $\tau_{i\overline{i}}=\dfrac{3}{4\pi^2}\dfrac{u_i^2}{|t_i|^2}(1+O(u_0))$ and $\tau^{i\overline{i}}=\dfrac{4\pi^2}{3}\dfrac{|t_i|^2}{u_i^2}(1+O(u_0))$ if $\leq i\leq m$;
\item[(ii)] $\tau_{i\overline{j}}=O(\dfrac{u_i^2u_j^2}{|t_it_j|}(u_i+u_j))$ and $\tau^{i\overline{j}}=O(|t_it_j|)$ if $\leq i,j\leq m$ and $i\neq j$;
\item[(iii)] $\tau_{i\overline{j}}=O(\dfrac{u_i^2}{|t_i|})$ and $\tau^{i\overline{j}}=O(|t_i|)$ if $i\leq m$ or $j\geq m+1$;
\item[(iv)] $\tau_{i\overline{j}}=O(1)$ and $\tau^{i\overline{j}}=O(1)$ if $i, j\geq m+1$.
\end{itemize}
\end{cor}
\begin{rem}
In fact Trapani\cite{bib10} has also estimated the Ricci metric, but the estimates in this corollary are more refined.
\end{rem}
\begin{proof}
Let's prove the first part.  In the following, all universal constants will be denoted as $C_0$.  Recall that $\tau_{i\overline{j}}=h^{\alpha\overline{\beta}}R_{i\overline{j}\alpha\overline{\beta}}$.

For (i), notice that if $\alpha\neq \beta$ or $\alpha=\beta\geq m+1$, then by Lemma 4.4.2 and Corollary 4.4.1 we have $|h^{\alpha\overline{\beta}}|\left\|A_{\alpha}\right\|_0\left\|A_{\beta}\right\|_0=O(1)$. Thus by the maximum principle and Lemma 4.4.4 we have
\begin{align*}
|R_{i\overline{i}\alpha\overline{\beta}}|&\leq\bigg|\int_{X}f_{i\overline{i}}e_{\alpha\overline{\beta}}dv\bigg|+\bigg|\int_{X}f_{i\overline{\beta}}e_{\alpha\overline{i}}dv\bigg|\\
&\leq\left\|e_{\alpha\overline{\beta}}\right\|_0|f_{i\overline{i}}|_{L^1}+\bigg(\int_X|e_{\alpha\overline{i}}|^2dv\int_X|f_{i\overline{\beta}}|^2dv\bigg)^{\frac{1}{2}}\\
&\leq\left\|f_{\alpha\overline{\beta}}\right\|_0|f_{i\overline{i}}|_{L^1}+\bigg(\int_X|f_{\alpha\overline{i}}|^2dv\int_X|f_{i\overline{\beta}}|^2dv\bigg)^{\frac{1}{2}}\\
&=\left\|f_{\alpha\overline{\beta}}\right\|_0|f_{i\overline{i}}|_{L^1}+\bigg(\int_{X}f_{\alpha\overline{\alpha}}f_{i\overline{i}}dv\int_{X}f_{i\overline{i}}f_{\beta\overline{\beta}}\bigg)^{\frac{1}{2}}dv\\
&\leq2\left\|A_{\alpha}\right\|_0\left\|A_{\beta}\right\|_0|f_{i\overline{i}}|_{L^1}.
\end{align*}
Hence in this case we have
\[|h^{\alpha\overline{\beta}}R_{i\overline{i}\alpha\overline{\beta}}|=O(1)|f_{i\overline{i}}|_{L^1}=O(\dfrac{u_i^3}{|t_i|^2}).\]

If $\alpha=\beta\leq m$ and $\alpha\neq i$, then we consider
\[
|R_{i\overline{i}\alpha\overline{\alpha}}|\leq\bigg|\int_{X}e_{i\overline{i}}f_{\alpha\overline{\alpha}}dv\bigg|+\bigg|\int_{X}e_{i\overline{\alpha}}f_{\alpha\overline{i}}dv\bigg|.
\]
Note that we have
\begin{align*}
\bigg|\int_{X}e_{i\overline{\alpha}}f_{\alpha\overline{i}}dv\bigg|&\leq\left\|e_{i\overline{\alpha}}\right\|_0|f_{\alpha\overline{i}}|_{L^1}\leq\left\|f_{i\overline{\alpha}}\right\|_0|f_{\alpha\overline{i}}|_{L^1}\\
&=O(\dfrac{u_iu_{\alpha}}{|t_it_{\alpha}|})O(\dfrac{u_i^3u_{\alpha}^3}{|t_it_{\alpha}|})=O(\dfrac{u_i^3u_{\alpha}^3}{|t_it_{\alpha}|^2})
\end{align*}
and by Lemma 4.4.3 we also have
\begin{align*}
\bigg|\int_{X}e_{i\overline{i}}f_{\alpha\overline{\alpha}}dv\bigg|&\leq\bigg|\int_{X}\tilde{e}_{i\overline{i}}f_{\alpha\overline{\alpha}}dv\bigg|+\bigg|\int_{X}(e_{i\overline{i}}-\tilde{e}_{i\overline{i}})f_{\alpha\overline{\alpha}}dv\bigg|\\
&\leq\bigg|\int_{\Omega_c^i}\tilde{e}_{i\overline{i}}f_{\alpha\overline{\alpha}}dv\bigg|+\left\|e_{i\overline{i}}-\tilde{e}_{i\overline{i}}\right\|_0|f_{\alpha\overline{\alpha}}|_{L^1}\\
&\leq|f_{\alpha\overline{\alpha}}|_{0,\Omega_c^i}|\tilde{e}_{i\overline{i}}|_{L^1}+\left\|e_{i\overline{i}}-\tilde{e}_{i\overline{i}}\right\|_0|f_{\alpha\overline{\alpha}}|_{L^1}\\
&=O(\dfrac{u_{\alpha}^6}{|t_{\alpha}|^2})O(\dfrac{u_i^3}{|t_i|^2})+O(\dfrac{u_i^4}{|t_i|^2})O(\dfrac{u_{\alpha}^3}{|t_{\alpha}|^2})=O(\dfrac{u_i^3u_{\alpha}^3}{|t_it_{\alpha}|^2}).
\end{align*}
Therefore we obtain $|h^{\alpha\overline{\alpha}}R_{i\overline{i}\alpha\overline{\alpha}}|=O(\dfrac{|t_{\alpha}|^2}{u_{\alpha}^3})O(\dfrac{u_i^3u_{\alpha}^3}{|t_it_{\alpha}|^2})=O(\dfrac{u_i^3}{|t_i|^2})$.

Now we have to consider $h^{i\overline{i}}R_{i\overline{i}i\overline{i}}$.  Notice that \[R_{i\overline{i}i\overline{i}}=2\int_X e_{i\overline{i}}f_{i\overline{i}}dv=\int_X \tilde{e}_{i\overline{i}}\tilde{f}_{i\overline{i}}dv+\int_X \tilde{e}_{i\overline{i}}(f_{i\overline{i}}-\tilde{f}_{i\overline{i}})dv+\int_X (e_{i\overline{i}}-\tilde{e}_{i\overline{i}})\tilde{f}_{i\overline{i}}.\]
First we have
\[
\bigg|\int_X \tilde{e}_{i\overline{i}}(f_{i\overline{i}}-\tilde{f}_{i\overline{i}})dv\bigg|\leq\left\|f_{i\overline{i}}-\tilde{f}_{i\overline{i}}\right\|_0|\tilde{e}_{i\overline{i}}|_{L^1}=O(\dfrac{u_i^4}{|t_i|^2})O(\dfrac{u_i^3}{|t_i|^2})=O(\dfrac{u_i^7}{|t_i|^4})
\]
and
\[
\bigg|\int_X f_{i\overline{i}}(e_{i\overline{i}}-\tilde{e}_{i\overline{i}})dv\bigg|\leq\left\|e_{i\overline{i}}-\tilde{e}_{i\overline{i}}\right\|_0|f_{i\overline{i}}|_{L^1}=O(\dfrac{u_i^4}{|t_i|^2})O(\dfrac{u_i^3}{|t_i|^2})=O(\dfrac{u_i^7}{|t_i|^4}).
\]
In addition, we have $\left\|\tilde{e}_{i\overline{i}}\right\|_{0,\Omega_c^i-\Omega_{c_1}^i}=O(\dfrac{u_i^4}{|t_i|^2})$, $\left\|\tilde{f}_{i\overline{i}}\right\|_{0,\Omega_c^i-\Omega_{c_1}^i}=O(\dfrac{u_i^4}{|t_i|^2})$ and
\begin{align*}
\int_{\Omega_{c_1}^i} \tilde{e}_{i\overline{i}}\tilde{f}_{i\overline{i}}dv&=\dfrac{|b_i|^4u_i^2}{4}\int_0^{2\pi}\int_{c_1^{-1}e^{-\frac{\pi}{u_i}}}^{c_1}\dfrac{1}{r_i}\sin^4\tau_idr_id\theta_i\\
&=\dfrac{|b_i|^4u_i^2}{4}\int_0^{2\pi}\dfrac{3\pi}{8u_i}(1+O(u_0))d\theta_i\\
&=\dfrac{3\pi^2}{16}|b_i|^4u_i(1+O(u_0)),
\end{align*}
so
\[
\int_X \tilde{e}_{i\overline{i}}\tilde{f}_{i\overline{i}}dv=\int_{\Omega_{c_1}^i} \tilde{e}_{i\overline{i}}\tilde{f}_{i\overline{i}}dv+\int_{\Omega_c^i-\Omega_{c_1}^i} \tilde{e}_{i\overline{i}}\tilde{f}_{i\overline{i}}dv=\dfrac{3\pi^2}{16}|b_i|^4u_i(1+O(u_0))+O(\dfrac{u_i^8}{|t_i|^4}).
\]
Hence we know that
\begin{align*}
h^{i\overline{i}}R_{i\overline{i}i\overline{i}}&=2\dfrac{|t_i|^2}{u_i^3}(1+O(u_0))(\dfrac{3\pi^2}{8}|b_i|^4u_i(1+O(u_0))+O(\dfrac{u_i^7}{|t_i|^4}))\\
&=\dfrac{3\pi^2}{4}\dfrac{|t_i|^2}{u_i^2}(\dfrac{u_i^4}{\pi^4|t_i|^4})(1+O(u_0))+O(\dfrac{u_i^4}{|t_i|^2})(1+O(u_0))\\
&=\dfrac{3}{4\pi^2}\dfrac{u_i^2}{|t_i|^2}(1+O(u_0)).
\end{align*}
Then combining all three cases we can obtain the result.

For (ii), if $\alpha\neq \beta$ and $\alpha\neq i$, or $\alpha=\beta\geq m+1$, then we have
\[
|R_{i\overline{j}\alpha\overline{\beta}}|\leq\bigg|\int_{X}f_{i\overline{j}}e_{\alpha\overline{\beta}}dv\bigg|+\bigg|\int_{X}f_{i\overline{\beta}}e_{\alpha\overline{j}}dv\bigg|.
\]
Let's estimate the first term:
\begin{align*}
\bigg|\int_{X}f_{i\overline{j}}e_{\alpha\overline{\beta}}dv\bigg|&\leq\left\|e_{\alpha\overline{\beta}}\right\|_0|f_{i\overline{j}}|_{L^1}\\
&\leq\left\|f_{\alpha\overline{\beta}}\right\|_0|f_{i\overline{j}}|_{L^1}=\left\|A_{\alpha}\right\|_0\left\|A_{\beta}\right\|_0O(\dfrac{u_i^3u_j^3}{|t_it_j|}).
\end{align*}
The second term can be estimated in this way:
\begin{align*}
\bigg|\int_{X}f_{i\overline{\beta}}e_{\alpha\overline{j}}dv\bigg|&\leq\bigg|\int_{X}f_{i\overline{\beta}}\tilde{e}_{\alpha\overline{j}}dv\bigg|+\bigg|\int_{X}f_{i\overline{\beta}}(e_{\alpha\overline{j}}-\tilde{e}_{\alpha\overline{j}})dv\bigg|\\
&=\bigg|\int_{X-\Omega_c^i}f_{i\overline{\beta}}\tilde{e}_{\alpha\overline{j}}dv\bigg|+\bigg|\int_{X}f_{i\overline{\beta}}(e_{\alpha\overline{j}}-\tilde{e}_{\alpha\overline{j}})dv\bigg|\\
&\leq\left\|f_{i\overline{\beta}}\right\|_{0,X-\Omega_c^i}|\tilde{e}_{\alpha\overline{j}}|_{L^1}+\left\|e_{\alpha\overline{j}}-\tilde{e}_{\alpha\overline{j}}\right\|_0|f_{i\overline{\beta}}|_{L^1}.
\end{align*}
Recall that $|h^{\alpha\overline{\beta}}|\left\|A_{\alpha}\right\|_0\left\|A_{\beta}\right\|_0=O(1)$ and also note that by Lemma 4.4.2, Lemma 4.4.3 and Corollary 4.4.1 we have
\[
|h^{\alpha\overline{\beta}}|\left\|f_{i\overline{\beta}}\right\|_{0,X-\Omega_c^i}|\tilde{e}_{\alpha\overline{j}}|_{L^1}=O(\dfrac{u_i^3u_j^3}{|t_it_j|})
\]
and
\[
|h^{\alpha\overline{\beta}}|\left\|e_{\alpha\overline{j}}-\tilde{e}_{\alpha\overline{j}}\right\|_0|f_{i\overline{\beta}}|_{L^1}=O(\dfrac{u_i^3u_j^3}{|t_it_j|})
\]
in this case.  Hence we know that $|h^{\alpha\overline{\beta}}R_{i\overline{j}\alpha\overline{\beta}}|=O(\dfrac{u_i^3u_j^3}{|t_it_j|})$ for $\alpha\neq \beta$ and $\alpha\neq i$ or for $\alpha=\beta\geq m+1$.  This is also true for $\alpha\neq \beta$ and $\alpha=i$ by noting that the second term above may be replaced by $\displaystyle \bigg|\int_{X}e_{i\overline{\beta}}f_{\alpha\overline{j}}dv\bigg|$ and thus similar estimate can be done.

Now if $\alpha=\beta\leq m$ and $\alpha\neq i, j$, then we have
\[
|R_{i\overline{j}\alpha\overline{\alpha}}|\leq\bigg|\int_{X}f_{i\overline{j}}e_{\alpha\overline{\alpha}}dv\bigg|+\bigg|\int_{X}f_{i\overline{\alpha}}e_{\alpha\overline{j}}dv\bigg|.
\]
Since
\begin{align*}
\bigg|\int_{X}f_{i\overline{j}}e_{\alpha\overline{\alpha}}dv\bigg|&\leq\bigg|\int_{\Omega_c^{\alpha}}f_{i\overline{j}}\tilde{e}_{\alpha\overline{\alpha}}dv\bigg|+\bigg|\int_{X}f_{i\overline{j}}(e_{\alpha\overline{\alpha}}-\tilde{e}_{\alpha\overline{\alpha}})dv\bigg|\\
&\leq\left\|f_{i\overline{j}}\right\|_{0,\Omega_c^{\alpha}}|\tilde{e}_{\alpha\overline{\alpha}}|_{L^1}+\left\|e_{\alpha\overline{\alpha}}-\tilde{e}_{\alpha\overline{\alpha}}\right\|_0|f_{i\overline{j}}|_{L^1}\\
&=O(\dfrac{u_i^3u_j^3}{|t_it_j|})O(\dfrac{u_{\alpha}^3}{|t_{\alpha}|^2})+O(\dfrac{u_{\alpha}^3}{|t_{\alpha}|^2})O(\dfrac{u_i^3u_j^3}{|t_it_j|})=O(\dfrac{u_{\alpha}^3}{|t_{\alpha}|^2})O(\dfrac{u_i^3u_j^3}{|t_it_j|})
\end{align*}
and
\begin{align*}
\bigg|\int_{X}f_{i\overline{\alpha}}e_{\alpha\overline{j}}dv\bigg|&\leq\bigg|\int_{X-\Omega_c^i}f_{i\overline{\alpha}}\tilde{e}_{\alpha\overline{j}}dv\bigg|+\bigg|\int_{X}f_{i\overline{\alpha}}(e_{\alpha\overline{j}}-\tilde{e}_{\alpha\overline{j}})dv\bigg|\\
&\leq\left\|f_{i\overline{\alpha}}\right\|_{0,X-\Omega_c^i}|\tilde{e}_{\alpha\overline{j}}|_{L^1}+\left\|e_{\alpha\overline{j}}-\tilde{e}_{\alpha\overline{j}}\right\|_0|f_{i\overline{\alpha}}|_{L^1}\\
&=O(\dfrac{u_i^3}{|t_it_{\alpha}|})O(\dfrac{u_{\alpha}^3u_j^3}{|t_{\alpha}t_j|})+O(\dfrac{u_{\alpha}^3u_j^3}{|t_{\alpha}t_j|})O(\dfrac{u_i^3u_{\alpha}^3}{|t_it_{\alpha}|})=O(\dfrac{u_i^3u_i^3u_{\alpha}^3}{|t_it_j||t_{\alpha}|^2}),
\end{align*}
we obtain $|h^{\alpha\overline{\alpha}}R_{i\overline{j}\alpha\overline{\alpha}}|=O(\dfrac{u_i^3u_j^3}{|t_it_j|})$.

Lastly, if $\alpha=\beta=i$, then we have
\[
|R_{i\overline{j}i\overline{i}}|=2\bigg|\int_{X}f_{i\overline{j}}e_{i\overline{i}}dv\bigg|\leq\left\|f_{i\overline{i}}\right\|_0|f_{i\overline{j}}|_{L^1}=O(\dfrac{u_i^2}{|t_i|^2})O(\dfrac{u_i^3u_j^3}{|t_it_j|}).
\]
Thus we have $|h^{i\overline{i}}R_{i\overline{j}i\overline{i}}|=O(\dfrac{u_i^2u_j^3}{|t_it_j|})$.  Similarly we can derive $|h^{j\overline{j}}R_{i\overline{j}j\overline{j}}|=O(\dfrac{u_i^3u_j^2}{|t_it_j|})$.  Therefore we know that
\[\tau_{i\overline{j}}=O(\dfrac{u_i^3u_j^3}{|t_it_j|})+O(\dfrac{u_i^2u_j^3}{|t_it_j|})+O(\dfrac{u_i^3u_j^2}{|t_it_j|})=O(\dfrac{u_i^2u_j^2}{|t_it_j|}(u_i+u_j))\]
for $i, j\leq m$ and $i\neq j$.

For (iii), if $\alpha\neq\beta$ or $\alpha=\beta\geq m+1$, then we have
\begin{align*}
|R_{i\overline{j}\alpha\overline{\beta}}|&\leq\bigg|\int_{X}f_{i\overline{j}}e_{\alpha\overline{\beta}}dv\bigg|+\bigg|\int_{X}f_{i\overline{\beta}}e_{\alpha\overline{j}}dv\bigg|\\
&\leq C_0(\left\|f_{\alpha\overline{\beta}}\right\|_0|f_{i\overline{j}}|_{L^1}+\left\|f_{\alpha\overline{j}}\right\|_0|f_{i\overline{\beta}}|_{L^1})\\
&=\left\|A_{\alpha}\right\|_0\left\|A_{\beta}\right\|_0O(\dfrac{u_i^3}{|t_i|})+O(1)\left\|A_{\alpha}\right\|_0|f_{i\overline{\beta}}|_{L^1}.
\end{align*}
Since by Lemma 4.4.2 we know that in this case $|h^{\alpha\overline{\beta}}|\left\|A_{\alpha}\right\|_0|f_{i\overline{\beta}}|_{L^1}=O(\dfrac{u_i^3}{|t_i|})$, we have
\[
|h^{\alpha\overline{\beta}}R_{i\overline{j}\alpha\overline{\beta}}|=|h^{\alpha\overline{\beta}}|\left\|A_{\alpha}\right\|_0\left\|A_{\beta}\right\|_0O(\dfrac{u_i^3}{|t_i|})+|h^{\alpha\overline{\beta}}|\left\|A_{\alpha}\right\|_0|f_{i\overline{\beta}}|_{L^1}=O(\dfrac{u_i^3}{|t_i|}).
\]

If $\alpha=\beta\leq m$ and $\alpha=i$, we have
\[
|R_{i\overline{j}\alpha\overline{\alpha}}|\leq\bigg|\int_{X}f_{i\overline{j}}e_{\alpha\overline{\alpha}}dv\bigg|+\bigg|\int_{X}f_{i\overline{\alpha}}e_{\alpha\overline{j}}dv\bigg|.
\]
Then as we also have
\[
\bigg|\int_{X}f_{i\overline{\alpha}}e_{\alpha\overline{j}}dv\bigg|\leq\left\|f_{\alpha\overline{j}}\right\|_0|f_{i\overline{\alpha}}|_{L^1}=O(\dfrac{u_{\alpha}}{|t_{\alpha}|})O(\dfrac{u_i^3u_{\alpha}^3}{|t_it_{\alpha}|})=O(\dfrac{u_i^3u_{\alpha}^3}{|t_i||t_{\alpha}|^2})
\]
and
\begin{align*}
\bigg|\int_{X}f_{i\overline{j}}e_{\alpha\overline{\alpha}}dv\bigg|&\leq\bigg|\int_{\Omega_c^{\alpha}}f_{i\overline{j}}\tilde{e}_{\alpha\overline{\alpha}}dv\bigg|+\bigg|\int_{X}f_{i\overline{j}}(e_{\alpha\overline{\alpha}}-\tilde{e}_{\alpha\overline{\alpha}})dv\bigg|\\
&\leq\left\|f_{i\overline{j}}\right\|_{0,\Omega_c^{\alpha}}|\tilde{e}_{\alpha\overline{\alpha}}|_{L^1}+\left\|e_{\alpha\overline{\alpha}}-\tilde{e}_{\alpha\overline{\alpha}}\right\|_0|f_{i\overline{j}}|_{L^1}\\
&=O(\dfrac{u_i^3}{|t_i|})O(\dfrac{u_{\alpha}^3}{|t_{\alpha}|^2})+O(\dfrac{u_{\alpha}^3}{|t_{\alpha}|^2})O(\dfrac{u_i^3}{|t_i|})=O(\dfrac{u_{\alpha}^3}{|t_{\alpha}|^2})O(\dfrac{u_i^3}{|t_i|}),
\end{align*}
we can derive that $|h^{\alpha\overline{\alpha}}R_{i\overline{j}\alpha\overline{\alpha}}|=O(\dfrac{u_i^3}{|t_i|})$.

Now if $\alpha=\beta=i$, we have
\[
|R_{i\overline{j}i\overline{i}}|\leq2\bigg|\int_{X}f_{i\overline{j}}e_{i\overline{i}}dv\bigg|\leq\left\|f_{i\overline{i}}\right\|_0|f_{i\overline{j}}|_{L^1}=O(\dfrac{u_i^2}{|t_i|^2})O(\dfrac{u_i^3}{|t_i|})=O(\dfrac{u_i^5}{|t_i|^3})
\]
which implies $|h^{i\overline{i}}R_{i\overline{j}i\overline{i}}|=O(\dfrac{u_i^2}{|t_i|})$.  Hence we proved (iii).

For (iv), if $\alpha\neq \beta$ or $\alpha=\beta\geq m+1$, then we have
\begin{align*}
|R_{i\overline{j}\alpha\overline{\beta}}|&\leq\bigg|\int_{X}e_{i\overline{j}}f_{\alpha\overline{\beta}}dv\bigg|+\bigg|\int_{X}e_{i\overline{\beta}}f_{\alpha\overline{j}}dv\bigg|\\
&\leq C_0(\left\|f_{i\overline{j}}\right\|_0\left\|f_{\alpha\overline{\beta}}\right\|_0+\left\|f_{i\overline{\beta}}\right\|_0\left\|f_{\alpha\overline{j}}\right\|_0)\\
&=O(1)\left\|A_i\right\|_0\left\|A_j\right\|_0\left\|A_{\alpha}\right\|_0\left\|A_{\beta}\right\|_0+O(1)\left\|A_i\right\|_0\left\|A_{\beta}\right\|_0\left\|A_{\alpha}\right\|_0\left\|A_j\right\|_0\\
&=O(1)\left\|A_{\alpha}\right\|_0\left\|A_{\beta}\right\|_0.
\end{align*}
Hence $|h^{\alpha\overline{\beta}}R_{i\overline{j}\alpha\overline{\beta}}|=|h^{\alpha\overline{\beta}}|\left\|A_{\alpha}\right\|_0\left\|A_{\beta}\right\|_0O(1)=O(1)$ in this case.  If $\alpha=\beta\leq m$, since we have
\begin{align*}
|R_{i\overline{j}\alpha\overline{\alpha}}|&\leq\bigg|\int_{X}e_{i\overline{j}}f_{\alpha\overline{\alpha}}dv\bigg|+\bigg|\int_{X}e_{i\overline{\alpha}}f_{\alpha\overline{j}}dv\bigg|\\
&\leq \left\|e_{i\overline{j}}\right\|_0|f_{\alpha\overline{\alpha}}|_{L^1}+\bigg(\int_{X}|e_{i\overline{\alpha}}|^2dv+\int_{X}|f_{\alpha\overline{j}}|^2dv\bigg)^{\frac{1}{2}}\qquad\quad
\end{align*}
\begin{align*}
&\leq O(1)O(\dfrac{u_{\alpha}^3}{|t_{\alpha}^2|})+\bigg(\int_{X}|f_{i\overline{\alpha}}|^2dv+\int_{X}|f_{\alpha\overline{j}}|^2dv\bigg)^{\frac{1}{2}}\\
&\leq O(\dfrac{u_{\alpha}^3}{|t_{\alpha}^2|})+\left\|A_i\right\|_0\left\|A_j\right\|_0||f_{\alpha\overline{\alpha}}|_{L^1}\\
&=O(\dfrac{u_{\alpha}^3}{|t_{\alpha}^2|})+O(1)O(1)O(\dfrac{u_{\alpha}^3}{|t_{\alpha}^2|})=O(\dfrac{u_{\alpha}^3}{|t_{\alpha}^2|}),
\end{align*}
we know that $|h^{\alpha\overline{\alpha}}R_{i\overline{j}\alpha\overline{\alpha}}|=O(1)$.  Therefore the result follows.

To prove the second part, we first note that the matrix $A:=(\tau_{i\overline{j}})_{i, j\geq m+1}$ can be extended to the boundary non-degenerately using the result of Masur\cite{bib9} about the non-degeneracy of the matrix $B:=(h_{i\overline{j}})_{i, j\geq m+1}$ on the boundary and the work by Wolpert showing that $A$ is bounded from below by a constant multiple of $B$ (in which the constant only depends on $g$).  This implies that $A$ has a positive lower bound.  By (iv) in the first part we also know that $A$ is bounded from above and each of its entries is bounded.  Then by direct computation we can show that $\displaystyle\det(\tau)=(\prod_{k=1}^m\dfrac{3}{4\pi^2}\dfrac{u_k^2}{|t_k|^2})\det(A)(1+O(u_0))$ (indeed we just need to use the explicit formula for $\det(\tau)$, i.e. $\displaystyle\det(\tau)=\sum_{\sigma\in S_n}\textrm{sgn}(\sigma)\tau_{1\overline{\sigma(1)}}\cdots\tau_{n\overline{\sigma(n)}}$, and notice that the sum $\displaystyle\sum_{\sigma\in S_n'}\textrm{sgn}(\sigma)\tau_{1\overline{1}}\cdots\tau_{m\overline{m}}\tau_{m+1\overline{\sigma(m+1)}}\cdots\tau_{n\overline{\sigma(n)}}$, where $S_n'$ denotes the subgroup of $S_n$ fixing $1$ to $m$, has the dominating order).  Let $\epsilon<1$ be fixed.  As we know that for $|(t, s)|$ small enough, $1+O(u_0)\geq 1-\epsilon$, we have
\[
\dfrac{\displaystyle\prod_{k=1}^m\dfrac{3}{4\pi^2}\dfrac{u_k^2}{|t_k|^2}}{\det(\tau)}\leq\dfrac{1}{\det(A)(1-\epsilon)}=O(1).
\]

Now let $\Phi_{ij}$ be the $(i, j)$-minor of $(\tau_{i\overline{j}})$.  Similarly we can also show that
\[
\det(\Phi_{ij})=\left\{
\begin{array}{ll}
(\prod_{k=1, k\neq i}^m\dfrac{3}{4\pi^2}\dfrac{u_k^2}{|t_k|^2})\det(A)(1+O(u_0)) & \text{if } i=j\leq m;\\
(\prod_{k=1, k\neq i, j}^m\dfrac{3}{4\pi^2}\dfrac{u_k^2}{|t_k|^2})O(\dfrac{u_i^2u_j^2}{|t_it_j|}) & \text{if } i, j\leq m \text{ and } i\neq j;\\
(\prod_{k=1, k\neq i}^m\dfrac{3}{4\pi^2}\dfrac{u_k^2}{|t_k|^2})O(\dfrac{u_i^2}{|t_i|}) & \text{if } i\leq m \text{ and } j\geq m+1;\\
(\prod_{k=1}^m\dfrac{3}{4\pi^2}\dfrac{u_k^2}{|t_k|^2})O(1) & \text{if } i, j\geq m+1.
\end{array}\right.
\]
Finally by the fact that $|\tau^{i\overline{j}}|=\bigg|\dfrac{\det(\Phi_{ij})}{\det(\tau)}\bigg|$ we proved the second part.
\end{proof}

\subsection{Estimates on the Curvature of the Ricci Metric}
Now we start to estimate the holomorphic sectional curvature of the Ricci metric.  The following estimates on the norms will be required for the estimate of the error terms.
\begin{lem} (Lemma 4.6, \cite{bib1})
Let $f, g\in C^{\infty}(X, \mathbb{C})$ be smooth functions such that $(\Box+1)f=g$.  Then there exists a constant $C_0$ such that
\begin{itemize}
\item[(i)] $|K_0f|_{L^2}\leq C_0|K_0g|_{L^2}$;
\item[(ii)] $|K_1K_0f|_{L^2}\leq C_0|K_0g|_{L^2}$.
\end{itemize}
\end{lem}
\begin{proof}
This lemma follows from applying Schwarz inequality on the Bochner formula:
\[(\Box+1)h+|K_1K_0f|^2=K_0f\overline{K_0g}+\overline{K_0f}K_0g-|f-g|^2\]
with $h=|K_0f|^2$.
\end{proof}

We will need the estimate on the section $K_0f_{i\overline{j}}$ as well (Lemma 4.7-4.8, \cite{bib1}):
\begin{lem}
Let $K_0$ and $K_1$ be the Maass operators, and let $P=K_1K_0$.  We have
\begin{itemize}
\item[(i)] $\left\|K_0f_{i\overline{i}}\right\|_{0,\Omega_c^i}=O(\dfrac{u_i^2}{|t_i|^2})$ and $\left\|K_0f_{i\overline{i}}\right\|_{0,X-\Omega_c^i}=O(\dfrac{u_i^6}{|t_i|^2})$ if $i\leq m$;
\item[(ii)] $\left\|K_0f_{i\overline{j}}\right\|_{0,\Omega_c^i}=O(\dfrac{u_iu_j^3}{|t_it_j|})$, $\left\|K_0f_{i\overline{j}}\right\|_{0,\Omega_c^j}=O(\dfrac{u_i^3u_j}{|t_it_j|})$ and $\left\|K_0f_{i\overline{j}}\right\|_{0,X-(\Omega_c^i\cup\Omega_c^j)}=O(\dfrac{u_i^3u_j^3}{|t_it_j|})$ if $i, j\leq m$ and $i\neq j$;
\item[(iii)] $\left\|K_0f_{i\overline{j}}\right\|_{0,\Omega_c^i}=O(\dfrac{u_i}{|t_i|})$ and $\left\|K_0f_{i\overline{j}}\right\|_{0,X-\Omega_c^i}=O(\dfrac{u_i^3}{|t_i|})$ if $i\leq m$ and $j\geq m+1$;
\item[(iv)] $\left\|K_0f_{i\overline{j}}\right\|_0=O(1)$ if $i, j\geq m+1$;
\item[(v)] $|K_0f_{i\overline{i}}|_{L^2}^2=O(\dfrac{u_i^5}{|t_i|^4})$ if $i\leq m$;
\item[(vi)] $|K_0f_{i\overline{j}}|_{L^2}^2=O(\dfrac{u_i^3u_j^3}{|t_it_j|^2})$ if $i, j\leq m$ and $i\neq j$;
\item[(vii)] $|K_0f_{i\overline{j}}|_{L^2}^2=O(\dfrac{u_i^3}{|t_i|^2})$ if $i\leq m$ and $j\geq m+1$;
\item[(viii)] $|K_0f_{i\overline{i}}|_{L^2}^2=O(1)$ if $i, j\geq m+1$;
\item[(ix)] $\left\|f_{i\overline{i}}-\tilde{f}_{i\overline{i}}\right\|_1=O(\dfrac{u_i^4}{|t_i|^2})$ if $i\leq m$;
\item[(x)] $|f_{i\overline{i}}|_{L^2}^2=O(\dfrac{u_i^5}{|t_i|^4})$ if $i\leq m$;
\item[(xi)] $|P(\tilde{e}_{i\overline{i}})|_{L^1}=O(\dfrac{u^3}{|t_i|^2})$ if $i\leq m$.
\end{itemize}
\end{lem}

\bigskip
Note that this lemma can be proved by similar computational method used in Lemma 4.4.2.

\bigskip
For the estimate of the curvature of $\tau$ by its curvature formula (Theorem 4.3.2), we need to control the term $\displaystyle \int_X Q_{k\overline{l}}(e_{i\overline{j}})e_{\alpha\overline{\beta}}dv$.  The next lemma (Theorem 4.9, \cite{bib1}) would help us to do so:
\begin{lem}
We have
\begin{align*}
\int_X Q_{k\overline{l}}(e_{i\overline{j}})e_{\alpha\overline{\beta}}dv&=-\int_Xf_{k\overline{l}}(K_0e_{i\overline{j}}\overline{K_0}e_{\alpha\overline{\beta}}+\overline{K_0}e_{i\overline{j}}K_0e_{\alpha\overline{\beta}})dv\\
&\quad-\int_X(\Box e_{i\overline{j}}K_0e_{\alpha\overline{\beta}}\overline{K_0}e_{k\overline{l}}+\Box e_{\alpha\overline{\beta}}K_0e_{i\overline{j}}\overline{K_0}e_{k\overline{l}})dv.
\end{align*}
\end{lem}

\bigskip
For the estimate of the holomorphic sectional curvature of $\tau$, let $i=j=k=l$ for the formula in Theorem 4.3.2.  Then we consider
\[\tilde{R}_{i\overline{i}i\overline{i}}=G_1+G_2\]
where $G_1$ consists of all the terms in the formula with all the indices $\alpha$, $\beta$, $\gamma$, $\delta$, $p$ and $q$ being $i$ and $G_2=\tilde{R}_{i\overline{i}i\overline{i}}-G_1$.  If $i\leq m$, the dominating term is $G_1$ which is given by
\begin{align*}
G_1&=24h^{i\overline{i}}\int_X(\Box+1)^{-1}(\xi_i(e_{i\overline{i}}))\overline{\xi}_i(e_{i\overline{i}})dv\\
&\quad+6h^{i\overline{i}}\int_XQ_{i\overline{i}}(e_{i\overline{i}})e_{i\overline{i}}dv\\
&\quad-36\tau^{i\overline{i}}(h^{i\overline{i}})^2\bigg|\int_X\xi_i(e_{i\overline{i}})e_{i\overline{i}}dv\bigg|^2\\
&\quad+\tau_{i\overline{i}}h^{i\overline{i}}R_{i\overline{i}i\overline{i}}.
\end{align*}

Now we are prepared to estimate the holomorphic sectional curvature of $\tau$ (Theorem 4.4, \cite{bib1}):
\begin{thm}
Let $X_0\in\overline{\mathcal{M}_g}-\mathcal{M}_g$ be a codimension $m$ boundary point and let $(t_1, ..., s_n)$ be the pinching coordinate near $X_0$.  Then the holomorphic sectional curvature of $\tau$ is negative in the degenerate directions and is bounded in the non-degenerate directions.  More precisely, there exists a constant $\delta>0$ such that if $|(t,s)|<\delta$, then
\[\tilde{R}_{i\overline{i}i\overline{i}}=\dfrac{3u_i^4}{8\pi^4|t_i|^4}(1+O(u_0))>0\]
for $i\leq m$ and
\[\tilde{R}_{i\overline{i}i\overline{i}}=O(1)\]
for $i\geq m+1$.
Furthermore, the holomorphic sectional curvature, the bisectional curvature and the Ricci curvature of $\tau$ are bounded from below and above on $\mathcal{M}_g$.
\end{thm}
\begin{proof}
The following lemma (indeed the method used in its proof) will be the main ingredient in the proof of this theorem:
\begin{lem}
If $i\leq m$, then $|G_2|=O(\dfrac{u_5}{|t_i|^4})$.
\end{lem}
This lemma can be proved by case-by-case checking and direct computation, and its proof can be found in the appendix of \cite{bib1}.

We first consider $\tilde{R}_{i\overline{i}i\overline{i}}$.  From Lemma 4.4.7 we have
\[\int_X Q_{i\overline{i}}(e_{i\overline{i}})e_{i\overline{i}}dv=\int_X|K_0e_{i\overline{i}}|^2(2e_{i\overline{i}}-4f_{i\overline{i}})dv,\]
which implies
\begin{align*}
G_1&=24h^{i\overline{i}}\int_X(\Box+1)^{-1}(\xi_i(e_{i\overline{i}}))\overline{\xi}_i(e_{i\overline{i}})dv+6h^{i\overline{i}}\int_X|K_0e_{i\overline{i}}|^2(2e_{i\overline{i}}-4f_{i\overline{i}})dv\\
&\quad-36\tau^{i\overline{i}}(h^{i\overline{i}})^2\bigg|\int_X\xi_i(e_{i\overline{i}})e_{i\overline{i}}dv\bigg|^2+\tau_{i\overline{i}}h^{i\overline{i}}R_{i\overline{i}i\overline{i}}.
\end{align*}
Now let $i\leq m$ first.  From the proof of the Corollary 4.4.2, we know that $h^{i\overline{i}}R_{i\overline{i}i\overline{i}}=\dfrac{3}{4\pi^2}\dfrac{u_i^2}{|t_i|^2}(1+O(u_0))$.  Thus by Corollary 4.4.2 we have
\[\tau_{i\overline{i}}h^{i\overline{i}}R_{i\overline{i}i\overline{i}}=\bigg(\dfrac{3}{4\pi^2}\dfrac{u_i^2}{|t_i|^2}(1+O(u_0))\bigg)^2=\dfrac{9}{16\pi^4}\dfrac{u_i^4}{|t_i|^4}(1+O(u_0)).\]
For the second term, we have
\begin{align*}
\int_X|K_0e_{i\overline{i}}|^2(2e_{i\overline{i}}-4f_{i\overline{i}})dv&=\int_X|K_0\tilde{e}_{i\overline{i}}|^2(2\tilde{e}_{i\overline{i}}-4\tilde{f}_{i\overline{i}})dv\\
&\quad+\int_X(|K_0e_{i\overline{i}}|^2-|K_0\tilde{e}_{i\overline{i}}|^2)(2\tilde{e}_{i\overline{i}}-4\tilde{f}_{i\overline{i}})dv\\
&\quad+\int_X|K_0e_{i\overline{i}}|^2(2(e_{i\overline{i}}-\tilde{e}_{i\overline{i}})-4(f_{i\overline{i}}-\tilde{f}_{i\overline{i}}))dv.
\end{align*}
Notice that by Lemma 4.4.6 and Lemma 4.4.3 we have
\begin{align*}
\bigg|\int_X(|K_0e_{i\overline{i}}|^2-|K_0\tilde{e}_{i\overline{i}}|^2)(2\tilde{e}_{i\overline{i}}-4\tilde{f}_{i\overline{i}})dv\bigg|&\leq\left\||K_0e_{i\overline{i}}|^2-|K_0\tilde{e}_{i\overline{i}}|^2\right\|_0\int_X(2|\tilde{e}_{i\overline{i}}|+4|\tilde{f}_{i\overline{i}}|)dv\\
&\leq \left\||K_0e_{i\overline{i}}|+|K_0\tilde{e}_{i\overline{i}}|\right\|_0\left\|K_0(e_{i\overline{i}}-\tilde{e}_{i\overline{i}})\right\|_0\\
&\quad\cdot\int_X(2|\tilde{e}_{i\overline{i}}|+4|\tilde{f}_{i\overline{i}}|)dv\\
&=O(\dfrac{u_i^2}{|t_i|^2})O(\dfrac{u_i^4}{|t_i|^2})O(\dfrac{u_i^3}{|t_i|^2})=O(\dfrac{u_i^9}{|t_i|^6})
\end{align*}
and
\begin{align*}
\bigg|\int_X|K_0e_{i\overline{i}}|^2(2(e_{i\overline{i}}-\tilde{e}_{i\overline{i}})-4(f_{i\overline{i}}-\tilde{f}_{i\overline{i}}))dv\bigg|&\leq C_0\left\|K_0e_{i\overline{i}}\right\|_0^2(2\left\|e_{i\overline{i}}-\tilde{e}_{i\overline{i}}\right\|_0+4\left\|f_{i\overline{i}}-\tilde{f}_{i\overline{i}}\right\|_0)\\
&=O(\dfrac{u_i^4}{|t_i|^4})O(\dfrac{u_i^4}{|t_i|^2})=O(\dfrac{u_i^8}{|t_i|^6}).
\end{align*}
Then consider
\[\int_X|K_0\tilde{e}_{i\overline{i}}|^2(2\tilde{e}_{i\overline{i}}-4\tilde{f}_{i\overline{i}})dv=\int_{\Omega_{c_1}^i}|K_0\tilde{e}_{i\overline{i}}|^2(2\tilde{e}_{i\overline{i}}-4\tilde{f}_{i\overline{i}})dv+\int_{\Omega_c^i-\Omega_{c_1}^i}|K_0\tilde{e}_{i\overline{i}}|^2(2\tilde{e}_{i\overline{i}}-4\tilde{f}_{i\overline{i}})dv.\]
Since we have
\begin{align*}
\bigg|\int_{\Omega_c^i-\Omega_{c_1}^i}|K_0\tilde{e}_{i\overline{i}}|^2(2\tilde{e}_{i\overline{i}}-4\tilde{f}_{i\overline{i}})dv\bigg|&\leq C_0\left\|K_0\tilde{e}_{i\overline{i}}\right\|_0^2(\left\|\tilde{e}_{i\overline{i}}\right\|_{0,\Omega_c^i-\Omega_{c_1}^i}+\left\|\tilde{f}_{i\overline{i}}\right\|_{0,\Omega_c^i-\Omega_{c_1}^i})\\
&=O(\dfrac{u_i^4}{|t_i|^4})O(\dfrac{u_i^4}{|t_i|^2})=O(\dfrac{u_i^8}{|t_i|^6})
\end{align*}
and by direct computation we also have
\[\int_{\Omega_{c_1}^i}|K_0\tilde{e}_{i\overline{i}}|^2(2\tilde{e}_{i\overline{i}}-4\tilde{f}_{i\overline{i}})dv=-\dfrac{3u_i^7}{64\pi^4|t_i|^6}(1+O(u_0)).\]
Therefore by Corollary 4.4.1 we obtain
\begin{align*}
6h^{i\overline{i}}\int_X|K_0e_{i\overline{i}}|^2(2e_{i\overline{i}}-4f_{i\overline{i}})dv&=\dfrac{12|t_i|^2}{u_i^3}(1+O(u_0))\bigg(-\dfrac{3u_i^7}{64\pi^4|t_i|^6}(1+O(u_0))+O(\dfrac{u_i^8}{|t_i|^6})\bigg)\\
&=-\dfrac{9u_i^4}{16\pi^4|t_i|^4}(1+O(u_0)).
\end{align*}

For the third term, using similar method as above, we consider
\begin{align*}
\int_X\xi_i(e_{i\overline{i}})e_{i\overline{i}}dv&=\int_X\xi_i(\tilde{e}_{i\overline{i}})\tilde{e}_{i\overline{i}}dv+\int_X\xi_i(\tilde{e}_{i\overline{i}})(e_{i\overline{i}}-\tilde{e}_{i\overline{i}})dv\\
&\quad+\int_X\xi_i(e_{i\overline{i}}-\tilde{e}_{i\overline{i}})e_{i\overline{i}}dv.
\end{align*}
Then since we have (by Lemma 4.3.1, Lemma 4.4.2 and Lemma 4.4.3)
\begin{align*}
\bigg|\int_X\xi_i(\tilde{e}_{i\overline{i}})(e_{i\overline{i}}-\tilde{e}_{i\overline{i}})dv\bigg|&\leq C_0\left\|\xi_i(\tilde{e}_{i\overline{i}})\right\|_0\left\|e_{i\overline{i}}-\tilde{e}_{i\overline{i}}\right\|_0\\
&\leq C_0\left\|A_i\right\|_0\left\|K_1K_0(\tilde{e}_{i\overline{i}})\right\|_0\left\|e_{i\overline{i}}-\tilde{e}_{i\overline{i}}\right\|_0\\
&\leq C_0\left\|A_i\right\|_0\left\|\tilde{e}_{i\overline{i}}\right\|_2\left\|e_{i\overline{i}}-\tilde{e}_{i\overline{i}}\right\|_0\\
&=O(\dfrac{u_i}{|t_i|})O(\dfrac{u_i^2}{|t_i|^2})O(\dfrac{u_i^4}{|t_i|^2})=O(\dfrac{u_i^7}{|t_i|^5}),
\end{align*}
\begin{align*}
\bigg|\int_X\xi_i(e_{i\overline{i}}-\tilde{e}_{i\overline{i}})e_{i\overline{i}}dv\bigg|&\leq \left\|\xi_i(e_{i\overline{i}}-\tilde{e}_{i\overline{i}})\right\|_0\int_Xe_{i\overline{i}}dv\\
&\leq\left\|A_i\right\|_0\left\|e_{i\overline{i}}-\tilde{e}_{i\overline{i}}\right\|_2\int_Xf_{i\overline{i}}dv\\
&\leq\left\|A_i\right\|_0\left\|f_{i\overline{i}}-\tilde{f}_{i\overline{i}}\right\|_1h_{i\overline{i}}\\
&=O(\dfrac{u_i}{|t_i|})O(\dfrac{u_i^4}{|t_i|^2})O(\dfrac{u_i^3}{|t_i|^2})=O(\dfrac{u_i^8}{|t_i|^5})
\end{align*}
and
\[\bigg|\int_{\Omega_c^i-\Omega_{c_1}^i}\xi_i(\tilde{e}_{i\overline{i}})\tilde{e}_{i\overline{i}}dv\bigg|\leq C_0\left\|\xi_i(\tilde{e}_{i\overline{i}})\right\|_0\left\|\tilde{e}_{i\overline{i}}\right\|_{0,\Omega_c^i-\Omega_{c_1}^i}=O(\dfrac{u_i^3}{|t_i|^3})O(\dfrac{u_i^4}{|t_i|^2})=O(\dfrac{u_i^7}{|t_i|^5}),\]
we know that
\[\int_X\xi_i(e_{i\overline{i}})e_{i\overline{i}}dv=\int_{\Omega_{c_1}^i}\xi_i(\tilde{e}_{i\overline{i}})\tilde{e}_{i\overline{i}}dv+O(\dfrac{u_i^7}{|t_i|^5}).\]
Note that we have the following explicit expression of $\xi_i(\tilde{e}_{i\overline{i}})$ on $\Omega_{c_1}^i$:
\[\xi_i(\tilde{e}_{i\overline{i}})=-\dfrac{z_i}{\overline{z_i}}\sin^2\tau_i\overline{b_i}P(\tilde{e}_{i\overline{i}})-\dfrac{z_i}{\overline{z_i}}\sin^2\tau_i\overline{p_i}P(\tilde{e}_{i\overline{i}}).\]
Since $\left\|\dfrac{z_i}{\overline{z_i}}\sin^2\tau_i\overline{p_i}P(\tilde{e}_{i\overline{i}})\right\|_{0,\Omega_{c_1}^i}=O(\dfrac{u_i^5}{|t_i|^3})$, we know that
\begin{align*}
\bigg|\int_{\Omega_{c_1}^i}\dfrac{z_i}{\overline{z_i}}\sin^2\tau_i\overline{p_i}P(\tilde{e}_{i\overline{i}})\tilde{e}_{i\overline{i}}dv\bigg|&\leq\left\|\dfrac{z_i}{\overline{z_i}}\sin^2\tau_i\overline{p_i}P(\tilde{e}_{i\overline{i}})\right\|_{0,\Omega_{c_1}^i}\int_{\Omega_{c_1}^i}\tilde{e}_{i\overline{i}}dv\\
&=O(\dfrac{u_i^5}{|t_i|^3})O(\dfrac{u_i^3}{|t_i|^2})=O(\dfrac{u_i^8}{|t_i|^5}).
\end{align*}
Then by direct computation again we have
\[\int_{\Omega_{c_1}^i}-\dfrac{z_i}{\overline{z_i}}\sin^2\tau_i\overline{b_i}P(\tilde{e}_{i\overline{i}})\tilde{e}_{i\overline{i}}dv=-\dfrac{u_i^6}{32\pi^3|t_i|^4t_i}(1+O(u_0)).\]
Therefore we obtain
\[\int_X\xi_i(e_{i\overline{i}})e_{i\overline{i}}dv=-\dfrac{u_i^6}{32\pi^3|t_i|^4t_i}(1+O(u_0))\]
which implies that
\begin{align*}
36\tau^{i\overline{i}}(h^{i\overline{i}})^2\bigg|\int_X\xi_i(e_{i\overline{i}})e_{i\overline{i}}dv\bigg|^2&\dfrac{144\pi^2|t_i|^2}{3u_i^2}(\dfrac{2|t_i|^2}{u_i^3})^2(\dfrac{u_i^6}{32\pi^3|t_i|^4t_i})^2(1+O(u_0))\\
&=\dfrac{3u_i^4}{16\pi^4|t_i|^4}(1+O(u_0)).
\end{align*}

For the estimate of the first term, again by similar method we can derive that
\[\int_X(\Box+1)^{-1}(\xi_i(e_{i\overline{i}}))\overline{\xi}_i(e_{i\overline{i}})dv=\int_X(\Box+1)^{-1}(\xi_i(\tilde{e}_{i\overline{i}}))\overline{\xi}_i(\tilde{e}_{i\overline{i}})dv+O(\dfrac{u_i^8}{|t_i|^6}).\]
However, to estimate $(\Box+1)^{-1}(\xi_i(\tilde{e}_{i\overline{i}}))$, we need to introduce another approximate function.  Choose $c_2<c_1$ and let $\eta_1\in C^{\infty}(\mathbb{R}, [0,1])$ be the cut-off function given by
\[
\left\{
\begin{array}{ll}
\eta(x)=1 & \text{if } x\leq\log c_2;\\
\eta(x)=0 & \text{if } x\geq \log c_1;\\
0<\eta(x)<1 & \text{if } \log c_2<x<\log c_1.
\end{array}\right.
\]
Now for $i\leq m$, define the function $d_i$ by
\[
d_i(z)=\left\{
\begin{array}{ll}
-\frac{1}{8}\sin^2\tau_i\cos2\tau_i|b_i|^2\overline{b_i} & \text{if } z\in\Omega_{c_2}^i;\\
-\frac{1}{8}\sin^2\tau_i\cos2\tau_i|b_i|^2\overline{b_i}\eta_1(\log r_i) & \text{if } z\in\Omega_{c_1}^i \text{ and } c_2<r_i<c_1;\\
-\frac{1}{8}\sin^2\tau_i\cos2\tau_i|b_i|^2\overline{b_i}\eta_1(\log\rho_i-\log r_i) & \text{if } z\in\Omega_{c_1}^i \text{ and } c_1^{-1}\rho_i<r_i<c_2^{-1}\rho_i;\\
0 & \text{if } z\in X-\Omega_{c_1}^i.
\end{array}
\right.
\]
Notice that we have this computational fact:
\[\left\|\xi_i(\tilde{e}_{i\overline{i}})-(\Box+1)d_i\right\|_0=O(\dfrac{u_i^5}{|t_i|^3}),\]
which implies
\[\left\|(\Box+1)^{-1}(\xi_i(\tilde{e}_{i\overline{i}}))-d_i\right\|_0=O(\dfrac{u_i^5}{|t_i|^3}).\]
Consider
\[\int_X(\Box+1)^{-1}(\xi_i(\tilde{e}_{i\overline{i}}))\overline{\xi}_i(\tilde{e}_{i\overline{i}})dv=\int_Xd_i\overline{\xi}_i(\tilde{e}_{i\overline{i}})dv+\int_X((\Box+1)^{-1}(\xi_i(\tilde{e}_{i\overline{i}}))-d_i)\overline{\xi}_i(\tilde{e}_{i\overline{i}})dv.\]
We know that
\begin{align*}
\bigg|\int_X((\Box+1)^{-1}(\xi_i(\tilde{e}_{i\overline{i}}))-d_i)\overline{\xi}_i(\tilde{e}_{i\overline{i}})dv\bigg|&\leq C_0\left\|(\Box+1)^{-1}(\xi_i(\tilde{e}_{i\overline{i}}))-d_i\right\|_0\left\|\overline{\xi}_i(\tilde{e}_{i\overline{i}})\right\|_0\\
&=O(\dfrac{u_i^5}{|t_i|^3})O(\dfrac{u_i^3}{|t_i|^3})=O(\dfrac{u_i^8}{|t_i|^6}).
\end{align*}
Moreover, we have
\[d_i\overline{\xi}_i(\tilde{e}_{i\overline{i}})=-d_i\dfrac{\overline{z_i}}{z_i}\sin^2\tau_ib_i\overline{P}(\tilde{e}_{i\overline{i}})-d_i\dfrac{\overline{z_i}}{z_i}\sin^2\tau_ip_i\overline{P}(\tilde{e}_{i\overline{i}}).\]
Using similar argument as before, from the facts
\[\left\|d_i\dfrac{\overline{z_i}}{z_i}\sin^2\tau_ip_i\overline{P}(\tilde{e}_{i\overline{i}})\right\|_0=O(\dfrac{u_i^8}{|t_i|^6}),\]
\[\left\|d_i\dfrac{\overline{z_i}}{z_i}\sin^2\tau_ib_i\overline{P}(\tilde{e}_{i\overline{i}})\right\|_{0,\Omega_{c_1}^i-\Omega_{c_2}^i}=O(\dfrac{u_i^8}{|t_i|^6})\]
and
\[-\int_{\Omega_{c_2}^i}d_i\dfrac{\overline{z_i}}{z_i}\sin^2\tau_ib_i\overline{P}(\tilde{e}_{i\overline{i}})=\dfrac{3u_i^7}{256\pi^4|t_i|^6}(1+O(u_0)),\]
we know that
\[\int_X(\Box+1)^{-1}(\xi_i(e_{i\overline{i}}))\overline{\xi}_i(e_{i\overline{i}})dv=\dfrac{3u_i^7}{256\pi^4|t_i|^6}(1+O(u_0)).\]
Hence we have
\[24h^{i\overline{i}}\int_X(\Box+1)^{-1}(\xi_i(e_{i\overline{i}}))\overline{\xi}_i(e_{i\overline{i}})dv=\dfrac{9u_i^4}{16\pi^4|t_i|^4}(1+O(u_0)).\]
Combining all the previous results, we have
\begin{align*}
G_1&=\dfrac{9u_i^4}{16\pi^4|t_i|^4}(1+O(u_0))-\dfrac{9u_i^4}{16\pi^4|t_i|^4}(1+O(u_0))-\dfrac{3u_i^4}{16\pi^4|t_i|^4}(1+O(u_0))\\
&\quad+\dfrac{9}{16\pi^4}\dfrac{u_i^4}{|t_i|^4}(1+O(u_0))\\
&=\dfrac{3u_i^4}{8\pi^4|t_i|^4}(1+O(u_0)).
\end{align*}
Thus with the aid of Lemma 4.4.8 we proved the case of $i\leq m$.  The case of $i\geq m+1$ can be proved by case-by-case checking as in the proof of Lemma 4.4.8 (Lemma 4.10, \cite{bib1}).

Before proving the boundedness of the curvatures, we first give a weak estimate on the full curvature of $\tau$.  Let
\[
\Lambda_i=\left\{
\begin{array}{ll}
\frac{u_i}{|t_i|}& \text{if } i\leq m;\\
1 & \text{if } i\geq m+1.
\end{array}\right.
\]
Again we can use the case-by-case checking to show the following estimates:
\[\tilde{R}_{i\overline{j}k\overline{l}}=O(\Lambda_i\Lambda_j\Lambda_k\Lambda_l)O(u_0)\]
if at least one of these indices $i, j, k, l$ is less than or equal to $m$ and not all of them are equal;
\[\tilde{R}_{i\overline{j}k\overline{l}}=O(1)\]
if $i, j, k, l\geq m+1$.

Now we can show the boundedness of the curvature.  For the holomorphic sectional curvature, from the first part of this theorem and Corollary 4.4.2, we know that there is a constant $C_0>1$ depending on $X_0$ and $\delta$ such that if $|(t,s)|<\delta$, then
\begin{itemize}
\item[(i)] $C_0^{-1}\tau_{i\overline{i}}^2\leq\tilde{R}_{i\overline{i}i\overline{i}}\leq C_0\tau_{i\overline{i}}^2$ if $i\leq m$;
\item[(ii)] $|\tilde{R}_{i\overline{i}i\overline{i}}|\leq C_0\tau_{i\overline{i}}^2$ if $i\geq m+1$.
\end{itemize}
We cover the divisor $Y=\tilde{\mathcal{M}_g}-\mathcal{M}_g$ by such open coordinate charts.  By the compactness of $Y$ we can pick a finite covering $\{\Xi_1, ..., \Xi_p\}$ of $Y$.  Note that there exists an open neighborhood $V$ of Y such that $\overline{V}\subset\cup_{s=1}^p\Xi_s$.  By the weak estimate and the argument above, we know that the holomorphic sectional curvature of $\tau$ is bounded from below and above on $V$.  However by the fact the $\mathcal{M}_g-V$ is a compact set in $\mathcal{M}_g$, the holomorphic sectional curvature is also bounded on $\mathcal{M}_g-V$ and thus they are bounded on $\mathcal{M}_g$.  Using this kind of compactness argument with the weak estimate, we can also show the boundedness of the bisectional curvature and the Ricci curvature of $\tau$.
\end{proof}

\begin{rem}
A complete proof of the boundedness of the bisectional curvature of $\tau$ can be found in the proof of the Theorem 3.2 in \cite{bib2}.
\end{rem}

\newpage

\section{Perturbed Ricci Metric}
As mentioned before, the perturbed Ricci metric is another useful new complete K\"{a}hler metric introduced by Yau et. al..  Now let's examine this metric.

\subsection{Definition and the Curvature Formula of the Perturbed Ricci Metric}
The perturbed Ricci metric is in fact obtained by adding a constant multiple of the Weil-Petersson metric to the Ricci metric.
\begin{defn}
For any constant $C>0$, the perturbed Ricci metric $\tilde{\tau}$ with constant $C$ is defined by
\[\tilde{\tau}_{i\overline{j}}=\tau_{i\overline{j}}+Ch_{i\overline{j}}.\]
\end{defn}

\bigskip
This metric is a natural complete metric whose holomorphic sectional curvature is negatively bounded.  To be more precise, the holomorphic sectional curvature of the perturbed Ricci metric near an interior point of the moduli space or in the non-degenerate directions near a boundary point is dominated by the curvature of the large constant multiple of the Weil-Petersson metric.  We will go through this in the next subsection.  To get prepared, we have to consider the curvature formula of the perturbed Ricci metric first.
\begin{thm} (Theorem 5.1, \cite{bib1})
Let $s_1, ..., s_n$ be normal coordinates at $s\in\mathcal{M}_g$ with respect to the Weil-Petersson metric.  Then at $s$, we have
\begin{align*}
P_{i\overline{j}k\overline{l}}&=h^{\alpha\overline{\beta}}\bigg\{\sigma_1\sigma_2\int_{X_s}\big\{(\Box+1)^{-1}(\xi_k(e_{i\overline{j}}))\overline{\xi}_l(e_{\alpha\overline{\beta}})+(\Box+1)^{-1}(\xi_k(e_{i\overline{j}}))\overline{\xi}_{\beta}(e_{\alpha\overline{l}})\big\}dv\bigg\}\\
&\quad+h^{\alpha\overline{\beta}}\bigg\{\sigma_1\int_{X_s}Q_{k\overline{l}}(e_{i\overline{j}})e_{\alpha\overline{\beta}}dv\bigg\}\\
&\quad-\tilde{\tau}^{p\overline{q}}h^{\alpha\overline{\beta}}h^{\gamma\overline{\delta}}\bigg\{\sigma_1\int_{X_s}\xi_k(e_{i\overline{q}})e_{\alpha\overline{\beta}}dv\bigg\}\bigg\{\tilde{\sigma}_1\int_{X_s}\overline{\xi}_l(e_{p\overline{j}})e_{\gamma\overline{\delta}}dv\bigg\}\\
&\quad+\tau_{p\overline{j}}h^{p\overline{q}}R_{i\overline{q}k\overline{l}}+CR_{i\overline{j}k\overline{l}},
\end{align*}
where $P_{i\overline{j}k\overline{l}}$ is the curvature tensor of the perturbed Ricci metric.
\end{thm}
\begin{proof}
By the normality of $s$ we know that $\partial_kh_{i\overline{q}}=\partial_{\overline{l}}h_{p\overline{j}}=0$, $\Gamma^p_{ik}=h^{p\overline{j}}\partial_kh_{i\overline{j}}=0$ and $R_{i\overline{j}k\overline{l}}=\partial_{\overline{l}}\partial_kh_{i\overline{j}}$ at $s$.  Then we have \[\partial_k\tilde{\tau}_{i\overline{q}}=\partial_k\tau_{i\overline{q}}+C\partial_kh_{i\overline{q}}=\partial_k\tau_{i\overline{q}},\]
\[\partial_{\overline{l}}\tilde{\tau}_{p\overline{j}}=\partial_{\overline{l}}\tau_{p\overline{j}}+C\partial_{\overline{l}}h_{p\overline{j}}=\partial_k\tau_{i\overline{j}}\]
and
\[\partial_{\overline{l}}\partial_k\tilde{\tau}_{i\overline{j}}=\partial_{\overline{l}}\partial_k\tau_{i\overline{j}}+C\partial_{\overline{l}}\partial_kh_{i\overline{j}}=\partial_{\overline{l}}\partial_k\tau_{i\overline{j}}+CR_{i\overline{j}k\overline{l}}\]
at $s$.  Hence using the formula of $\partial_{\overline{l}}\partial_k\tau_{i\overline{j}}$ and Theorem 4.3.2, we have
\begin{align*}
P_{i\overline{j}k\overline{l}}&=\partial_{\overline{l}}\partial_k\tilde{\tau}_{i\overline{j}}-\tilde{\tau}^{p\overline{q}}\partial_k\tilde{\tau}_{i\overline{q}}\partial_{\overline{l}}\tilde{\tau}_{p\overline{j}}\\
&=\partial_{\overline{l}}\partial_k\tau_{i\overline{j}}+CR_{i\overline{j}k\overline{l}}-\tilde{\tau}^{p\overline{q}}\partial_k\tau_{i\overline{q}}\partial_k\tau_{i\overline{j}}\\
&=h^{\alpha\overline{\beta}}\bigg\{\sigma_1\sigma_2\int_{X_s}(\Box+1)^{-1}(\xi_k(e_{i\overline{j}}))(\overline{\xi}_l(e_{\alpha\overline{\beta}})+\overline{\xi}_{\beta}(e_{\alpha\overline{l}}))dv\bigg\}\\
&\quad+h^{\alpha\overline{\beta}}\bigg\{\sigma_1\int_{X_s}Q_{k\overline{l}}(e_{i\overline{j}})e_{\alpha\overline{\beta}}dv\bigg\}+\tau_{p\overline{j}}h^{p\overline{q}}R_{i\overline{q}k\overline{l}}+CR_{i\overline{j}k\overline{l}}\\
&-\tilde{\tau}^{p\overline{q}}(h^{\alpha\overline{\beta}}\bigg\{\sigma_1\int_{X_s}\xi_k(e_{i\overline{q}})e_{\alpha\overline{\beta}}dv\bigg\})(h^{\gamma\overline{\delta}}\bigg\{\tilde{\sigma}_1\int_{X_s}\overline{\xi}_l(e_{p\overline{j}})e_{\gamma\overline{\delta}}dv\bigg\}),
\end{align*}
which is the desired formula.
\end{proof}

\subsection{Estimates on the Curvature of the Perturbed Ricci Metric}
The following linear algebra lemma will help us estimate the inverse matrix $(\tilde{\tau}^{i\overline{j}})$ near an interior point and a boundary point.
\begin{lem}
Let $D$ be a neighborhood of $0$ in $\mathbb{C}^n$ and let $A$ and $B$ be two positive definite $n\times n$ Hermitian matrix functions on $D$ such that they are bounded from above and below on $D$ and each entry of them is bounded on $D$.  Then each entry of the inverse matrix $(A+CB)^{-1}$ is of order $O(C^{-1})$ when $C$ is very large.
\end{lem}
\begin{proof}
Notice that $\det(A+CB)$ is a polynomial of $C$ of degree $n$ whose leading term is $\det(B)$, which is bounded from below.  Moreover, all the other coefficients are bounded as they are solely products of the entries of $A$ and $B$.  Thus we can choose $C$ large enough so that $\dfrac{1}{2}|\det(B)|C^n\geq|\det(A+CB)-\det(B)C^n|$, which implies that $|\det(A+CB)|\geq|\det(B)|C^n-|\det(A+CB)-\det(B)C^n|\geq\dfrac{1}{2}|\det(B)|C^n$.  We also know that the determinant of the $(i, j)$-minor of $A+CB$ is a polynomial of $C$ of degree at most $n-1$ whose coefficients are all bounded by the same reason as above.  Since the $(i, j)$-entry of $(A+CB)^{-1}$ is $(-1)^{i+j}$ multiplying the quotient of the determinant of the $(j, i)$-minor of $A+CB$ by $\det(A+CB)$, the result follows.
\end{proof}

The estimates on the inverse matric $(\tilde{\tau}^{i\overline{j}})$ near a boundary point is like this:
\begin{lem} (Lemma 5.2, \cite{bib1})
Let $X_0\in\mathcal{M}_g$ be a codimension $m$ boundary point and let $(t_1, ..., s_n)$ be the pinching coordinates near $X_0$.  Then for $|(t, s)|<\delta$ with $\delta$ small and for any $C>0$, we have
\begin{itemize}
\item[(i)] $0<\tilde{\tau}^{i\overline{i}}<\tau^{i\overline{i}}$ for all $i$;
\item[(ii)] $\tilde{\tau}^{i\overline{j}}=O(|t_it_j|)$ if $i, j\leq m$ and $i\neq j$;
\item[(iii)] $\tilde{\tau}^{i\overline{j}}=O(|t_i|)$ if $i\leq m$ and $j\geq m+1$;
\item[(iv)] $\tilde{\tau}^{i\overline{j}}=O(1)$ if $i, j\geq m+1$.
\end{itemize}
Furthermore, the bounds in $(ii)$ to $(iv)$ are independent of the choice of $C$.
\end{lem}
\begin{proof}
(i) is indeed a fact of linear algebra.  To prove (ii) to (iv), first let $A=(\tau_{i\overline{j}})_{i, j\geq m+1}$ and $B=(h_{i\overline{j}})_{i, j\geq m+1}$ again (notice that these two matrices in fact represent the non-degenerate directions of the Ricci metric and the Weil-Petersson metric).  Then using similar argument as in the proof of the second part of Corollary 4.4.2, we know that both $A$ and $B$ have positive lower bounds.  Moreover, by Corollary 4.4.1, $B$ is bounded from above and by Corollary 4.4.2, $A$ is bounded from above.  Hence both $A$ and $B$ are bounded from above and below and each of their entries is bounded on $\{(t, s):|(t, s)|\leq\delta\}$.  Now by Corollary 4.4.1 and Corollary 4.4.2 again we have
\[
(\tilde{\tau}_{i\overline{j}})=\left(
\begin{array}{cc}
\Upsilon & \Psi\\
\overline{\Psi}^T & A+CB
\end{array} \right)
\]
where $\Upsilon$ is an $m\times m$ matrix given by
\[
\Upsilon=\left(
\begin{array}{ccc}
\dfrac{u_1^2}{|t_1|^2}(\dfrac{3}{4\pi^2}+\dfrac{Cu_1}{2})(1+O(u_0)) & \cdots & \dfrac{u_1^2u_m^2}{|t_1t_m|}(O(u_0)+CO(u_1u_m))\\
\vdots & \ddots & \vdots\\
\dfrac{u_1^2u_m^2}{|t_1t_m|}(O(u_0)+CO(u_1u_m)) & \cdots & \dfrac{u_m^2}{|t_m|^2}(\dfrac{3}{4\pi^2}+\dfrac{Cu_m}{2})(1+O(u_0))
\end{array} \right)
\]
and $\Psi$ is an $m\times(n-m)$ matrix given by
\[
\Psi=\left(
\begin{array}{ccc}
\dfrac{u_1^2}{|t_1|}(O(1)+CO(u_1)) & \cdots & \dfrac{u_1^2}{|t_1|}(O(1)+CO(u_1))\\
\vdots & \ddots & \vdots\\
\dfrac{u_m^2}{|t_m|}(O(1)+CO(u_m)) & \cdots & \dfrac{u_m^2}{|t_m|}(O(1)+CO(u_m))
\end{array} \right).
\]
In fact, $\Upsilon$ represents the degenerate directions of the perturbed Ricci metric and $\Psi$ represents the mixed directions of the perturbed Ricci metric.  Now let $\tilde{\Phi}_{ij}$ be the $(i, j)$-minor of $(\tilde{\tau}_{i\overline{j}})$.  Then using similar computation techniques as in the proof of the second part of Corollary 4.4.2 we can show that
\[
\det(\tilde{\tau})=(\prod_{k=1}^m\dfrac{u_k^2}{|t_k|^2}(\dfrac{3}{4\pi^2}+\dfrac{Cu_k}{2}))\det(A+CB)(1+O(u_0))
\]
where the term $O(u_0)$ is independent of $C$, and
\[\det(\tilde{\Phi}_{ij})=(\prod_{k=1, k\neq i, j}^m\dfrac{u_k^2}{|t_k|^2}(\dfrac{3}{4\pi^2}+\dfrac{Cu_k}{2}))\det(A+CB)\dfrac{u_i^2u_j^2}{|t_it_j|}(O(1)+CO(u_i))\]
if $i, j\leq m$ and $i\neq j$;
\[\det(\tilde{\Phi}_{ij})=(\prod_{k=1, k\neq i}^m\dfrac{u_k^2}{|t_k|^2}(\dfrac{3}{4\pi^2}+\dfrac{Cu_k}{2}))\det(A+CB)O(\dfrac{u_i^2}{|t_i|})\]
if $i\leq m$ and $j\geq m+1$;
\[(\prod_{k=1}^m\dfrac{u_k^2}{|t_k|^2}(\dfrac{3}{4\pi^2}+\dfrac{Cu_k}{2}))\det(D_{ij})O(1)\]
if $i, j\geq m+1$, where $D_{ij}$ is the $(i, j)$-minor of $A+CB$.  Note that all the $O(\cdot)$ above are independent of $C$ as well.

It is worth to recall that $(\dfrac{3}{4\pi^2}+\dfrac{Cu_k}{2})^{-1}\leq\dfrac{4\pi^2}{3}=O(1)$ and by Lemma 4.5.1 we have $\bigg|\dfrac{\det(D_{ij})}{\det(A+CB)}\bigg|=O(C^{-1})=O(1)$, where these two bounds are also independent of $C$.  Hence by the fact that $|\tilde{\tau}^{i\overline{j}}|=\bigg|\dfrac{\det(\tilde{\Phi}_{ij})}{\det(\tilde{\tau})}\bigg|$, the result follows.
\end{proof}

Now we can estimate the curvature of the perturbed Ricci metric.
\begin{thm} (Theorem 5.2, \cite{bib1})
For a suitable choice of the positive constant $C$, the perturbed Ricci metric $\tilde{\tau}_{i\overline{j}}=\tau_{i\overline{j}}+Ch_{i\overline{j}}$ is complete and its holomorphic sectional curvature is negative and negatively bounded from above and below.  Furthermore, its bisectional curvature and Ricci curvature are bounded from above and below.
\end{thm}
\begin{proof}
As long as $C$ is non-negative it is obvious that $\tilde{\tau}$ is complete since it is greater than $\tau$ which is complete.

The proof of the boundedness of the holomorphic sectional curvature, the bisectional curvature and the Ricci curvature is in principle the same as the proof of Theorem 4.4.3 with the aid of Lemma 4.5.1 and Lemma 4.5.2.  For a detailed proof one can consult the proof of the Theorem 3.3 in \cite{bib2}.
\end{proof}

\subsection{Equivalence of the Perturbed Ricci Metric and the Ricci Metric}
The equivalence among the perturbed Ricci metric and the Ricci metric would be apparent after considering this lemma (Lemma 6.1, \cite{bib1}):
\begin{lem}
The Weil-Petersson metric is bounded from above by a constant multiple of the Ricci metric, i.e., there exists a constant $\alpha>0$ such that $\omega_{WP}\leq\alpha\tau$.
\end{lem}
\begin{proof}
This lemma follows directly form Corollary 4.4.1 and Corollary 4.4.2.
\end{proof}

Here we prove the equivalence:
\begin{thm} (Theorem 6.1, \cite{bib1})
The Ricci metric and the perturbed Ricci metric are equivalent.
\end{thm}
\begin{proof}
As $\tilde{\tau}_{i\overline{j}}=\tau_{i\overline{j}}+Ch_{i\overline{j}}$ and $C>0$, the Ricci metric is indeed bounded above by the perturbed Ricci metric.  Then by Lemma, we also have $\dfrac{1}{1+C\alpha}\tilde{\tau}=\dfrac{1}{1+C\alpha}(\tau+C\omega_{WP})\leq\tau$.  Hence we know that $\dfrac{1}{1+C\alpha}\tilde{\tau}\leq\tau\leq(1+C\alpha)\tilde{\tau}$.
\end{proof}

\chapter{Equivalence of the K\"{a}hler Metrics on the Teichm\"{u}ller Spaces and the Moduli Spaces of Riemann Surfaces}
In this final chapter we will go through two main results of Liu-Sun-Yau\cite{bib1}, namely, the equivalence of the Ricci metric and the K\"{a}hler-Einstein metric, and the equivalence of the Ricci metric and the McMullen metric.  These equivalences together imply that the Teichm\"{u}ller metric is equivalent to the K\"{a}hler-Einstein metric, which provides a confirmed answer to Yau's conjecture\cite{bib15}.

\section{Equivalence of the Ricci Metric and the K\"{a}hler-Einstein Metric}
Let's quote a fact of linear algebra which we will need in the proof of the equivalence of the Ricci metric $\tau$ and the K\"{a}hler-Einstein metric $g_{KE}$:
\begin{lem}
Let $A$ and $B$ be positive definite $n\times n$ Hermitian matrices and let $\alpha$ and $\beta$ be positive constants such that $B\geq \alpha A$ and $\det(B)\leq \beta\det(A)$.  Then there exists a positive constant $\gamma$ depending on $\alpha$, $\beta$ and $n$ such that $B\leq \gamma A$.
\end{lem}

\bigskip
Now we can start proving the first equivalence (Theorem 6.2, \cite{bib1}).
\begin{thm}
The Ricci metric and the K\"{a}hler-Einstein metric are equivalent.
\end{thm}
\begin{proof}
Consider the identity map $i_1: (\mathcal{M}_g, g_{KE}) \rightarrow (\mathcal{M}_g, \tilde{\tau})$.  From classical results we know that $g_{KE}$ is complete and its Ricci curvature is $-1$.  Moreover, by Theorem 4.5.2 we know that the holomorphic sectional curvature of $\tilde{\tau}$ is bounded from above by a negative constant.  Hence by Theorem 2.3.1, there is a positive constant $c_0$ such that $g_{KE}\geq c_0\tilde{\tau}$.  Thus we also have $g_{KE}\geq c_1\tau$ for some constant $c_1>0$ by Theorem 4.5.3.

Now we turn to consider the identity map $i_2: (\mathcal{M}_g, \tau) \rightarrow (\mathcal{M}_g, g_{KE})$.  By Theorem 4.4.3 we know that the scalar curvature and the Ricci curvature of $\tau$ is bounded from below.  In addition, the Ricci curvature of $g_{KE}$ is $-1$.  Then by Theorem 2.3.2 there exists a constant $c_2>0$ such that $\det(g_{KE})\leq c_2\det(\tau)$.  Hence using Lemma 5.1.1 there is a constant $c_3>0$ such that $g_{KE}\leq c_3\tau$, and this proves the theorem.
\end{proof}

\section{Equivalence of the Ricci Metric and the McMullen Metric}
For the comparison between the Ricci metric $\tau$ and the McMullen metric $g_{1/l}$, we are required to compute the first derivatives of the length functions of short geodesics (Lemma 6.3, \cite{bib1}).
\begin{lem}
Let $X_0\in\overline{\mathcal{M}_g}$ be a codimension $m$ boundary point and let $(t_1, ..., s_n)$ be the pinching coordinates around $X_0$.  Also, let $l_j$ be the length function of the short geodesic on the collar $\Omega_c^j$.  Then we have
\[
\partial_il_j=\left\{
\begin{array}{ll}
-\pi u_j\overline{b_i^j} & \text{if } i\neq j;\\
-\pi u_i\overline{b_i} & \text{if } i=j,
\end{array}\right.
\]
where $b_i^j$ and $b_i$ are defined in Lemma 4.4.1.
\end{lem}
\begin{proof}
By the harmonicity we know that on $\Omega_c^j$, $\lambda A_i$ is an anti-holomorphic quadratic differential, so we can write $\lambda A_i$ as $\kappa_i(\overline{z})d\overline{z}^2$ using the rs-coordinate $z$ on $\Omega_c^j$.  Let $C_{-2}(\kappa_i)$ be the coefficient of the term $\dfrac{1}{\overline{z}^2}$ in the expansion of $\kappa_i$.  By Lemma 4.4.1 we know that
\[
C_{-2}(\kappa_i)=\left\{
\begin{array}{ll}
\dfrac{1}{2}u_j^2\overline{b_i^j} & \text{if } i\neq j;\\
\dfrac{1}{2}u_i^2\overline{b_i} & \text{if } i=j.
\end{array}\right.
\]
Now fix $(t_0,s_0)$ with small norm and let $X=X_{t_0,s_0}$.  Also, let $w$ and $z$ be the rs-coordinates on the $j$-th collar of $X_{t,s}$ and $X$ respectively.  Notice that $w=w(z,t,s)$ is holomorphic with respect to $z$ and $w(z,t_0,s_0)=z$.  Consider the pulled-back metric on the $j$-th collar of $X_{t,s}$ to $X$.  We then obtain
\[\Lambda=\dfrac{1}{2}\dfrac{u_j^2}{|w|^2}\csc^2(u_j\log|w|)\bigg|\dfrac{\partial w}{\partial z}\bigg|^2\]
being the K\"{a}hler-Einstein metric on the $j$-th collar of $X_{t,s}$.  Note that $A_i=\partial_{\overline{z}} a_i=-\partial_{\overline{z}}(\Lambda^{-1}\partial_i\partial_{\overline{z}}\log\Lambda)$, so we can compute that at $(t_0,s_0)$,
\[\kappa_i(\overline{z})=-\dfrac{u_j\partial_i u_j}{\overline{z}^2}+\dfrac{u_j^2+1}{\overline{z}^3}\partial_i\overline{w}|_{(t_0,s_0)}-\dfrac{u_j^2+1}{\overline{z}^2}\partial_i\partial_{\overline{z}}\overline{w}|_{(t_0,s_0)}-\partial_i\partial_{\overline{z}}\partial_{\overline{z}}\partial_{\overline{z}}\overline{w}|_{(t_0,s_0)}.\]
Then since $\overline{w}$ is anti-holomorphic with respect to $z$, we know from the above expression that $C_{-2}(\kappa_i)=(-u_j\partial_i u_j)+c(u_j^2+1)-c(u_j^2+1)+0=-u_j\partial_i u_j$, where $c=\partial_i\partial_{\overline{z}}\overline{w}(0,t_0,s_0)$.  Finally we have
\[
\partial_i u_j=\left\{
\begin{array}{ll}
-\dfrac{1}{2}u_j\overline{b_i^j} & \text{if } i\neq j;\\
-\dfrac{1}{2}u_i\overline{b_i} & \text{if } i=j,
\end{array}\right.
\]
and thus the lemma follows by noting that $l_j=2\pi u_j$.
\end{proof}

We can prove the second equivalence (Theorem 6.3, \cite{bib1}) now.
\begin{thm}
The Ricci metric and the McMullen metric are equivalent.
\end{thm}
\begin{proof}
Since $g_{1/l}$ is complete and the Ricci curvature of $g_{1/l}$ is bounded form below, by Theorem 2.3.1 we know that
\[\tau<\tilde{\tau}\leq C_0g_{1/l}\]
for some positive constant $C_0$.

For the reverse inequality, we first fix a boundary point $X_0$ and the pinching coordinates near $X_0$.  By Lemma 4.1.1 and Theorem 4.1.3, we have
\begin{align*}
(g_{1/l})_{i\overline{i}}&=\left\|\dfrac{\partial}{\partial t_i}\right\|^2_{g_{1/l}}<C_1\left\|\dfrac{\partial}{\partial t_i}\right\|^2_T\\
&\leq C_2\bigg(\left\|\dfrac{\partial}{\partial t_i}\right\|^2_{WP}+\sum_{l_{\gamma}(X_0)<\epsilon}\bigg|(\partial\log l_{\gamma})\dfrac{\partial}{\partial t_i}\bigg|^2\bigg)=C_2\bigg(h_{i\overline{i}}+\sum^m_{j=1}|\partial_i\log l_j|^2\bigg).
\end{align*}
when $i\leq m$, for some constants $C_1$ and $C_2$.

Now we use Lemma 5.2.1 to obtain
\[
|\partial_i\log l_j|^2=\left\{
\begin{array}{ll}
\bigg|\dfrac{-\pi u_j\overline{b_i^j}}{l_j}\bigg|^2=\dfrac{1}{4}|b_i^j|^2 & \text{if } i\neq j;\\
\bigg|\dfrac{-\pi u_j\overline{b_i}}{l_j}\bigg|^2=\dfrac{1}{4}|b_i|^2 & \text{if } i=j.
\end{array}\right.
\]
Then by Lemma 4.4.1 we know that
\begin{align*}
\sum^m_{j=1}|\partial_i\log l_j|^2&=\dfrac{1}{4}|b_i|^2+\dfrac{1}{4}\sum^m_{j=1, j\neq i}|b_i^j|^2\\
&=\dfrac{1}{4}\dfrac{u_i^2}{\pi^2|t_i|^2}(1+O(u_0))+\dfrac{1}{4}\dfrac{u_i^2}{\pi^2|t_i|^2}O(u_0)=\dfrac{1}{4}\dfrac{u_i^2}{\pi^2|t_i|^2}(1+O(u_0)).
\end{align*}
From Corollary 4.4.1 and Corollary 4.4.2 we know that $h_{i\overline{i}}=\dfrac{1}{2}\dfrac{u_i^3}{|t_i|^2}(1+O(u_0))$ and $\tau_{i\overline{i}}=\dfrac{3}{4\pi^2}\dfrac{u_i^2}{|t_i|^2}(1+O(u_0))$ respectively, so we have
\[h_{i\overline{i}}+\sum^m_{j=1}|\partial_i\log l_j|^2\leq C_3\tau_{i\overline{i}}\]
for some constant $C_3$.  Hence we can find a constant $C_4=C_2C_3$ such that
\[(g_{1/l})_{i\overline{i}}\leq C_4\tau_{i\overline{i}}\]
for $i\leq m$.  For $i\geq m+1$, we just need to note that adopting similar argument as above we have
\[h_{i\overline{i}}+\sum^m_{j=1}|\partial_i\log l_j|^2=O(1)+O(u_0).\]
Then since $\tau_{i\overline{i}}=O(1)$ in this case, we have the above inequality for $i\geq m+1$ as well.  Now notice that $g_{1/l}$ is bounded from below by a constant multiple of $\tau$ and the diagonal entries of $((g_{1/l})_{i\overline{j}})$ is bounded from above by a constant multiple of the diagonal entries of $(\tau_{i\overline{j}})$ (and in fact the diagonal entries have dominating orders in respective matrix for both cases).  Therefore by a linear algebra fact similar to Lemma 5.1.1, we know that there is a constant $C_5$ such that
\[\tau\geq C_5g_{1/l}\]
at $X_0$.

Lastly by the same compactness argument as in the proof of Theorem 4.5.2 the theorem is proved.
\end{proof}
\begin{rem}
Since by Theorem 5.1.3, McMullen metric is equivalent to the Teichm\"{u}ller metric, and by Theorem 5.1.1 and Theorem 5.2.1 we know that McMullen metric is equivalent to the K\"{a}hler-Einstein metric, we can conclude that the K\"{a}hler-Einstein metric is equivalent to the Teichm\"{u}ller metric, which proves the Yau's conjecture.
\end{rem}

\newpage
\pagestyle{myheadings} \markright{\large{Bibliography} }
\addcontentsline{toc}{chapter}{Bibliography}
\bibliographystyle{ieee}
\bibliography{database}

\end{document}